\documentclass[10pt,a4paper]{article}
\usepackage[utf8]{inputenc}
\usepackage{amsmath}
\usepackage{amsfonts}
\usepackage{amssymb}
\usepackage{amsthm}
\usepackage{mathrsfs}
\usepackage{enumerate}
\usepackage{tikz-cd}
\usepackage{stmaryrd}
\usepackage{pgf}

\newcommand{\Q}{\mathbb{Q}}
\newcommand{\Z}{\mathbb{Z}}
\newcommand{\K}{\mathbb{K}}
\newcommand{\N}{\mathbb{N}}
\newcommand{\M}{\mathcal{M}}
\newcommand{\oh}{\mathcal{O}}

\newcommand{\I}{\mathcal{I}}

\newcommand{\Frac}{\mathrm{Frac}}
\newcommand{\PI}{{\rm PI}\text{-}{\rm deg}}
\newcommand{\End}{\mathrm{End}}
\newcommand{\orth}{\mathbf{1}^{\bot}}

\newcommand{\Span}{\mathrm{span}}

\newcommand{\oqmmn}{\oh_q(M_{m,n}(\K))}
\newcommand{\yl}{\oh_q(Y_{\lambda})}
\newcommand{\partitionk}{\oh_{q^{-1}}(Y_{\lambda})}
\newcommand{\oqgmn}{\oh_q(G_{m,n}(\K))}

\DeclareMathOperator{\Ima}{Im}
\DeclareMathOperator{\Id}{Id}

\newenvironment{rcases}
  {\left.\begin{aligned}}
  {\end{aligned}\right\rbrace}

\theoremstyle{plain}
\newtheorem{theorem}{Theorem}[section]
\newtheorem{lemma}[theorem]{Lemma}
\newtheorem{proposition}[theorem]{Proposition}
\newtheorem{corollary}[theorem]{Corollary}

\newtheorem{maintheorem}{Theorem}

\theoremstyle{definition}
\newtheorem{definition}[theorem]{Definition}
\newtheorem{remark}[theorem]{Remark}
\newtheorem{example}[theorem]{Example}

\theoremstyle{remark}

\makeatletter
\newenvironment{rsmallmatrix}{\null\,\vcenter\bgroup
  \Let@\restore@math@cr\default@tag
  \baselineskip6\ex@ \lineskip1.5\ex@ \lineskiplimit\lineskip
  \ialign\bgroup\hfil$\m@th\scriptstyle##$&&\thickspace\hfil
  $\m@th\scriptstyle##$\crcr
}{%
  \crcr\egroup\egroup\,%
}
\makeatother

\title{Quantum Grassmannians and their Associated Quantum Schubert Varieties at roots of unity}
\author{Jason P. Bell\footnote{The first-named author's research was partially supported by NSERC grant RGPIN-2022-02951.},~~ St\'{e}phane Launois\footnote{The second-named author's research was partly supported by EPSRC grant EP/R009279/1.},~~and Alexandra Rogers}

\begin{document}
\maketitle
\begin{abstract}

We study the PI degree of various quantum algebras at roots of unity, including quantum Grassmannians, quantum Schubert varieties, partition subalgebras, and their associated quantum affine spaces. By a theorem of De Concini  and Procesi, the PI degree of partition subalgebras and their associated quantum affine spaces is controlled by skew-symmetric integral matrices associated to (Cauchon-Le) diagrams. We prove that the invariant factors of these matrices are always powers of 2. This allows us to compute explicitly the PI degree of partition subalgebras. 

Our results also apply to certain completely prime (homogeneous) quotients of partition subalgebras. In particular, our results allow us to extend results of Jakobsen and J\o ndrup regarding the PI degree of quantum determinantal rings at roots of unity \cite{JakobsenJondrup} and we present a method to construct an irreducible representation of maximal dimension for quantum determinantal ideals.

Building on these results, we use the strong connection between partition subalgebras and quantum Schubert varieties through noncommutative dehomogenisation \cite{LenaganRigal-SchubertVar} to obtain expressions for the PI degree of quantum Schubert varieties. In particular, we compute the PI degree of quantum Grassmannians. 
\end{abstract}

\section{Introduction}
The PI degree of a PI algebra is an invariant from which various properties of the algebra can be deduced. For example, Brown and Yakimov \cite{BrownYakimov} showed that, under some mild conditions on a prime PI algebra, knowing its Azumaya locus (an invariant which is linked to the PI degree of the algebra and the PI degree of quotients by maximal ideals) provides valuable information about its discriminant ideal (another invariant with applications in the study of automorphism groups of PI algebras (\cite{CekenPalmieriWangZhang1} and \cite{CekenPalmieriWangZhang2})). The PI degree also plays an important role in the representation theory of prime affine PI algebras, giving an upper bound on the dimension of irreducible representations (\cite[Theorem I.13.5]{BrownGoodearl}). Our present aim is to provide explicit computation of the PI degree of various quantum algebras at roots of unity including quantum Grassmannians, quantum Schubert varieties, partition subalgebras, and associated quantum affine spaces.

Let $\K$ be an arbitrary field and $1\neq q\in \K^*$  be a field element. We will specify where we require $q$ to be a primitive $\ell^{\text{th}}$ root of unity, with $\ell > 2$. 
We start with the standard (single parameter) quantised coordinate ring of $m\times n$ matrices over $\K$, denoted $\oh_q(M_{m,n}(\K))$. This is a $\K$-algebra generated by $mn$ indeterminates $(X_{i,j})_{i,j=1}^{m,n}$ subject to some commutation relations which will be recapped in Section \ref{sectionQDRs}. For $\lambda=(\lambda_1 \geq \dots \geq \lambda_m)$ a partition with $n\geq \lambda_1$ and with associated Young diagram denoted by $Y_{\lambda}$, the partition subalgebra $\yl$ of $\oh_q(M_{m,n}(\K))$ is the subalgebra of $\oh_q(M_{m,n}(\K))$ generated by the generators $X_{i,j}$ fitting into $Y_{\lambda}$. In particular, $\oh_q(M_{m,n}(\K))$ is itself a partition subalgebra associated to partition $(n, \dots, n)=n^m$. These algebras have appeared naturally in the study of so-called quantum Grassmannians (see for instance \cite{LenaganRigal-SchubertVar,LLR,LLN}). 

By taking the $\K$-algebra generated by all maximal quantum minors of $\oh_q(M_{m,n}(\K))$ we obtain a quantum analogue of the homogeneous coordinate ring of the Grassmannian of the $m$-dimensional subspaces of $\K^n$. This is called the \textit{quantum Grassmannian} and is denoted $\oh_q(G_{m,n}(\K))$. The maximal quantum minors of $\oh_q(M_{m,n}(\K))$ are referred to as {\em quantum Pl\"ucker coordinates}, and the set of quantum Pl\"ucker coordinates of $\oh_q(G_{m,n}(\K))$ is denoted by $\Pi_{m,n}$. There is a natural partial order on $\Pi_{m,n}$ (see (\ref{partial order on Pi})) which allows to attach to each quantum Pl\"ucker coordinate $\gamma$ the so-called quantum Schubert variety $\oqgmn_{\gamma}$ associated to $\gamma$. This is the quotient of $\oh_q(G_{m,n}(\K))$ by the Schubert ideal associated to $\gamma$, that is, the quotient by the ideal generated by  $ \Pi_\gamma:=\{ \alpha \in \Pi_{m,n}| ~ \alpha \ngeq \gamma \}$. This algebra is a noncommutative deformation of the homogeneous coordinate ring of the Schubert variety associated to the (classical) Pl\"ucker coordinate $\gamma$. Note that $\oh_q(G_{m,n}(\K))$ is itself a quantum Schubert variety (associated to $\gamma =[1, \dots ,m]$).

All these algebras are related by noncommutative dehomogenisation \cite[Theorem 3.1.6]{LenaganRigal-SchubertVar}: the quantum Schubert variety associated to $\gamma$ is a noetherian domain and $\overline{\gamma}$ is a nonzero normal element of this algebra. This allows us to form the localised ring $\oqgmn_{\gamma}[\overline{\gamma}^{-1}]$. The dehomogenisation isomorphism of Lenagan and Rigal then shows that this localisation is a skew-Laurent extension of a partition algebra. More precisely, we have:
$$ \partitionk [y^{\pm 1};\sigma] \cong \oqgmn_{\gamma}[\overline{\gamma}^{-1}], $$
where $\sigma$ denotes the automorphism of $\partitionk$ defined by $\sigma(X_{i,j})=qX_{i,j}$ for all $(i,j)$ in the Young diagram of $\lambda$, where $\lambda$ is the partition naturally associated to $\gamma$ (see Subsection \ref{subsec-lambda-gamma} for more details). 

When $q$ is a root of unity, all these algebras are PI, and this dehomogenisation isomorphism provides a way to compute the PI degree of quantum Schubert varieties. First, we should compute the PI-degree of quantum partition subalgebras and then study the effect of a skew-Laurent extension on the PI degree. This is roughly the strategy we will be using in this paper. Before we explain our results, let us introduce more details.

Partition subalgebras can be presented as iterated Ore extensions over $\K$ with the indeterminates $X_{i,j}$ fitting into the Young diagram of $\lambda$ added in the lexicographic order: 
$$ \yl = \K[X_{1,1}]\cdots [X_{1,\lambda_1}; \sigma_{1,\lambda_1},\delta_{1,\lambda_1}]\dots [X_{m,\lambda_m};\sigma_{m,\lambda_m}, \delta_{m,\lambda_m}],$$
for suitable automorphisms $\sigma_{i,j}$ and $\sigma_{i,j}$-skew derivations  $\delta_{i,j}$. We note that each automorphism $\sigma_{i,j}$ acts by multiplication by powers of $q$ on the generators $X_{u,v}$ with $(u,v) <_{\mathrm{lex}}(i,j)$. The behaviour of these algebras depend very much on the choice of the parameter $q$. Indeed, it follows from \cite[Corollary 4.7]{Haynal} that $\yl$ is PI if and only if $q$ is a root of unity. 

For the rest of this introduction, we assume that $q$ is a primitive $\ell^{\text{th}}$ root of unity, with $\ell > 2$, so that $\yl$ is a PI algebra.  In this case, the PI degree of the iterated Ore extension above does not depend on the skew-derivations $\delta_{i,j}$ by  \cite[Corollary 4.7]{Haynal}. In other words, 
$$\PI( \yl) = \PI(\K[T_{1,1}]\cdots [T_{1,\lambda_1}; \sigma_{1,\lambda_1}]\dots [T_{m,\lambda_m};\sigma_{m,\lambda_m}]),$$
with slightly abusive notation. We denote the iterated Ore extension on the right-hand side by $\overline{\yl}$.  This algebra $\overline{\yl}$ is a (uniparameter) quantum affine space, that is, it is the $\K$-algebra generated by the $T_{i,j}$ subject to $T_{i,j} T_{u,v} = q^{a_{(i,j),(u,v)}}T_{u,v}T_{i,j}$ for all $(i,j), (u,v)$, where the $a_{(i,j),(u,v)}$ are integers coming from the automorphisms $\sigma_{i,j}$. We denote by $M_{\lambda}$ the integral skew-symmetric matrix associated to this quantum affine space, that is,  $M_{\lambda}=(a_{(i,j),(u,v)})\in M_{N}(\mathbb{Z})$, where the rows and columns of $M_{\lambda}$ are indexed by $(i,j)\in Y_{\lambda}$ ordered in the lexicographic order.

The above discussion shows that the PI degree of $\yl$ is equal to the PI degree of the (uniparameter) quantum affine space associated to the matrix  $M_{\lambda}$.  The PI degree of quantum affine spaces at roots of unity has been studied in \cite{DeConciniProcesi}, where it was shown that this PI degree is controlled by the invariant factors of the skew-symmetric integral matrix $M_{\lambda}$. A priori, computing invariant factors can prove a difficult task. In the present situation, matrix $M_{\lambda}$ can be obtained from a combinatorial gadget called {\em (Cauchon-Le) diagrams}. These objects have naturally appeared in the study of the primitive spectrum of the quantum Grassmannian on one hand \cite{LLR} and the cell decomposition of the totally nonnegative  Grassmannian on the other hand \cite{Postnikov}. Our first main result asserts that the invariant factors of any matrix associated to a diagram (in the sense of Section \ref{sectionCauchonDiag}) are always powers of $2$. Combining with results from \cite{BellCasteelsLaunois} where the dimension of the kernel of $M_{\lambda}$ was computed (in terms of the so-called toric permutation associated to a given diagram--see Section \ref{section-toric-permutation}), we obtain a closed formula for the PI degree of $\yl$ as follows:  

\begin{maintheorem}
Assume that $q$ is a primitive $\ell^{\mathrm{th}}$ root of unity with $\ell$ odd. Then 
$$\PI(\yl)= \ell^{\frac{N-r}{2}},$$
where $N=\sum_{i=1}^m \lambda_i$  is the number of boxes in $Y_{\lambda}$ and $r$ is the number of odd cycles in the disjoint cycle decomposition of the toric permutation associated to $\lambda$.
\end{maintheorem}

In the case of the partition $\lambda=n^n$, we retrieve that the PI degree of $\oh_q(M_{m,n}(\K))$ is equal to $\ell^{\frac{n(n-1)}{2}}$. This result was first obtained by Jakobsen and Zhang in \cite{JakobsenZhang} with the assumption that $\mathrm{char}(\K)=0$, and then reproved in \cite{Jondrup} and \cite{JakobsenJondrup}.  

We note that our results are not only applicable to partition subalgebras, but also to their completely prime quotients by {\em Cauchon ideals} in the sense of \cite[Definition 3.20]{LauLopRo}. We illustrate this in the specific case of quantum determinantal rings, that is, the quotient of $\oh_q(M_{n}(\K))$ by the ideal generated by all quantum minors of a given size. More precisely, our results allow us to compute the PI degree of quantum determinantal ideals, thus extending \cite[Proposition 6.2]{JakobsenJondrup} which was established with the assumption that the base field has characteristic $0$. We also construct an explicit irreducible representation of maximal dimension (equal to the PI degree) for all quantum determinantal rings. 

Now that we have a good grasp on the PI degree of partition subalgebras, we turn our attention quantum Schubert varieties. Quantum Schubert varieties are PI precisely when $q$ is a root of unity thanks to \cite{Haynal}, and so we continue to assume that $q$ is a primitive $\ell^{\text{th}}$ root of unity, with $\ell > 2$.  In view of the dehomogenisation isomorphism above and \cite[Corollary 4.7]{Haynal}, the computation of the PI degree of $\oqgmn_{\gamma}$ boils down to computing the dimension of the kernel and the invariant factors of the skew-symmetric matrix $M_{\gamma} =\left(\begin{array}{cc} M_{\lambda} &{\bf 1} \\{\bf -1}^T &0 \end{array}\right)$ with ${\bf 1} \in \mathbb{Z}^{\lambda_1+\dots+\lambda_m}$ being a column vector consisting of $1$ in each entry and with ${\bf -1} \in \mathbb{Z}^{\lambda_1+\dots+\lambda_m}$ being a row vector consisting of $-1$ in each entry. This naturally leads us to compare the invariant factors of an integral skew-symmetric matrix $M$ with the invariant factors of the matrix $M^E =\left(\begin{array}{cc} M &{\bf 1} \\{\bf -1}^T &0 \end{array}\right)$. We refer to the matrix $M^E$ as the {\em extended matrix} of $M$.   Interestingly, computing the invariant factors of $M_{\lambda}^E$ is trickier than computing those of $M_{\lambda}$. For instance, invariant factors of $M^E$ do not need to be powers of $2$ anymore. However, by a careful combinatorial study, we prove the following result: 

\begin{maintheorem}
\label{mainB}
 Let $q$ be a primitive $\ell^{\mathrm{th}}$ root of unity such that the smallest prime factor $\ell_0$ of $\ell$ satisfies $\ell_0>\min\{m,n,2\}$.
 
 Let $\gamma=[\gamma_1 < \dots < \gamma_m]$ be a quantum Pl\"ucker coordinate of $\oqgmn$ and let $\lambda $ be the partition associated to $\gamma$. We denote by $N$ the number of boxes in $Y_{\lambda}$ and by $r$ the number of odd cycles for the toric permutation associated to  $\lambda$.
 
  Then the PI degree of the quantum Schubert variety $\oqgmn_{\gamma}$ is given by 
\[ \PI \left(\oqgmn_{\gamma} \right) = \begin{cases}  \ell^{\frac{N-r}{2}} & \text{ if  $\ker(M_{\lambda}) \subseteq \mathbf{1}^{\bot}$}; \\
 \ell^{\frac{N-r}{2} + 1} & \text{ if $\ker(M_{\lambda}) \nsubseteq \mathbf{1}^{\bot}$}.
								\end{cases}
\]
\end{maintheorem}

It would be interesting to characterise when $\ker(M_{\lambda}) \subseteq \mathbf{1}^{\bot}$ via $\gamma$ or $\lambda$. We are able to do so in the case of the quantum Grassmannian. The result suggests that this might indeed be an interesting combinatorial question. 

If $i$ is a positive integer greater than or equal to $2$, we denote by $\mu_2(i)$ the 2-adic valuation of $i$; that is, the largest integer $j$ such that $2^j$ divides $i$. Then we have:

\begin{maintheorem}
\label{mainC}
Let $q$ be a primitive $\ell^{\mathrm{th}}$ root of unity such that the smallest prime factor $\ell_0$ of $\ell$ satisfies $\ell_0>\min\{m,n,2\}$. Then the PI degree of the quantum Schubert variety $\oqgmn$ is given by 
\[ \PI \left(\oqgmn \right) = \begin{cases}  \ell^{\frac{m(n-m)}{2}} & \text{ if  } \mu_2(m)\neq \mu_2(n); \\
 \ell^{\frac{m(n-m)-\gcd(m,n)}{2} + 1} & \text{ otherwise}.
								\end{cases}
\]
\end{maintheorem}

The paper is organised as follows. In Section \ref{section-algebras}, we introduce all the algebras involved and we use results of Haynal \cite{Haynal} to reduce the computation of their PI degree to the computation of the PI degree of a (uniparameter) quantum affine space. In Section \ref{sectionB_tItOreExt}, we explain De Concini-Procesi's strategy to compute the PI degree of a quantum affine space through the invariant factors of their associated integral skew-symmetric matrix \cite{DeConciniProcesi} and proceed to prove that these invariant factors are always powers of $2$ for integral skew-symmetric matrices coming from a diagram. This allows us to compute the PI degree of partition subalgebras (see Theorem \ref{thm-PIdegree-partition subalgebras}) and quantum determinantal rings (see Theorem \ref{PIdegQDR}). Section \ref{sectionIrreds} is devoted to the construction of an irreducible representation of maximal dimension for quantum determinantal rings. Finally, in Section \ref{sectionQSV}, we compare invariant factors of an integral skew-symmetric matrix and its associated extended matrix. This allows us to prove Theorems \ref{mainB} and \ref{mainC}.\\

\noindent {\bf Acknowledgment:} In \cite{LauLopRo}, the second-named author, third-named author and Samuel Lopes have generalised the computation of the PI degree and the construction of irreducible representations to completely prime quotients of Quantum Nilpotent Algebras at roots of unity. The authors of the present article would like to thank Samuel Lopes for the useful discussions that have led to Section \ref{sectionIrreds} and for allowing us to include some of the material of \cite{LauLopRo} here. 

We also thank H.~P. Jakobsen for pointing out that the PI degree of quantum determinantal rings had already been computed in \cite{JakobsenJondrup} when the characteristic of the base field is $0$.

The material presented in this article is part of the PhD thesis of the third-named author \cite{Alex-thesis}. She would like to thank the School of Mathematics, Statistics and Actuarial Science at the University of Kent for their financial support.

\section{Quantum Schubert Varieties}\label{sectionQDRs}\label{section-algebras}

\subsection{Quantum matrices and partition subalgebras} 

We denote the quantised coordinate ring of $m\times n$ matrices over $\K$ by $\oh_q(M_{m,n}(\K))$. This is the $\K$-algebra generated by $m\times n$ indeterminates $\{X_{i,j}\}_{i,j=1}^{m,n}$ subject to the following relations: for $(1,1) \leq  (i,j) < (s,t) \leq (m,n)$ (in lexicographic ordering), we have
\begin{align*}
X_{i,j}X_{s,t}= \begin{cases} 
							X_{s,t}X_{i,j} & i<s, \, j>t; \\
							qX_{s,t}X_{i,j} & (i=s, \, j<t) \text{ or } (i<s, \, j=t); \\
							X_{s,t}X_{i,j}+(q-q^{-1})X_{i,t}X_{s,j} & i<s, \, j<t.							
				 \end{cases}
\end{align*}
The \emph{quantum determinant} $D_q$ of $\oh_q(M_{n}(\K))$ can be expressed as
\[D_q =\sum_{\pi \in S_n} (-q)^{\ell(\pi)} X_{\pi(1),1} \cdots X_{\pi(n),n}, \]
where $\ell(\pi)$ gives the length of the permutation $\pi$.
For a pair of subsets $(I,J) \in \llbracket 1, m \rrbracket \times \llbracket 1, n \rrbracket$, with $|I|=|J|=s \leq \min(m,n)$, called an \emph{index pair}, the \emph{quantum minor} $[I|J] \in \oh_q(M_{m,n}(\K))$ is defined to be the quantum determinant of the subalgebra $\oh_q(M_s(\K))$ generated by $\{X_{i,j}\}_{i\in I, j\in J}$. We denote by $\Delta_{m,n}$ the set of index pairs.


We finish this section with the definition of partition subalgebras: 
\begin{definition}\label{definition-partition-subalgebra} 
Let $\lambda = (\lambda_1\geq \lambda_2\geq\dots\geq\lambda_m)$ be a partition with associated Young diagram $Y_\lambda$. 
In $\oqmmn$, with $n\geq \lambda_1$, look at the subring  
$\yl$ generated 
over $\K$ by those 
$X_{ij}$ that fit  into the Young diagram for $\lambda$. We call this subalgebra the {\em partition 
subalgebra of $\oqmmn$
 associated with the partition $\lambda$}.
\end{definition} 

For instance, for the partition $\lambda = (4,3,3,1)$, $\yl$ is generated by $X_{1,1}$,  $X_{1,2}$, $X_{1,3}$, $X_{1,4}$, $X_{2,1}$, $X_{2,2}$, $X_{2,3}$, $X_{3,1}$, $X_{3,2}$, $X_{3,3}$ and $X_{4,1}$.\\

It is well known that $\yl$ is an iterated Ore extension over $\K$, with the indeterminates $X_{i,j}$ fitting into the Young diagram of $\lambda$ added in the lexicographic order.

\subsection{Quantum Grassmannians and quantum Schubert varieties}

The \emph{quantum Grassmannian}, $\oqgmn$, $m<n$, is defined to be the subalgebra of $\oqmmn$ generated by all $m\times m$ maximal quantum minors. As for quantum matrices, a maximal quantum minor $\gamma:=[\{1,\dots,m\}| \{\gamma_1<\dots <\gamma_m\}]$ is uniquely defined by its \emph{index set} $\gamma:=\{\gamma_1<\gamma_2<\cdots<\gamma_m\} \subseteq \llbracket 1, n\rrbracket$. Such maximal quantum minors  $[ \gamma_1\cdots\gamma_m]$ are called the  \emph{quantum Pl\"ucker coordinates} of $\oqgmn$. Again, we identify the set of all index sets with the set of elements in $\oqgmn$ and label this $\Pi_{m,n}$. There is a natural partial order on $\Pi_{m,n}$ given by 
\begin{equation}\label{partial order on Pi}
[\gamma_1< \cdots < \gamma_m]\leq [\gamma_1'< \cdots < \gamma_m']\iff (\gamma_i\leq \gamma_i'\mbox{ for all }i\in \llbracket 1,m \rrbracket).
\end{equation}

 This partial ordering is used to define a quantum analogue of the homogeneous coordinate ring of Schubert varieties in the Grassmannian.

\begin{definition}[\cite{LenaganRigal-SchubertVar}]
Let $\gamma\in \Pi_{m,n}$ and set $\Pi_{m,n}^{\gamma}:=\{ \alpha \in \Pi_{m,n} \mid \alpha \ngeq \gamma \}$. The \emph{quantum Schubert variety} associated to $\gamma$ is:
\[ \oqgmn_{\gamma} := \oqgmn/ \langle \Pi_{m,n}^{\gamma} \rangle. \]
\end{definition}

It follows from \cite[Corollary 3.1.7]{LenaganRigal-SchubertVar} that $\oqgmn_{\gamma}$ is a noetherian domain. Moreover, this algebra is a quantum algebra with a straightening law  \cite[Example 2.1.3]{LenaganRigal-SchubertVar}, and it follows from \cite[Remark 2.1.4]{LenaganRigal-SchubertVar} and \cite[Lemma 1.2.1]{LenaganRigal-Straightening}  that the coset $\overline{\gamma}$ of $\gamma$ in $\oqgmn_{\gamma}$ is a normal element.

\subsection{Quantum Schubert varieties and partition subalgebras}
\label{subsec-lambda-gamma}

We now describe an isomorphism between a localisation of $ \oqgmn_{\gamma}$ and a skew-Laurent extension of a partition subalgebra of $\oh_{q^{-1}}(M_{m,n-m}(\K))$. This isomorphism mainly follows from \cite[Theorem 3.1.6]{LenaganRigal-SchubertVar}. It was established under the assumption that $q$ is not a root of unity in \cite{LLN}. However, the existence of this isomorphism, see \cite[Equation (55) in Section 9.5]{LLN}, carries over when we just assume that $q\neq 0$. 

Before we describe this isomorphism, we introduce the necessary ingredients. 

\begin{definition}{\rm 
The \emph{ladder} associated to $\gamma$ is denoted by $\mathcal{L}_\gamma$ and defined by 
\[
\mathcal{L}_\gamma=\{(i,j)\in \llbracket 1,m \rrbracket\times \llbracket 1,n \rrbracket\ |\ j>\gamma_{m+1-i}\mbox{ and }j\neq \gamma_l\ \mbox{ for all } l\in \llbracket 1,m \rrbracket \}.
\]
}\end{definition}

For $(i,j) \in \mathcal{L}_\gamma$, one defines $m_{i,j}:=[\{\gamma_1, \dots,\gamma_m\} \setminus \{\gamma_{m+1-i}\}\sqcup \{j\}]$ (which clearly belongs to $\Pi_{m,n} \setminus \Pi_{m,n}^{\gamma}$).

 Notice that  $\gamma_i-i=|\{a\in \llbracket 1,n \rrbracket\setminus\gamma\ |\ a<\gamma_i\}|$ 
for each $i\in\llbracket 1,m \rrbracket$. It follows easily that if we define $\lambda_i=n-m-(\gamma_i-i)$ for each $i\in \llbracket 1,m \rrbracket$, then $(\lambda_1,\ldots,\lambda_m)$ is a partition with $n-m\geq \lambda_1\geq \lambda_2\geq\cdots\geq \lambda_m\geq 0$.
Let $c$ be as large as possible such that $\lambda_c\neq 0$ and denote by $\lambda$ the partition $(\lambda_1,\ldots,\lambda_c)$.  
We say that $\lambda$ is the {\em partition associated to $\gamma$}.

Note that the south and east borders of $Y_{\lambda}$ give rise to path of length $n$, from the north-east corner to the south-west corner of the $m\times (n-m)$ rectangle. Label each of edges of this path with the numbers $1$ through $n$ (starting from the north-east corner). Then the elements of $\gamma$ coincide with the vertical steps in this numbering. The following example illustrates this construction when $\gamma =[1347]$.
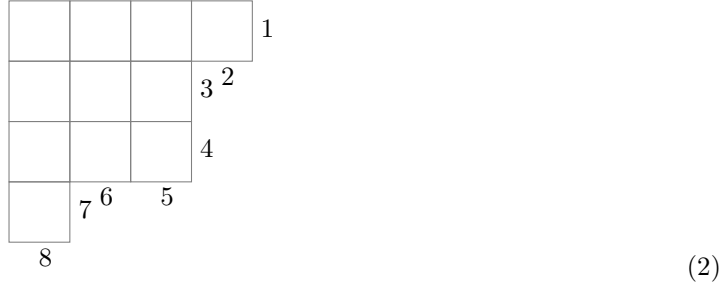
\begin{equation}\label{1347}
\begin{tikzpicture}[xscale=0.8, yscale=0.8]
\draw[color=gray] (0,0) rectangle (1,1);
\draw[color=gray] (0,1) rectangle (1,2);
\draw[color=gray] (1,1) rectangle (2,2);
\draw[color=gray] (2,1) rectangle (3,2);
\draw[color=gray] (0,2) rectangle (1,3);
\draw[color=gray] (1,2) rectangle (2,3);
\draw[color=gray] (2,2) rectangle (3,3);
\draw[color=gray] (0,3) rectangle (1,4);
\draw[color=gray] (1,3) rectangle (2,4);
\draw[color=gray] (2,3) rectangle (3,4);
\draw[color=gray] (3,3) rectangle (4,4);
\node at (4.25,3.55) {$1$};
\node at (3.6,2.75) {$2$};
\node at (3.25,2.55) {$3$};
\node at (3.25,1.55) {$4$};
\node at (2.6,0.75) {$5$};
\node at (1.6,0.75) {$6$};
\node at (1.25,0.55) {$7$};
\node at (0.6,-0.25) {$8$};
\end{tikzpicture}
\end{equation}

We are now ready to state the following result that makes a link between quantum Schubert varieties and partition subalgebras. 

\begin{theorem}
\label{the godfather iso}
There exists an isomorphism 
\begin{align}
\begin{split}
\Phi: \partitionk [y^{\pm 1};\sigma]&\xrightarrow{\cong} (\oqgmn/\langle \Pi_\gamma \rangle)[\overline{\gamma}^{-1}] \\
X_{i,j}&\mapsto \overline{m_{m+1-i,a_{n-m+1-j}}} \cdot \overline{\gamma}^{-1} \hspace{7mm} ((i,j)\in Y_\lambda) \\
y&\mapsto \overline{\gamma}, \\
\end{split}
\end{align}
where $\sigma$ denotes the automorphism of $\partitionk$ defined by $\sigma(X_{i,j})=qX_{i,j}$ for all $(i,j)$ in the Young diagram of $\lambda$, where we have set $\{a_1<\cdots<a_{n-m}\}:=\llbracket 1,n \rrbracket\setminus\gamma$ and where $\lambda$ is the partition associated to $\gamma$.
\end{theorem}
\begin{proof}
This is proved in a similar manner than \cite[Equation (55) in Section 9.5]{LLN} thanks to \cite[Theorem 3.1.6]{LenaganRigal-SchubertVar}.
\end{proof}

\subsection{Partition subalgebras and quantum Schubert varieties at roots of unity}

Recall that $\oqmmn$ can be written as an iterated Ore extension of the form:
$$\oqmmn =\K[X_{11}][X_{12};\sigma_{12},\delta_{12}]\cdots [X_{mn};\sigma_{mn},\delta_{mn}],$$
where the indeterminates $X_{i,j}$ are added in the lexicographic order, and where the $\sigma_{ij}$ are automorphisms and the $\delta_{ij}$ are left $\sigma_{ij}$-derivations of the appropriate subalgebras. 

With a slightly abusive notation, the partition subalgebra $\yl$ associated to the partition $\lambda = (\lambda_1\geq \lambda_2\geq\dots\geq\lambda_m)$ with $n\geq \lambda_1$ can also be written as an iterated Ore extention as follows: 
$$\yl =\K[X_{11}][X_{12};\sigma_{12},\delta_{12}]\cdots [X_{mn};\sigma_{mn},\delta_{mn}],$$
where the indeterminates $X_{i,j}$ with $(i,j)$ fitting in the Young diagram $Y_{\lambda}$ are added in the lexicographic order. Note that these automorphisms $\sigma_{i,j}$ and skew-derivations $\delta_{i,j}$ are in general distinct from the ones above. However, we kept the same notation to emphasize the fact that they are defined on the previous generators in the same way in both cases.

Assume $q$ is a primitive $\ell^{\mbox{th}}$ root of unity. Under this assumption, it follows from \cite[Theorem 1.2]{Haynal} that $\oqmmn$ is a PI algebra. The fact that $\oqmmn$ satisfies the required assumptions to apply \cite[Theorem 1.2]{Haynal} follows from \cite[Theorem 2.8]{Haynal}.  More generally, the same arguments can be used to prove that $\yl$ is a PI algebra.

It then easily follows from Theorem \ref{the godfather iso} that quantum Schubert varieties are also PI algebra. Note that quantum Schubert algebras are actually PI only under the assumption that $q$ is a root of unity. 

Moreover, it follows from Theorem \ref{the godfather iso} that:
$$\PI (\partitionk [y;\sigma])=\PI (\oqgmn_{\gamma}). $$

In this article, we are interested in computing the PI degree of $\yl$ and $\oqgmn_{\gamma}$. It follows from \cite[Theorem 1.2]{Haynal} that 
$$\PI(\yl) = \PI (\K[T_{11}][T_{12};\sigma_{12}]\cdots [T_{mn};\sigma_{mn}]),$$
where the indeterminates $T_{i,j}$ with $(i,j)$ fitting in the Young diagram $Y_{\lambda}$ are added in the lexicographic order. We denote this algebra $\K[T_{11}][T_{12};\sigma_{12}]\cdots [T_{mn};\sigma_{mn}]$ by $\overline{\yl}$. Note that once again we have used slightly abusive notation for the automorphisms $\sigma_{i,j}$. These are defined by 
$$\sigma_{i,j}(T_{s,t})= \begin{cases} 
							T_{s,t}& i<s, \, j>t; \\
							qT_{s,t} & (i=s, \, j<t) \text{ or } (i<s, \, j=t); \\
							T_{s,t} & i<s, \, j<t,	\end{cases}	$$
for all $(s,t) <_{\mbox{lex}} (i,j)$ such that $(s,t),(i,j) \in Y_{\lambda}$. 

We note that the algebra $ \overline{\yl}=\K[T_{11}][T_{12};\sigma_{12}]\cdots [T_{mn};\sigma_{mn}]$ is a quantum affine space. Recall that the (uniparameter) quantum affine space $\oh_{q^M}(\K^N)$ associated to the integral skew-symmetric matrix $M=(a_{i,j})\in M_{N}(\mathbb{Z})$ is the $\K$-algebra generated $T_1,\dots,T_N$ subject to $T_i T_j = q^{a_{i,j}}T_jT_i$ for all $1\leq i,j\leq N$. 

The above discussion shows that 
$$\PI(\yl) = \PI(\oh_{q^{M_{\lambda}}}[T_1, \dots , T_N]),$$
where $N = \lambda_1 + \dots + \lambda_m$ is the number of boxes in the Young diagram $Y_{\lambda}$ and $M_{\lambda} \in M_{N}(\mathbb{Z})$ is the skew-symmetric integral matrix defined as follows:  Label the boxes of $Y_{\lambda}$ from $1$ to $N$ in such a way that the labels increase along the rows (from left to right) and down the columns (from top to bottom). Given such a labelling, we construct $M_{\lambda}$ according the rule
\[ M_{\lambda}[i,j] = \begin{cases} 1 & \text{ if square } i \text{ is strictly below or strictly to the right of square } j; \\
							-1 & \text{ if square } i \text{ is strictly above or strictly to the left of square } j; \\
							0 & \text{ otherwise}. \end{cases} \]

We are now ready to summarise the main result of this section. 

\begin{proposition}
\label{prop-PIdegree-yl}
Assume $q$ is a primitive $\ell^{\mbox{th}}$ root of unity. 
\begin{enumerate}
\item Let $\lambda = (\lambda_1\geq \lambda_2\geq\dots\geq\lambda_m)$ be a partition. Then $\yl$ is a PI algebra and its PI degree is given by:
$$\PI (\yl)=\PI(\oh_{q^{M_{\lambda}}}(\K^N)),$$
where $M_{\lambda}$ is the matrix deduced from $M$ by deleting rows and columns indexed by $(i,j)$ not fitting in $Y_{\lambda}$.
\item Let $\gamma$ be an indexed set. Then the quantum Schubert variety $\oqgmn_{\lambda}$ is a PI algebra and 
$$\PI(\oqgmn_{\lambda}) =\PI(\oh_{q^{M_{\gamma}}}(\K^N)),$$
where $M_{\gamma} =\left(\begin{array}{cc} M_{\lambda} &{\bf 1} \\{\bf -1}^T &0 \end{array}\right)$ with ${\bf 1} \in \mathbb{Z}^{\lambda_1+\dots+\lambda_m}$ being a column vector consisting of $1$ in each entry and with ${\bf -1} \in \mathbb{Z}^{\lambda_1+\dots+\lambda_m}$ being a row vector consisting of $-1$ in each entry.
\end{enumerate}
\end{proposition}

In the next section, we develop new results to compute the PI degree of quantum affine spaces attached to matrices coming from (Cauchon-Le) diagrams. This will allow us to give an explicit formula for $\PI (\yl)$ thanks to Proposition \ref{prop-PIdegree-yl}.

\section{PI degree of quantum affine spaces associated to diagrams }\label{sectionB_tItOreExt}

\subsection{The PI Setting} \label{sectionB_tPI}
First, we present a general result relating the PI degree of an arbitrary quantum affine space $\oh_{q^M}(\K^N)$ to properties of its skew-symmetric commutation matrix $M \in M_N(\Z)$. It is a well-known result that $M$ is congruent (in the sense of \cite[Chapter IV]{Newman}) to its skew-normal form. We denote the skew normal form by $S$ or $\mathrm{Sk}(M)$ if it is not clear to which matrix $M$ the skew-normal form $S$ is congruent, and we write $M \sim_C S$. This is a block diagonal matrix of the form
\[ S=\left(\begin{smallmatrix}
 0 & h_1 & & & & & & & \\
 -h_1 & 0 & & & & & & & \\
 & & 0 & h_2 & & & & & \\
& & -h_2 & 0 & & & & & \\
 & & & & \ddots & & & \\
 & & & & & 0 & h_s & \\
 & & & & & -h_s & 0 & \\
 & & & & & & & \mathbf{0} 
\end{smallmatrix}\right), \]
where $\mathbf{0}$ is a square matrix of zeros of dimension $\dim(\ker(M))$ so that $2s=N-\dim(\ker(M))$, and $h_1,\, h_1,\, h_2,\, h_2,\, \ldots,\, h_s,\, h_s \in \Z\backslash \{0\}$ are called the \emph{invariant factors} of $M$. As they always come in pairs, from now on we will avoid repetition and list the invariant factors simply as $h_1,\, h_2,\,  \ldots,\, h_s$. These have the property that $h_i|h_{i+1}$ for all $i\in \llbracket 1, s \rrbracket$.

\begin{lemma}\label{lemPIdeg}
Take $q$ to be a primitive $\ell^{\mathrm{th}}$ root of unity. For some $N>1$, let $M\in M_N(\Z)$ be a skew-symmetric integral matrix with invariant factors $h_1, \ldots, h_s$. Then the PI degree of $\oh_{q^M}(\K^N)$ is given as
\[\PI(\oh_{q^M}(\K^N))=\prod_{i=1}^{\frac{N-\dim(\ker(M))}{2}} \frac{\ell}{\gcd(h_i, \ell)}.\]
\end{lemma}
\begin{proof}
By \cite[Proposition 7.1]{DeConciniProcesi}, the PI degree of this quantum affine space is $\sqrt{h}$, where $h$ is the cardinality of the image of the homomorphism
\[\begin{tikzcd}
 \Z^{N} \arrow{r}{M} & \Z^{N} \arrow{r}{\pi} & (\Z/\ell\Z)^N,
\end{tikzcd}\]
where $\pi$ denotes the canonical epimorphism. Let $S$ be the skew-normal form of $M$. Then $S$ is congruent to $M$ and, by \cite[Lemma 2.4]{PanovRussian}, the two quantum tori $\mathcal{O}_{q^{M}}((\K^*)^N)$ and $\mathcal{O}_{q^{S}}((\K^*)^N)$ are isomorphic. Also, since the PI degree of a prime, noetherian ring does not change upon localisation \cite[Corollary I.13.3]{BrownGoodearl}, the PI degree of a quantum affine space is equal to the PI degree of the associated quantum torus, and thus
\begin{align*}
 \PI(\oh_{q^M}(\K^N)) = \PI(\oh_{q^M}((\K^*)^N) = \PI(\oh_{q^S}((\K^*)^N) = \PI(\oh_{q^S}(\K^N)).
\end{align*}
It is therefore enough to compute the cardinality, $h$, of the image of the homomorphism
\begin{equation}\label{card}
\begin{tikzcd}
\Z^N \arrow{r}{S} & \Z^N \arrow{r}{\pi} & (\Z/\ell\Z)^N.
\end{tikzcd}
\end{equation}
Applying this map to a general element $\underline{z}=(z_1, \ldots, z_N)^T \in \Z^N$, we obtain the following:
\begin{equation*}
(\pi \circ S) (\underline{z})  =  (\overline{h_1 z_2},\overline{-h_1 z_1}, \overline{h_2 z_4}, \overline{-h_2 z_3}, \ldots ,\overline{h_s z_{2s}}, \overline{-h_s z_{2s-1}}, \mathbf{\overline{0}})^T,
\end{equation*}
where $\overline{h_i z_j}$ denotes the canonical image of $h_i z_j$ in $\Z/\ell \Z$. Since $\dim(\ker(M))=\dim(\ker(S))$, the dimension of the zero matrix $\mathbf{0}$ in $S$ is equal to $\dim(\ker(M))$ and hence $2s=N-\dim(\ker(M))$.

We now turn our attention to the entries of $(\pi \circ S)(\underline{z})$.  Each nonzero entry is of the form $\pm \overline{h_i} \overline{z_j}$ for some $\overline{z_j}\in \Z/\ell\Z$. Consider then, for each invariant factor $h_i$, the map
\begin{align*}
f_i: ~\Z &{} ~\longrightarrow~ \Z/\ell\Z \\
z_j &{} ~\longmapsto~ \overline{h_i z_j}.
\end{align*}
The image of $f_i$ is
\[ \Ima(f_i)=\{ \overline{h}_i \overline{z}_j \mid \overline{z}_j \in \Z/\ell\Z\} \subseteq \Z/\ell\Z \]
and this forms an additive subgroup of $(\Z/\ell\Z, +)$.  Furthermore, the image of $f_i$ is the cyclic subgroup generated by $\overline{h}_i$ and denoted by $\langle \overline{h}_i \rangle$. Since $(\Z/\ell\Z, +)$ is the cyclic group generated by $1$ containing $\ell$ elements then the order of $\overline{h}_i\in (\Z/\ell\Z, +)$ is $\ell/\gcd(h_i, \ell)$. Hence
\[ |\Ima(f_i)| = |\langle \overline{h}_i \rangle|= \frac{\ell}{\gcd(h_i, \ell)}. \]
The image of $(\pi \circ S)$ in $(\Z/\ell \Z)^N$ consists of copies of the subgroup $\Ima(f_i)\in \Z/\ell \Z$ in positions $2i$ and $2i-1$, for each $i\in \llbracket 1, s \rrbracket$, and zeros in positions $2s+1, \ldots, N$. Therefore, for each $i \in \llbracket 1, s \rrbracket$ the image of $f_i$ describes two entries in the image of $(\pi \circ S)$. The cardinality of the whole image of $(\pi \circ S)$ is therefore given as
\begin{equation*}
h = \prod_{i=1}^s |\Ima(f_i)|^2 = \prod_{i=1}^s | \langle \overline{h}_i \rangle |^2 = \prod_{i=1}^s \left( \frac{\ell}{\gcd(h_i, \ell)}\right)^2.
\end{equation*}
Substituting $2s=N-\dim(\ker(M))$, the PI degree of $\oh_{q^M}(\K^N)$ becomes
\begin{equation}\label{eqnrooth}
\sqrt{h} = \sqrt{\prod_{i=1}^{\frac{N-\dim(\ker(M))}{2}} \left( \frac{\ell}{\gcd(h_i, \ell)}\right)^2} = \prod_{i=1}^{\frac{N-\dim(\ker(M))}{2}} \frac{\ell}{\gcd(h_i, \ell)},
\end{equation}
as desired.
\end{proof}

Lemma \ref{lemPIdeg} tells us we need to find out the invariant factors and the kernel of the matrix $M$ in order to obtain the PI degree of the associated quantum affine space. For an arbitrary matrix $M$ this may need to be calculated specifically. However, we will see in the next section that for matrices associated to a (Cauchon-Le) diagram we can compute explicitly the PI degree of the corresponding quantum affine space. 

\subsection{Properties of Matrices associated to (Cauchon-Le) Diagrams and PI degree of their associated quantum affine spaces}\label{sectionCauchonDiag}
An \emph{$m\times n$ diagram} is an $m\times n$ grid with each square coloured either black or white. If the diagram satisfies the rule that if a square is black then all squares directly above it, or all squares directly to the left of it, are also black then we say that the diagram is a \emph{Cauchon-Le diagram}. This name recognises the fact that these diagrams were discovered independently by both Cauchon \cite{Cauchon-SpectrePremiers}, in relation to torus-invariant prime ideals of generic quantum matrices, and Postnikov \cite{Postnikov}, in relation to the totally nonnegative Grassmannian where they are called ``Le'' diagrams.

Given an $m\times n$ diagram $D$ we may compute its \emph{toric permutation} $\tau$, as defined in \cite[Section 4.1]{BellCasteelsLaunois}, by laying pipes over the squares such that we place a ``cross" on each black square and a ``hyperbola" on each white square. We label the sides of the diagram with the numbers $1,\ldots, m+n$ such that each pair of opposite sides share the same labels in the same order. The permutation, $\tau$, may then be read off this diagram by defining $\tau(i)$ to be the label (on the left or top side of $D$) reached by following the pipe starting at label $i$ (on the right or bottom side of $D$). See Figure \ref{FigToricPerm} for an example of a diagram with $\tau=(17)(26384)$.

\definecolor{light-gray}{gray}{0.6}
\begin{figure}[h]
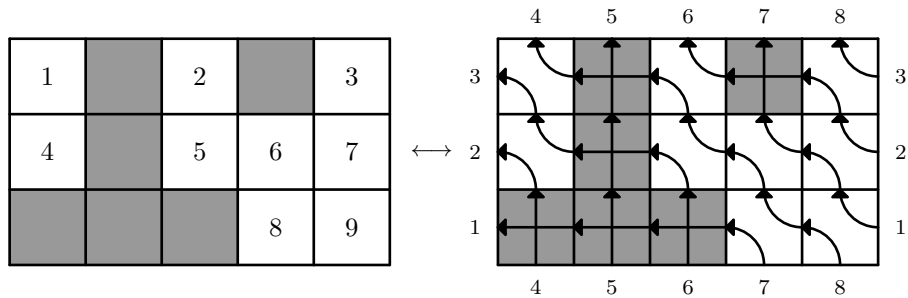

\begin{center}
\begin{tabular}{cc}
\begin{pgfpicture}{0cm}{0cm}{6cm}{4cm}%
\pgfsetroundjoin \pgfsetroundcap%
\pgfsetfillcolor{light-gray}
\pgfmoveto{\pgfxy(0.5,0.5)}\pgflineto{\pgfxy(0.5,1.5)}\pgflineto{\pgfxy(3.5,1.5)}\pgflineto{\pgfxy(3.5,0.5)}\pgflineto{\pgfxy(0.5,0.5)}\pgffill
\pgfmoveto{\pgfxy(1.5,1.5)}\pgflineto{\pgfxy(1.5,3.5)}\pgflineto{\pgfxy(2.5,3.5)}\pgflineto{\pgfxy(2.5,1.5)}\pgflineto{\pgfxy(1.5,1.5)}\pgffill
\pgfmoveto{\pgfxy(3.5,2.5)}\pgflineto{\pgfxy(3.5,3.5)}\pgflineto{\pgfxy(4.5,3.5)}\pgflineto{\pgfxy(4.5,2.5)}\pgflineto{\pgfxy(3.5,2.5)}\pgffill
\pgfsetlinewidth{1pt}
\pgfxyline(0.5,0.5)(4.5,0.5)
\pgfxyline(0.5,1.5)(4.5,1.5)
\pgfxyline(0.5,2.5)(4.5,2.5)
\pgfxyline(0.5,3.5)(4.5,3.5)
\pgfxyline(0.5,0.5)(0.5,3.5)
\pgfxyline(1.5,0.5)(1.5,3.5)
\pgfxyline(2.5,0.5)(2.5,3.5)
\pgfxyline(3.5,0.5)(3.5,3.5)
\pgfxyline(4.5,0.5)(4.5,3.5)
\pgfxyline(5.5,0.5)(5.5,3.5)
\pgfxyline(4.5,0.5)(5.5,0.5)
\pgfxyline(4.5,1.5)(5.5,1.5)
\pgfxyline(4.5,2.5)(5.5,2.5)
\pgfxyline(4.5,3.5)(5.5,3.5)

\pgfputat{\pgfxy(1,3)}{\pgfnode{rectangle}{center}{\color{black} $1$}{}{\pgfusepath{}}}
\pgfputat{\pgfxy(3,3)}{\pgfnode{rectangle}{center}{\color{black} $2$}{}{\pgfusepath{}}}
\pgfputat{\pgfxy(5,3)}{\pgfnode{rectangle}{center}{\color{black} $3$}{}{\pgfusepath{}}}
\pgfputat{\pgfxy(1,2)}{\pgfnode{rectangle}{center}{\color{black} $4$}{}{\pgfusepath{}}}
\pgfputat{\pgfxy(3,2)}{\pgfnode{rectangle}{center}{\color{black} $5$}{}{\pgfusepath{}}}
\pgfputat{\pgfxy(4,2)}{\pgfnode{rectangle}{center}{\color{black} $6$}{}{\pgfusepath{}}}
\pgfputat{\pgfxy(5,2)}{\pgfnode{rectangle}{center}{\color{black} $7$}{}{\pgfusepath{}}}
\pgfputat{\pgfxy(4,1)}{\pgfnode{rectangle}{center}{\color{black} $8$}{}{\pgfusepath{}}}
\pgfputat{\pgfxy(5,1)}{\pgfnode{rectangle}{center}{\color{black} $9$}{}{\pgfusepath{}}}

\pgfputat{\pgfxy(6,2)}{\pgfnode{rectangle}{center}{\color{black}\footnotesize $  ~~~~\longleftrightarrow~~~$}{}{\pgfusepath{}}}
\end{pgfpicture}

 &

\begin{pgfpicture}{0cm}{0cm}{6cm}{4cm}%
\pgfsetroundjoin \pgfsetroundcap%
\pgfsetfillcolor{light-gray}
\pgfmoveto{\pgfxy(0.5,0.5)}\pgflineto{\pgfxy(0.5,1.5)}\pgflineto{\pgfxy(3.5,1.5)}\pgflineto{\pgfxy(3.5,0.5)}\pgflineto{\pgfxy(0.5,0.5)}\pgffill
\pgfmoveto{\pgfxy(1.5,1.5)}\pgflineto{\pgfxy(1.5,3.5)}\pgflineto{\pgfxy(2.5,3.5)}\pgflineto{\pgfxy(2.5,1.5)}\pgflineto{\pgfxy(1.5,1.5)}\pgffill
\pgfmoveto{\pgfxy(3.5,2.5)}\pgflineto{\pgfxy(3.5,3.5)}\pgflineto{\pgfxy(4.5,3.5)}\pgflineto{\pgfxy(4.5,2.5)}\pgflineto{\pgfxy(3.5,2.5)}\pgffill
\pgfsetlinewidth{1pt}
\pgfxyline(0.5,0.5)(4.5,0.5)
\pgfxyline(0.5,1.5)(4.5,1.5)
\pgfxyline(0.5,2.5)(4.5,2.5)
\pgfxyline(0.5,3.5)(4.5,3.5)
\pgfxyline(0.5,0.5)(0.5,3.5)
\pgfxyline(1.5,0.5)(1.5,3.5)
\pgfxyline(2.5,0.5)(2.5,3.5)
\pgfxyline(3.5,0.5)(3.5,3.5)
\pgfxyline(4.5,0.5)(4.5,3.5)
\pgfxyline(5.5,0.5)(5.5,3.5)
\pgfxyline(4.5,0.5)(5.5,0.5)
\pgfxyline(4.5,1.5)(5.5,1.5)
\pgfxyline(4.5,2.5)(5.5,2.5)
\pgfxyline(4.5,3.5)(5.5,3.5)

\pgfsetlinewidth{1pt}
\color{black}

\pgfmoveto{\pgfxy(0.5,3)}\pgfpatharc{90}{0}{0.5cm}
\pgfstroke
\pgfmoveto{\pgfxy(0.5,3)}\pgflineto{\pgfxy(0.6,3.1)}\pgflineto{\pgfxy(0.6,2.9)}\pgflineto{\pgfxy(0.5,3)}\pgfclosepath\pgffillstroke

\pgfmoveto{\pgfxy(1,3.5)}\pgfpatharc{180}{270}{0.5cm}
\pgfstroke
\pgfmoveto{\pgfxy(1,3.5)}\pgflineto{\pgfxy(0.9,3.4)}\pgflineto{\pgfxy(1.1,3.4)}\pgflineto{\pgfxy(1,3.5)}\pgfclosepath\pgffillstroke

\color{black}
\pgfxyline(2,2.5)(2,3.5)
\pgfmoveto{\pgfxy(1.9,3.4)}\pgflineto{\pgfxy(2,3.5)}\pgflineto{\pgfxy(2.1,3.4)}\pgflineto{\pgfxy(1.9,3.4)}\pgfclosepath\pgffillstroke

\pgfxyline(1.5,3)(2.5,3)
\pgfmoveto{\pgfxy(1.5,3)}\pgflineto{\pgfxy(1.6,3.1)}\pgflineto{\pgfxy(1.6,2.9)}\pgflineto{\pgfxy(1.5,3)}\pgfclosepath\pgffillstroke

\pgfmoveto{\pgfxy(2.5,3)}\pgfpatharc{90}{0}{0.5cm}
\pgfstroke
\pgfmoveto{\pgfxy(2.5,3)}\pgflineto{\pgfxy(2.6,3.1)}\pgflineto{\pgfxy(2.6,2.9)}\pgflineto{\pgfxy(2.5,3)}\pgfclosepath\pgffillstroke

\pgfmoveto{\pgfxy(3,3.5)}\pgfpatharc{180}{270}{0.5cm}
\pgfstroke
\pgfmoveto{\pgfxy(3,3.5)}\pgflineto{\pgfxy(2.9,3.4)}\pgflineto{\pgfxy(3.1,3.4)}\pgflineto{\pgfxy(3,3.5)}\pgfclosepath\pgffillstroke


\pgfxyline(4,2.5)(4,3.5)
\pgfmoveto{\pgfxy(3.9,3.4)}\pgflineto{\pgfxy(4,3.5)}\pgflineto{\pgfxy(4.1,3.4)}\pgflineto{\pgfxy(3.9,3.4)}\pgfclosepath\pgffillstroke

\pgfxyline(3.5,3)(4.5,3)
\pgfmoveto{\pgfxy(3.5,3)}\pgflineto{\pgfxy(3.6,3.1)}\pgflineto{\pgfxy(3.6,2.9)}\pgflineto{\pgfxy(3.5,3)}\pgfclosepath\pgffillstroke


\pgfmoveto{\pgfxy(0.5,2)}\pgfpatharc{90}{0}{0.5cm}
\pgfstroke
\pgfmoveto{\pgfxy(0.5,2)}\pgflineto{\pgfxy(0.6,2.1)}\pgflineto{\pgfxy(0.6,1.9)}\pgflineto{\pgfxy(0.5,2)}\pgfclosepath\pgffillstroke

\pgfmoveto{\pgfxy(1,2.5)}\pgfpatharc{180}{270}{0.5cm}
\pgfstroke
\pgfmoveto{\pgfxy(1,2.5)}\pgflineto{\pgfxy(0.9,2.4)}\pgflineto{\pgfxy(1.1,2.4)}\pgflineto{\pgfxy(1,2.5)}\pgfclosepath\pgffillstroke


\pgfxyline(2,1.5)(2,2.5)
\pgfmoveto{\pgfxy(1.9,2.4)}\pgflineto{\pgfxy(2,2.5)}\pgflineto{\pgfxy(2.1,2.4)}\pgflineto{\pgfxy(1.9,2.4)}\pgfclosepath\pgffillstroke

\pgfxyline(1.5,2)(2.5,2)
\pgfmoveto{\pgfxy(1.5,2)}\pgflineto{\pgfxy(1.6,2.1)}\pgflineto{\pgfxy(1.6,1.9)}\pgflineto{\pgfxy(1.5,2)}\pgfclosepath\pgffillstroke

\pgfmoveto{\pgfxy(2.5,2)}\pgfpatharc{90}{0}{0.5cm}
\pgfstroke
\pgfmoveto{\pgfxy(2.5,2)}\pgflineto{\pgfxy(2.6,2.1)}\pgflineto{\pgfxy(2.6,1.9)}\pgflineto{\pgfxy(2.5,2)}\pgfclosepath\pgffillstroke

\pgfmoveto{\pgfxy(3,2.5)}\pgfpatharc{180}{270}{0.5cm}
\pgfstroke
\pgfmoveto{\pgfxy(3,2.5)}\pgflineto{\pgfxy(2.9,2.4)}\pgflineto{\pgfxy(3.1,2.4)}\pgflineto{\pgfxy(3,2.5)}\pgfclosepath\pgffillstroke


\pgfmoveto{\pgfxy(3.5,2)}\pgfpatharc{90}{0}{0.5cm}
\pgfstroke
\pgfmoveto{\pgfxy(3.5,2)}\pgflineto{\pgfxy(3.6,2.1)}\pgflineto{\pgfxy(3.6,1.9)}\pgflineto{\pgfxy(3.5,2)}\pgfclosepath\pgffillstroke

\pgfmoveto{\pgfxy(4,2.5)}\pgfpatharc{180}{270}{0.5cm}
\pgfstroke
\pgfmoveto{\pgfxy(4,2.5)}\pgflineto{\pgfxy(3.9,2.4)}\pgflineto{\pgfxy(4.1,2.4)}\pgflineto{\pgfxy(4,2.5)}\pgfclosepath\pgffillstroke


\pgfxyline(1,0.5)(1,1.5)
\pgfmoveto{\pgfxy(0.9,1.4)}\pgflineto{\pgfxy(1,1.5)}\pgflineto{\pgfxy(1.1,1.4)}\pgflineto{\pgfxy(0.9,1.4)}\pgfclosepath\pgffillstroke

\pgfxyline(0.5,1)(1.5,1)
\pgfmoveto{\pgfxy(0.5,1)}\pgflineto{\pgfxy(0.6,1.1)}\pgflineto{\pgfxy(0.6,0.9)}\pgflineto{\pgfxy(0.5,1)}\pgfclosepath\pgffillstroke


\pgfxyline(2,0.5)(2,1.5)
\pgfmoveto{\pgfxy(1.9,1.4)}\pgflineto{\pgfxy(2,1.5)}\pgflineto{\pgfxy(2.1,1.4)}\pgflineto{\pgfxy(1.9,1.4)}\pgfclosepath\pgffillstroke

\pgfxyline(1.5,1)(2.5,1)
\pgfmoveto{\pgfxy(1.5,1)}\pgflineto{\pgfxy(1.6,1.1)}\pgflineto{\pgfxy(1.6,0.9)}\pgflineto{\pgfxy(1.5,1)}\pgfclosepath\pgffillstroke


\pgfxyline(3,0.5)(3,1.5)
\pgfmoveto{\pgfxy(2.9,1.4)}\pgflineto{\pgfxy(3,1.5)}\pgflineto{\pgfxy(3.1,1.4)}\pgflineto{\pgfxy(2.9,1.4)}\pgfclosepath\pgffillstroke

\pgfxyline(2.5,1)(3.5,1)
\pgfmoveto{\pgfxy(2.5,1)}\pgflineto{\pgfxy(2.6,1.1)}\pgflineto{\pgfxy(2.6,0.9)}\pgflineto{\pgfxy(2.5,1)}\pgfclosepath\pgffillstroke

\pgfmoveto{\pgfxy(3.5,1)}\pgfpatharc{90}{0}{0.5cm}
\pgfstroke
\pgfmoveto{\pgfxy(3.5,1)}\pgflineto{\pgfxy(3.6,1.1)}\pgflineto{\pgfxy(3.6,0.9)}\pgflineto{\pgfxy(3.5,1)}\pgfclosepath\pgffillstroke

\pgfmoveto{\pgfxy(4,1.5)}\pgfpatharc{180}{270}{0.5cm}
\pgfstroke
\pgfmoveto{\pgfxy(4,1.5)}\pgflineto{\pgfxy(3.9,1.4)}\pgflineto{\pgfxy(4.1,1.4)}\pgflineto{\pgfxy(4,1.5)}\pgfclosepath\pgffillstroke

\pgfmoveto{\pgfxy(4.5,3)}\pgfpatharc{90}{0}{0.5cm}
\pgfstroke
\pgfmoveto{\pgfxy(4.5,3)}\pgflineto{\pgfxy(4.6,3.1)}\pgflineto{\pgfxy(4.6,2.9)}\pgflineto{\pgfxy(4.5,3)}\pgfclosepath\pgffillstroke

\pgfmoveto{\pgfxy(5,3.5)}\pgfpatharc{180}{270}{0.5cm}
\pgfstroke
\pgfmoveto{\pgfxy(5,3.5)}\pgflineto{\pgfxy(4.9,3.4)}\pgflineto{\pgfxy(5.1,3.4)}\pgflineto{\pgfxy(5,3.5)}\pgfclosepath\pgffillstroke


\pgfmoveto{\pgfxy(4.5,2)}\pgfpatharc{90}{0}{0.5cm}
\pgfstroke
\pgfmoveto{\pgfxy(4.5,2)}\pgflineto{\pgfxy(4.6,2.1)}\pgflineto{\pgfxy(4.6,1.9)}\pgflineto{\pgfxy(4.5,2)}\pgfclosepath\pgffillstroke

\pgfmoveto{\pgfxy(5,2.5)}\pgfpatharc{180}{270}{0.5cm}
\pgfstroke
\pgfmoveto{\pgfxy(5,2.5)}\pgflineto{\pgfxy(4.9,2.4)}\pgflineto{\pgfxy(5.1,2.4)}\pgflineto{\pgfxy(5,2.5)}\pgfclosepath\pgffillstroke

\pgfmoveto{\pgfxy(4.5,1)}\pgfpatharc{90}{0}{0.5cm}
\pgfstroke
\pgfmoveto{\pgfxy(4.5,1)}\pgflineto{\pgfxy(4.6,1.1)}\pgflineto{\pgfxy(4.6,0.9)}\pgflineto{\pgfxy(4.5,1)}\pgfclosepath\pgffillstroke

\pgfmoveto{\pgfxy(5,1.5)}\pgfpatharc{180}{270}{0.5cm}
\pgfstroke
\pgfmoveto{\pgfxy(5,1.5)}\pgflineto{\pgfxy(4.9,1.4)}\pgflineto{\pgfxy(5.1,1.4)}\pgflineto{\pgfxy(5,1.5)}\pgfclosepath\pgffillstroke

\pgfputat{\pgfxy(0.2,1)}{\pgfnode{rectangle}{center}{\color{black}\footnotesize $1$}{}{\pgfusepath{}}}
\pgfputat{\pgfxy(0.2,2)}{\pgfnode{rectangle}{center}{\color{black}\footnotesize $2$}{}{\pgfusepath{}}}
\pgfputat{\pgfxy(0.2,3)}{\pgfnode{rectangle}{center}{\color{black}\footnotesize $3$}{}{\pgfusepath{}}}

\pgfputat{\pgfxy(5.8,1)}{\pgfnode{rectangle}{center}{\color{black}\footnotesize $1$}{}{\pgfusepath{}}}
\pgfputat{\pgfxy(5.8,2)}{\pgfnode{rectangle}{center}{\color{black}\footnotesize $2$}{}{\pgfusepath{}}}
\pgfputat{\pgfxy(5.8,3)}{\pgfnode{rectangle}{center}{\color{black}\footnotesize $3$}{}{\pgfusepath{}}}

\pgfputat{\pgfxy(1,0.2)}{\pgfnode{rectangle}{center}{\color{black}\footnotesize $4$}{}{\pgfusepath{}}}
\pgfputat{\pgfxy(2,0.2)}{\pgfnode{rectangle}{center}{\color{black}\footnotesize $5$}{}{\pgfusepath{}}}
\pgfputat{\pgfxy(3,0.2)}{\pgfnode{rectangle}{center}{\color{black}\footnotesize $6$}{}{\pgfusepath{}}}
\pgfputat{\pgfxy(4,0.2)}{\pgfnode{rectangle}{center}{\color{black}\footnotesize $7$}{}{\pgfusepath{}}}
\pgfputat{\pgfxy(5,0.2)}{\pgfnode{rectangle}{center}{\color{black}\footnotesize $8$}{}{\pgfusepath{}}}

\pgfputat{\pgfxy(1,3.8)}{\pgfnode{rectangle}{center}{\color{black}\footnotesize $4$}{}{\pgfusepath{}}}
\pgfputat{\pgfxy(2,3.8)}{\pgfnode{rectangle}{center}{\color{black}\footnotesize $5$}{}{\pgfusepath{}}}
\pgfputat{\pgfxy(3,3.8)}{\pgfnode{rectangle}{center}{\color{black}\footnotesize $6$}{}{\pgfusepath{}}}
\pgfputat{\pgfxy(4,3.8)}{\pgfnode{rectangle}{center}{\color{black}\footnotesize $7$}{}{\pgfusepath{}}}
\pgfputat{\pgfxy(5,3.8)}{\pgfnode{rectangle}{center}{\color{black}\footnotesize $8$}{}{\pgfusepath{}}}
\end{pgfpicture}
\end{tabular}
\caption{A labelled $3\times 5$ (Cauchon-Le) diagram (left) with pipe dream construction (right).\label{FigToricPerm}}
\end{center}
\end{figure}

For each $m\times n$ diagram, $D$, with $N$ white squares, we may construct a skew-symmetric integral matrix, $M(D) \in M_{N}(\Z)$, as follows: Label the white squares from $1$ to $N$ in such a way that the labels increase along the rows (from left to right) and down the columns (from top to bottom). Given such a labelling, we construct $M(D)$ according the rule
\[ M(D)[i,j] = \begin{cases} 1 & \text{ if square } i \text{ is strictly below or strictly to the right of square } j; \\
							-1 & \text{ if square } i \text{ is strictly above or strictly to the left of square } j; \\
							0 & \text{ otherwise}. \end{cases} \]

The matrix associated to the diagram $D$ in Figure \ref{FigToricPerm}  is:
\[ M(D) = \left(\begin{rsmallmatrix}
0 & 1 & 1 & 1 & 0 & 0 & 0 & 0 & 0 \\
-1 & 0 & 1 & 0 & 1 & 0 & 0 & 0 & 0 \\
-1 & -1 & 0 & 0 & 0 & 0 & 1 & 0 & 1 \\
-1 & 0 & 0 & 0 & 1 & 1 & 1 & 0 & 0 \\
0 & -1 & 0 & -1 & 0 & 1 & 1 & 0 & 0 \\
0 & 0 & 0 & -1 & -1 & 0 & 1 & 1 & 0 \\
0 & 0 & -1 & -1 & -1 & -1 & 0 & 0 & 1\\
0 & 0 & 0 & 0 & 0 & -1 & 0 & 0 & 1 \\
0 & 0 & -1 & 0 & 0 & 0 & -1 & -1 & 0
\end{rsmallmatrix} \right) . \]

Our main aim in this section is to show that the invariant factors of a matrix $M(D)$ are always powers of $2$. The next lemma is a simplification of \cite[Proposition 4.6]{BellLaunois} where we consider diagrams without requiring them to have the Cauchon-Le property. This is needed for the theorem that follows. The proof follows that of the aforementioned proposition. We incorporate it for completeness.

\begin{lemma}\label{lemSubDiag}
Let $D$ be an $m \times n$ diagram. Suppose that the kernel of $M(D)$ has dimension $e \geq 1$. Then there is a diagram $D' \supseteq D$ obtained by adding exactly one black box to $D$ such that $M(D')$ has an $(e-1)$-dimensional kernel.
\end{lemma}
\begin{proof}
Without loss of generality we may assume that $D$ has no all-black columns. Make $D$ into a labelled diagram by assigning the white boxes labels $\llbracket 1, N \rrbracket$, for some $N\geq 1$. Clearly if any one of these boxes is coloured black, we are left with a diagram with $N-1$ white boxes.

For every $i\in \llbracket 1, N \rrbracket$, denote by $D_i$ the diagram obtained by colouring the white block with label $i$ black. Assume for contradiction, that $\dim(\ker(M(D_i))) \geq e$ for all $i$. Due to parity arguments, this means that $\dim(\ker(M(D_i))) \geq e+1$ and we can take $e+1$ linearly independent vectors of $\ker(M(D_i))$, which we denote $\{v_1^{(i)}, \ldots, v_{e+1}^{(i)}\}$.

We now construct $e+1$ linearly independent vectors of $\ker(M(D))$ as follows. For $1\leq j \leq e+1$ and $1\leq i \leq N$, let $w_j^{(i)}$ be the vector whose $\ell$-th coordinate is the $\ell$-th coordinate of $v_j^{(i)}$ if $\ell<i$, is $0$ if $\ell=i$, and is the $(\ell-1)$-th coordinate of $v_j^{(i)}$ if $\ell> i$. By construction, the set $\{w_j^{(i)}\}_{j=1}^{e+1}$ is orthogonal to every row of $M(D)$, except possibly row $i$.  Let $r$ denote row $i$ of $M(D)$ and define the following map:
\begin{eqnarray*}
\mathrm{Span}\{w_1^{(i)}, \ldots, w_{e+1}^{(i)}\} & \longrightarrow & \Q  \\
w & \longmapsto & r\cdot w
\end{eqnarray*}
This map must be surjective, since if not then it is the zero map, in which case the linearly independent vectors $\{w_j^{(i)}\}_{i=1}^{e+1}$ are contained within the kernel of $M(D)$. However, per the conditions of the lemma, the kernel of $M(D)$ has dimension $e$ so this is not possible. The kernel of this map is therefore an $e$-dimensional subspace of the span which lies in the kernel of $M(D)$, itself having dimension $e$. From this we see that every vector in $\ker(M(D))$ has $i$-th coordinate equal to zero. Furthermore, this is true for all $i \in \llbracket 1, N \rrbracket$, thus $\ker(M(D)) = \{\mathbf{0}\}$. This contradicts the dimension condition again, hence there must exist at least one $i \in \llbracket 1, N \rrbracket$ with $\dim(\ker(M(D_i))) \leq e-1$. 

To show that $\dim(\ker(M(D_i)))=e-1$ we simply observe that it is no loss of generality to assume that in the basis $\{v_1, \ldots, v_e\} \subseteq \ker(M(D))$, the vectors $v_1, \ldots, v_{e-1}$ have $i$-th coordinate equal to $0$. Therefore, deleting the $i$-th coordinate from each of these vectors to obtain $v'_1, \ldots, v'_{e-1}$ gives us a linearly independent set of vectors of $\ker(M(D_i))$. The result follows.
\end{proof}

We are now ready to state our main result regarding the invariant factors of a matrix $M(D)$. 

\begin{theorem}\label{leminvariant}
Let $D$ be an $m\times n$ diagram and $M(D)$ be its associated matrix. Then the invariant factors of $M(D)$ are all powers of $2$.
\end{theorem}
\begin{proof}
We prove that all invariant factors of $M(D)$ are powers of $2$ using increasing induction on the dimension, $d$, of the kernel of $M(D)$.

If $d=0$ then the matrix $M(D)$ is invertible and it follows, from \cite[Theorem 2.2]{BellLaunoisLutley}, that the determinant of $M(D)$ is a power of $4$. Let $\det(M(D))=4^a$ for some $a\in \N$. The matrix $M(D)$ is congruent to the block diagonal matrix
\begin{equation*}
\mathrm{Sk}(D):=\left(\begin{smallmatrix}
        0 & h_1 & & & & & \\
        -h_1 & 0 & & & & & \\
         & & 0 & h_2 & & & \\
         & & -h_2 & 0 & & & \\
         & & & & \ddots & & \\
         & & & & & 0 & h_s \\
         & & & & & -h_s & 0 \\
        \end{smallmatrix} \right),
\end{equation*}
where $\mathrm{rank}(M(D))=2s$ and $h_i \mid h_{i+1} \in \Z\backslash \{0\}$, for $i \in \llbracket 1, s-1 \rrbracket$. The matrix $\mathrm{Sk}(D)$ is clearly equivalent (in the sense of \cite[Chapter II]{Newman}) to the Smith normal form of $M(D)$, which we denote by
\[\mathrm{Sm}(D)= \mathrm{diag}(h_1,h_1,h_2,h_2,\ldots,h_s,h_s).\]
We see that $\det(\mathrm{Sk}(D))=\det(\mathrm{Sm}(D))$ since $\mathrm{Sm}(D)$ is obtained from $\mathrm{Sk}(D)$ by performing $s$ distinct row swaps and multiplying $s$ rows by $-1$, thus changing the determinant of $\mathrm{Sk}(D)$ by a factor of $(-1)^{2s}=1$. The congruence relation then implies that
\[\det(M(D))=\det(\mathrm{Sk}(D))=\det(\mathrm{Sm}(D)).\]
Since the nonzero entries $h_i$ in $\mathrm{Sm}(D)$ are the invariant factors of the matrix $M(D)$, we may use the observations above to deduce that
\begin{equation*}
4^a = \det(M(D))= \det(\mathrm{Sm}(D)) = \prod_{i=1}^s h_i^2 \quad \Longrightarrow \quad  2^a = \prod_{i=1}^s h_i.
\end{equation*}
Therefore, each $h_i$ is a power of $2$.

Now let $d>0$ and assume the statement of the theorem holds for all diagrams whose associated matrix has kernel of dimension less than $d$. That is, for all diagrams $D'$, with $\dim(\ker(M(D')))<d$, the invariant factors of $M(D')$ are powers of $2$.  Consider a diagram $D$ whose associated matrix $M(D)$ has kernel of dimension $d$, and let $N$ be the number of white squares in $D$ so that $M(D)$ is an $N\times N$ matrix.  Then, by Lemma \ref{lemSubDiag}, there exists a diagram $D'$ obtained by adding exactly one black square to $D$, with the property
\[\dim(\ker(M(D')))=d'=d-1.\]
Label the white squares of $D$ as $l_1 < \ldots < l_N$ and suppose we obtain $D'$ by colouring the white square labelled $l_{i_0}$ black, for some $i_0 \in \llbracket 1, N \rrbracket$.  Then $M(D')$ is the $(N-1)\times (N-1)$ submatrix of $M(D)$ obtained by removing the row and column indexed by $i_0$. Noting that $d'=d-1<d$, we apply the induction hypothesis to $M(D')$ to see that all its invariant factors, denoted by $h'_1,\ldots, h'_{s'}$ for some $s'\in \N$, are powers of $2$. Let $h_1,\ldots, h_s$ be the invariant factors of $M(D)$ and note that
\[s'=N-1-(d-1)=N-d=s.\]
We may therefore write the Smith normal forms of $M(D)$ and $M(D')$, respectively, as
\begin{align}
\mathrm{Sm}(D)&{}= \mathrm{diag}(h_1, h_1,\ldots, h_{s}, h_{s},0,\ldots, 0),\\
\mathrm{Sm}(D')&{}= \mathrm{diag}(h'_1, h'_1,\ldots, h'_{s}, h'_{s},0,\ldots, 0).
\end{align}

A well-known result for the Smith normal form (see \cite[Chapter II subsection 16]{Newman}) applied to $M(D')$ allows us to write
\begin{align}
D_{2i-1}(M(D'))&{}=h'_i \cdot D_{2i-2}(M(D')), \label{gcd1} \\
D_{2i}(M(D'))&{}=h'_i \cdot D_{2i-1}(M(D')), \label{gcd2}
\end{align}
where $D_j(M(D'))$ is the greatest common divisor of all $j\times j$ minors of the matrix $M(D')$, for all $j \in \llbracket 1, N-1 \rrbracket$, and $D_0(M(D')):=1$. Clearly any minor of $M(D')$ of size $j$ is also a minor of $M(D)$ of size $j$, so $D_j(M(D))$ divides $D_j(M(D'))$. Therefore, if $D_j(M(D'))$ is a power of 2, for all $j \in \llbracket 1, 2s \rrbracket$, then $D_j(M(D))$ is also a power of 2, for all $j \in \llbracket 1, 2s \rrbracket$. We now apply an induction argument on $i \in \llbracket 1, s \rrbracket$ to show that $D_{2i-1}(M(D'))$ and $D_{2i}(M(D'))$ are powers of 2 for all $i$. 

The base case, when $i=1$, is easily seen to hold since $h'_1$ is a power of 2, by the inductive hypothesis on $d$, and (\ref{gcd1}) and (\ref{gcd2}) give
\begin{align*}
D_1(M(D'))&{}=h'_1 \cdot D_0(M(D'))=h'_1\cdot 1, \\
D_2(M(D'))&{}=h'_1 \cdot D_1(M(D'))=h'_1 \cdot h'_1.
\end{align*}
Assume now that $D_{2i-1}(M(D'))$ and $D_{2i}(M(D'))$ are powers of 2, for some $i\in \llbracket 1, s-1 \rrbracket$, and consider $D_{2i+1}(M(D'))$ and $D_{2i+2}(M(D'))$. Using (\ref{gcd1}) and (\ref{gcd2}), we write these as
\begin{align*}
D_{2i+1}(M(D'))&{}=h'_{i+1} \cdot D_{2i}(M(D')), \\
D_{2i+2}(M(D'))&{}=h'_{i+1} \cdot D_{2i+1}(M(D')),
\end{align*}
where $h'_{i+1}$ is a power of $2$, by the induction hypothesis on $d$, and $D_{2i}(M(D'))$ is a power of $2$, by the induction hypothesis on $i \in \llbracket 1, s \rrbracket$. Hence, using the equations above, we may write
\begin{align*}
D_{2i+1}(M(D'))&{}=2^a \cdot 2^b=2^{a+b}, \\
D_{2i+2}(M(D'))&{}= 2^a \cdot 2^{a+b} = 2^{2a+b},
\end{align*}
for some $a,b \in \N$. This proves the inductive step for the induction on $i$ and, hence, we conclude that $D_j(M(D))$ is a power of $2$, for all $j \in \llbracket 1, 2s \rrbracket$.

Finally, to conclude the main induction on $d$ we apply the identity (\ref{gcd2}) to $M(D)$, and use the result of the previous induction argument, to show that, for any $i \in \llbracket 1, s \rrbracket$,
\[h_i=\frac{D_{2i}(M(D))}{D_{2i-1}(M(D))}=\frac{2^e}{2^f}=2^{e-f} \in \Z, \]
for some $e, f \in \N$ such that  $e\geq f$. This proves that all the invariant factors, $h_i$, of $M(D)$ are powers of 2, thus completing the main proof by induction.
\end{proof}

Before we turn our attention back to partition subalgebras, we observe that combining  \cite[Theorem 4.6 and Lemma 4.3]{BellCasteelsLaunois} gives the following result:
\begin{proposition}[\cite{BellCasteelsLaunois}]\label{propkerneldimension}
Let $D$ be an $m\times n$ diagram and $M(D)$ be its associated matrix. Then the dimension of the kernel of $M(D)$ is the number of odd cycles in the disjoint cycle decomposition of its toric permutation $\tau$.
\end{proposition}

From the two results above we see that if a matrix $M=M(D)$ is associated to an $m\times n$ diagram with $N$ white boxes, then we have methods to obtain the properties required to calculate the PI degree of $\oh_{q^M}(\K^N)$, when $q$ is a root of unity. More precisely, if $q$ is a primitive $\ell^{\mathrm{th}}$ root of unity, then 
\[\PI(\oh_{q^{M(D)}}(\K^{N}))= \begin{cases} \ell^{\frac{N-r}{2}} & \text{if $\ell$ is odd};  \\
\prod_{i=1}^{\frac{N-r}{2}} \frac{\ell}{\gcd(h_i, \ell)} & \text{if $\ell$ is even}, 
\end{cases}\]
where $h_1, \ldots, h_{(N-r)/2} \in \Z$ are the invariant factors of the matrix $M(D)$, $N$ is the number of white squares in $D$, $r$ is the number of odd cycles in the disjoint cycle decomposition of $\tau$.\\


In fact, our results work more generally as explained in the following remark. 

\begin{remark}
The results in this section dealing with diagrams can easily be stated for Young diagrams upon noting that every diagram on a Young tableau can be considered a diagram (in a rectangle) by embedding it in a suitably sized rectangle and colouring all squares outside the Young tableau black. More formally: Let $\lambda=(\lambda_1, \ldots, \lambda_m)$ be a partition with $n \geq \lambda_1 \geq \lambda_2, \cdots \geq \lambda_m \geq 0$ and let $Y_{\lambda}$ be the corresponding Young tableau. We can put this inside an $m \times n$ diagram, colouring all boxes outside of $Y_{\lambda}$ black (see Figure \ref{FigYoung}).  Furthermore, any Young diagram on $Y_{\lambda}$ obtained by colouring a set of white boxes black will, in turn, become another $m\times n$ diagram, $D'$, by colouring the corresponding white boxes in $D$ black. Combinatorially, there is no difference between the Young diagram and the corresponding diagram $D'$; their associated matrices will have the same invariant factors and dimension of kernel. 

Therefore, we can apply the results above to any Young diagram and see that all invariant factors of its associated matrix are powers of 2 and that the kernel dimension can be calculated by counting the odd cycles in its toric permutation. We note that (Cauchon-Le) diagrams on Young tableaux naturally arise in the study of the prime and primitive spectra of the quantum Grassmannian \cite{LLR}. \end{remark}

\definecolor{light-gray}{gray}{0.6}
\begin{figure}[h]
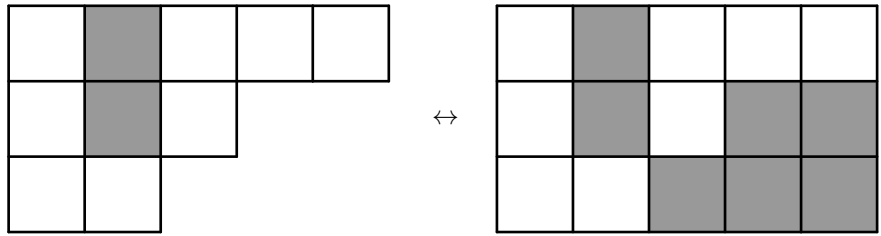

\begin{center}
\begin{tabular}{cc}
\begin{pgfpicture}{0cm}{0cm}{6cm}{4cm}%
\pgfputat{\pgfxy(0,2)}{\pgfnode{rectangle}{center}{\color{black} }{}{\pgfusepath{}}}
\pgfsetroundjoin \pgfsetroundcap%
\pgfsetfillcolor{light-gray}
\pgfmoveto{\pgfxy(1.5,1.5)}\pgflineto{\pgfxy(2.5,1.5)}\pgflineto{\pgfxy(2.5,3.5)}\pgflineto{\pgfxy(1.5,3.5)}\pgffill
\pgfsetlinewidth{1pt}
\pgfxyline(0.5,0.5)(2.5,0.5)
\pgfxyline(0.5,1.5)(3.5,1.5)
\pgfxyline(0.5,2.5)(4.5,2.5)
\pgfxyline(0.5,3.5)(4.5,3.5)
\pgfxyline(0.5,0.5)(0.5,3.5)
\pgfxyline(1.5,0.5)(1.5,3.5)
\pgfxyline(2.5,0.5)(2.5,3.5)
\pgfxyline(3.5,1.5)(3.5,3.5)
\pgfxyline(4.5,2.5)(4.5,3.5)
\pgfxyline(5.5,2.5)(5.5,3.5)
\pgfxyline(4.5,2.5)(5.5,2.5)
\pgfxyline(4.5,3.5)(5.5,3.5)

\end{pgfpicture}

 &

\begin{pgfpicture}{0cm}{0cm}{6cm}{4cm}%
\pgfputat{\pgfxy(0,2)}{\pgfnode{rectangle}{center}{\color{black} $ \leftrightarrow \quad $}{}{\pgfusepath{}}}
\pgfsetroundjoin \pgfsetroundcap%
\pgfsetroundjoin \pgfsetroundcap%
\pgfsetfillcolor{light-gray}
\pgfmoveto{\pgfxy(2.5,0.5)}\pgflineto{\pgfxy(2.5,1.5)}\pgflineto{\pgfxy(5.5,1.5)}\pgflineto{\pgfxy(5.5,0.5)}\pgflineto{\pgfxy(2.5,0.5)}\pgffill
\pgfmoveto{\pgfxy(3.5,1.5)}\pgflineto{\pgfxy(3.5,2.5)}\pgflineto{\pgfxy(5.5,2.5)}\pgflineto{\pgfxy(5.5,1.5)}\pgflineto{\pgfxy(3.5,1.5)}\pgffill
\pgfmoveto{\pgfxy(1.5,1.5)}\pgflineto{\pgfxy(2.5,1.5)}\pgflineto{\pgfxy(2.5,3.5)}\pgflineto{\pgfxy(1.5,3.5)}\pgffill
\pgfsetlinewidth{1pt}
\pgfxyline(0.5,0.5)(4.5,0.5)
\pgfxyline(0.5,1.5)(4.5,1.5)
\pgfxyline(0.5,2.5)(4.5,2.5)
\pgfxyline(0.5,3.5)(4.5,3.5)
\pgfxyline(0.5,0.5)(0.5,3.5)
\pgfxyline(1.5,0.5)(1.5,3.5)
\pgfxyline(2.5,0.5)(2.5,3.5)
\pgfxyline(3.5,0.5)(3.5,3.5)
\pgfxyline(4.5,0.5)(4.5,3.5)
\pgfxyline(5.5,0.5)(5.5,3.5)
\pgfxyline(4.5,0.5)(5.5,0.5)
\pgfxyline(4.5,1.5)(5.5,1.5)
\pgfxyline(4.5,2.5)(5.5,2.5)
\pgfxyline(4.5,3.5)(5.5,3.5)

\end{pgfpicture}
\end{tabular}
\caption{Diagram on a Young diagram (left) embedded into an $3\times 5$ diagram (right).}\label{FigYoung}
\end{center}
\end{figure}

We are now ready to state our main result for this section. 
\begin{theorem}
\label{thm-PIdegree-diag}
Let $Y_{\lambda}$ be a Young tableau, $D$ be a diagram on $Y_{\lambda}$ with $N$ white boxes, and $\tau$ be the toric permutation on $Y_{\lambda}$. If $q$ is a primitive $\ell^{\mathrm{th}}$ root of unity, then 
\[\PI(\oh_{q^{M(D)}}(\K^{N}))= \begin{cases} \ell^{\frac{N-r}{2}} & \text{if $\ell$ is odd};  \\
\prod_{i=1}^{\frac{N-r}{2}} \frac{\ell}{\gcd(h_i, \ell)} & \text{if $\ell$ is even}, 
\end{cases}\]
where $h_1, \ldots, h_{(N-r)/2} \in \Z$ are the invariant factors of the matrix $M(D)$ and $r$ is the number of odd cycles in the disjoint cycle decomposition of $\tau$.
\end{theorem}

%

\subsection{PI degree of partition subalgebras}\label{section-toric-permutation}

In view of Theorem \ref{thm-PIdegree-diag} and Proposition \ref{prop-PIdegree-yl}, we know that the PI degree of $\yl$ depends only on the toric permutation associated to $Y_{\lambda}$. More precisely, let $\lambda=\{ n=\lambda_1 \geq \dots \geq \lambda_m\}$ be a partition with associated Young diagram $Y_{\lambda}$. We can associate to $\lambda$ an $m \times n$ diagram $D_{\lambda}$ by colouring black every box $(i,j)$ with $j> \lambda_i$, see Figure \ref{FigYoung}. Then, we deduce from Proposition \ref{prop-PIdegree-yl} that
$$\PI (\yl )= \PI(\oh_{q^{M(D_{\lambda})}}(\K^{N})).$$
Moreover, Theorem \ref{thm-PIdegree-diag} shows that $\PI(\oh_{q^{M(D_{\lambda})}}(\K^{N}))$ essentially depends on the toric permutation associated to $D_{\lambda}$. \\

Our next aim is to identify this toric permutation $\tau_{\lambda}$. To achieve this, we need another permutation attached to a diagram, which we call the {\em $w$-permutation}. In the construction of toric permutations through pipe dreams, we made a choice for the labelling of the sides of a diagram. By making a different choice, we obtain a different permutation. More precisely,  the permutation $w$ attached to an $m\times n$ diagram $D$ is obtained by laying pipes over the squares such that we place a ``cross" on each black square and a ``hyperbola" on each white square. The difference with the toric permutation associated to $D$ is in the way we label the sides of $D$: to compute $w$, we label the sides of the diagram with the numbers $1,\ldots, m+n$ starting from the North East box. The permutation, $w$, may then be read off this diagram by defining $w(i)$ to be the label (on the left or top side of $D$) reached by following the pipe starting at label $i$ (on the right or bottom side of $D$). See Figure \ref{FigRestrictedPerm} for an example of a diagram with $w=(23)(4587)$.

\definecolor{light-gray}{gray}{0.6}
\begin{figure}[h]
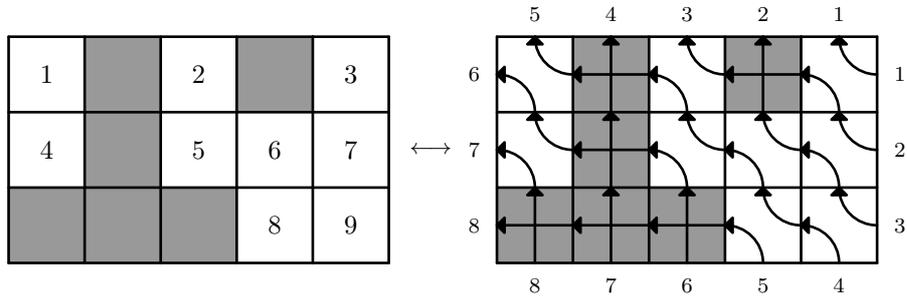

\begin{center}
\begin{tabular}{cc}
\begin{pgfpicture}{0cm}{0cm}{6cm}{4cm}%
\pgfsetroundjoin \pgfsetroundcap%
\pgfsetfillcolor{light-gray}
\pgfmoveto{\pgfxy(0.5,0.5)}\pgflineto{\pgfxy(0.5,1.5)}\pgflineto{\pgfxy(3.5,1.5)}\pgflineto{\pgfxy(3.5,0.5)}\pgflineto{\pgfxy(0.5,0.5)}\pgffill
\pgfmoveto{\pgfxy(1.5,1.5)}\pgflineto{\pgfxy(1.5,3.5)}\pgflineto{\pgfxy(2.5,3.5)}\pgflineto{\pgfxy(2.5,1.5)}\pgflineto{\pgfxy(1.5,1.5)}\pgffill
\pgfmoveto{\pgfxy(3.5,2.5)}\pgflineto{\pgfxy(3.5,3.5)}\pgflineto{\pgfxy(4.5,3.5)}\pgflineto{\pgfxy(4.5,2.5)}\pgflineto{\pgfxy(3.5,2.5)}\pgffill
\pgfsetlinewidth{1pt}
\pgfxyline(0.5,0.5)(4.5,0.5)
\pgfxyline(0.5,1.5)(4.5,1.5)
\pgfxyline(0.5,2.5)(4.5,2.5)
\pgfxyline(0.5,3.5)(4.5,3.5)
\pgfxyline(0.5,0.5)(0.5,3.5)
\pgfxyline(1.5,0.5)(1.5,3.5)
\pgfxyline(2.5,0.5)(2.5,3.5)
\pgfxyline(3.5,0.5)(3.5,3.5)
\pgfxyline(4.5,0.5)(4.5,3.5)
\pgfxyline(5.5,0.5)(5.5,3.5)
\pgfxyline(4.5,0.5)(5.5,0.5)
\pgfxyline(4.5,1.5)(5.5,1.5)
\pgfxyline(4.5,2.5)(5.5,2.5)
\pgfxyline(4.5,3.5)(5.5,3.5)

\pgfputat{\pgfxy(1,3)}{\pgfnode{rectangle}{center}{\color{black} $1$}{}{\pgfusepath{}}}
\pgfputat{\pgfxy(3,3)}{\pgfnode{rectangle}{center}{\color{black} $2$}{}{\pgfusepath{}}}
\pgfputat{\pgfxy(5,3)}{\pgfnode{rectangle}{center}{\color{black} $3$}{}{\pgfusepath{}}}
\pgfputat{\pgfxy(1,2)}{\pgfnode{rectangle}{center}{\color{black} $4$}{}{\pgfusepath{}}}
\pgfputat{\pgfxy(3,2)}{\pgfnode{rectangle}{center}{\color{black} $5$}{}{\pgfusepath{}}}
\pgfputat{\pgfxy(4,2)}{\pgfnode{rectangle}{center}{\color{black} $6$}{}{\pgfusepath{}}}
\pgfputat{\pgfxy(5,2)}{\pgfnode{rectangle}{center}{\color{black} $7$}{}{\pgfusepath{}}}
\pgfputat{\pgfxy(4,1)}{\pgfnode{rectangle}{center}{\color{black} $8$}{}{\pgfusepath{}}}
\pgfputat{\pgfxy(5,1)}{\pgfnode{rectangle}{center}{\color{black} $9$}{}{\pgfusepath{}}}

\pgfputat{\pgfxy(6,2)}{\pgfnode{rectangle}{center}{\color{black}\footnotesize $  ~~~~\longleftrightarrow~~~$}{}{\pgfusepath{}}}
\end{pgfpicture}

 &

\begin{pgfpicture}{0cm}{0cm}{6cm}{4cm}%
\pgfsetroundjoin \pgfsetroundcap%
\pgfsetfillcolor{light-gray}
\pgfmoveto{\pgfxy(0.5,0.5)}\pgflineto{\pgfxy(0.5,1.5)}\pgflineto{\pgfxy(3.5,1.5)}\pgflineto{\pgfxy(3.5,0.5)}\pgflineto{\pgfxy(0.5,0.5)}\pgffill
\pgfmoveto{\pgfxy(1.5,1.5)}\pgflineto{\pgfxy(1.5,3.5)}\pgflineto{\pgfxy(2.5,3.5)}\pgflineto{\pgfxy(2.5,1.5)}\pgflineto{\pgfxy(1.5,1.5)}\pgffill
\pgfmoveto{\pgfxy(3.5,2.5)}\pgflineto{\pgfxy(3.5,3.5)}\pgflineto{\pgfxy(4.5,3.5)}\pgflineto{\pgfxy(4.5,2.5)}\pgflineto{\pgfxy(3.5,2.5)}\pgffill
\pgfsetlinewidth{1pt}
\pgfxyline(0.5,0.5)(4.5,0.5)
\pgfxyline(0.5,1.5)(4.5,1.5)
\pgfxyline(0.5,2.5)(4.5,2.5)
\pgfxyline(0.5,3.5)(4.5,3.5)
\pgfxyline(0.5,0.5)(0.5,3.5)
\pgfxyline(1.5,0.5)(1.5,3.5)
\pgfxyline(2.5,0.5)(2.5,3.5)
\pgfxyline(3.5,0.5)(3.5,3.5)
\pgfxyline(4.5,0.5)(4.5,3.5)
\pgfxyline(5.5,0.5)(5.5,3.5)
\pgfxyline(4.5,0.5)(5.5,0.5)
\pgfxyline(4.5,1.5)(5.5,1.5)
\pgfxyline(4.5,2.5)(5.5,2.5)
\pgfxyline(4.5,3.5)(5.5,3.5)

\pgfsetlinewidth{1pt}
\color{black}

\pgfmoveto{\pgfxy(0.5,3)}\pgfpatharc{90}{0}{0.5cm}
\pgfstroke
\pgfmoveto{\pgfxy(0.5,3)}\pgflineto{\pgfxy(0.6,3.1)}\pgflineto{\pgfxy(0.6,2.9)}\pgflineto{\pgfxy(0.5,3)}\pgfclosepath\pgffillstroke

\pgfmoveto{\pgfxy(1,3.5)}\pgfpatharc{180}{270}{0.5cm}
\pgfstroke
\pgfmoveto{\pgfxy(1,3.5)}\pgflineto{\pgfxy(0.9,3.4)}\pgflineto{\pgfxy(1.1,3.4)}\pgflineto{\pgfxy(1,3.5)}\pgfclosepath\pgffillstroke

\color{black}
\pgfxyline(2,2.5)(2,3.5)
\pgfmoveto{\pgfxy(1.9,3.4)}\pgflineto{\pgfxy(2,3.5)}\pgflineto{\pgfxy(2.1,3.4)}\pgflineto{\pgfxy(1.9,3.4)}\pgfclosepath\pgffillstroke

\pgfxyline(1.5,3)(2.5,3)
\pgfmoveto{\pgfxy(1.5,3)}\pgflineto{\pgfxy(1.6,3.1)}\pgflineto{\pgfxy(1.6,2.9)}\pgflineto{\pgfxy(1.5,3)}\pgfclosepath\pgffillstroke

\pgfmoveto{\pgfxy(2.5,3)}\pgfpatharc{90}{0}{0.5cm}
\pgfstroke
\pgfmoveto{\pgfxy(2.5,3)}\pgflineto{\pgfxy(2.6,3.1)}\pgflineto{\pgfxy(2.6,2.9)}\pgflineto{\pgfxy(2.5,3)}\pgfclosepath\pgffillstroke

\pgfmoveto{\pgfxy(3,3.5)}\pgfpatharc{180}{270}{0.5cm}
\pgfstroke
\pgfmoveto{\pgfxy(3,3.5)}\pgflineto{\pgfxy(2.9,3.4)}\pgflineto{\pgfxy(3.1,3.4)}\pgflineto{\pgfxy(3,3.5)}\pgfclosepath\pgffillstroke


\pgfxyline(4,2.5)(4,3.5)
\pgfmoveto{\pgfxy(3.9,3.4)}\pgflineto{\pgfxy(4,3.5)}\pgflineto{\pgfxy(4.1,3.4)}\pgflineto{\pgfxy(3.9,3.4)}\pgfclosepath\pgffillstroke

\pgfxyline(3.5,3)(4.5,3)
\pgfmoveto{\pgfxy(3.5,3)}\pgflineto{\pgfxy(3.6,3.1)}\pgflineto{\pgfxy(3.6,2.9)}\pgflineto{\pgfxy(3.5,3)}\pgfclosepath\pgffillstroke


\pgfmoveto{\pgfxy(0.5,2)}\pgfpatharc{90}{0}{0.5cm}
\pgfstroke
\pgfmoveto{\pgfxy(0.5,2)}\pgflineto{\pgfxy(0.6,2.1)}\pgflineto{\pgfxy(0.6,1.9)}\pgflineto{\pgfxy(0.5,2)}\pgfclosepath\pgffillstroke

\pgfmoveto{\pgfxy(1,2.5)}\pgfpatharc{180}{270}{0.5cm}
\pgfstroke
\pgfmoveto{\pgfxy(1,2.5)}\pgflineto{\pgfxy(0.9,2.4)}\pgflineto{\pgfxy(1.1,2.4)}\pgflineto{\pgfxy(1,2.5)}\pgfclosepath\pgffillstroke


\pgfxyline(2,1.5)(2,2.5)
\pgfmoveto{\pgfxy(1.9,2.4)}\pgflineto{\pgfxy(2,2.5)}\pgflineto{\pgfxy(2.1,2.4)}\pgflineto{\pgfxy(1.9,2.4)}\pgfclosepath\pgffillstroke

\pgfxyline(1.5,2)(2.5,2)
\pgfmoveto{\pgfxy(1.5,2)}\pgflineto{\pgfxy(1.6,2.1)}\pgflineto{\pgfxy(1.6,1.9)}\pgflineto{\pgfxy(1.5,2)}\pgfclosepath\pgffillstroke

\pgfmoveto{\pgfxy(2.5,2)}\pgfpatharc{90}{0}{0.5cm}
\pgfstroke
\pgfmoveto{\pgfxy(2.5,2)}\pgflineto{\pgfxy(2.6,2.1)}\pgflineto{\pgfxy(2.6,1.9)}\pgflineto{\pgfxy(2.5,2)}\pgfclosepath\pgffillstroke

\pgfmoveto{\pgfxy(3,2.5)}\pgfpatharc{180}{270}{0.5cm}
\pgfstroke
\pgfmoveto{\pgfxy(3,2.5)}\pgflineto{\pgfxy(2.9,2.4)}\pgflineto{\pgfxy(3.1,2.4)}\pgflineto{\pgfxy(3,2.5)}\pgfclosepath\pgffillstroke


\pgfmoveto{\pgfxy(3.5,2)}\pgfpatharc{90}{0}{0.5cm}
\pgfstroke
\pgfmoveto{\pgfxy(3.5,2)}\pgflineto{\pgfxy(3.6,2.1)}\pgflineto{\pgfxy(3.6,1.9)}\pgflineto{\pgfxy(3.5,2)}\pgfclosepath\pgffillstroke

\pgfmoveto{\pgfxy(4,2.5)}\pgfpatharc{180}{270}{0.5cm}
\pgfstroke
\pgfmoveto{\pgfxy(4,2.5)}\pgflineto{\pgfxy(3.9,2.4)}\pgflineto{\pgfxy(4.1,2.4)}\pgflineto{\pgfxy(4,2.5)}\pgfclosepath\pgffillstroke


\pgfxyline(1,0.5)(1,1.5)
\pgfmoveto{\pgfxy(0.9,1.4)}\pgflineto{\pgfxy(1,1.5)}\pgflineto{\pgfxy(1.1,1.4)}\pgflineto{\pgfxy(0.9,1.4)}\pgfclosepath\pgffillstroke

\pgfxyline(0.5,1)(1.5,1)
\pgfmoveto{\pgfxy(0.5,1)}\pgflineto{\pgfxy(0.6,1.1)}\pgflineto{\pgfxy(0.6,0.9)}\pgflineto{\pgfxy(0.5,1)}\pgfclosepath\pgffillstroke


\pgfxyline(2,0.5)(2,1.5)
\pgfmoveto{\pgfxy(1.9,1.4)}\pgflineto{\pgfxy(2,1.5)}\pgflineto{\pgfxy(2.1,1.4)}\pgflineto{\pgfxy(1.9,1.4)}\pgfclosepath\pgffillstroke

\pgfxyline(1.5,1)(2.5,1)
\pgfmoveto{\pgfxy(1.5,1)}\pgflineto{\pgfxy(1.6,1.1)}\pgflineto{\pgfxy(1.6,0.9)}\pgflineto{\pgfxy(1.5,1)}\pgfclosepath\pgffillstroke


\pgfxyline(3,0.5)(3,1.5)
\pgfmoveto{\pgfxy(2.9,1.4)}\pgflineto{\pgfxy(3,1.5)}\pgflineto{\pgfxy(3.1,1.4)}\pgflineto{\pgfxy(2.9,1.4)}\pgfclosepath\pgffillstroke

\pgfxyline(2.5,1)(3.5,1)
\pgfmoveto{\pgfxy(2.5,1)}\pgflineto{\pgfxy(2.6,1.1)}\pgflineto{\pgfxy(2.6,0.9)}\pgflineto{\pgfxy(2.5,1)}\pgfclosepath\pgffillstroke

\pgfmoveto{\pgfxy(3.5,1)}\pgfpatharc{90}{0}{0.5cm}
\pgfstroke
\pgfmoveto{\pgfxy(3.5,1)}\pgflineto{\pgfxy(3.6,1.1)}\pgflineto{\pgfxy(3.6,0.9)}\pgflineto{\pgfxy(3.5,1)}\pgfclosepath\pgffillstroke

\pgfmoveto{\pgfxy(4,1.5)}\pgfpatharc{180}{270}{0.5cm}
\pgfstroke
\pgfmoveto{\pgfxy(4,1.5)}\pgflineto{\pgfxy(3.9,1.4)}\pgflineto{\pgfxy(4.1,1.4)}\pgflineto{\pgfxy(4,1.5)}\pgfclosepath\pgffillstroke

\pgfmoveto{\pgfxy(4.5,3)}\pgfpatharc{90}{0}{0.5cm}
\pgfstroke
\pgfmoveto{\pgfxy(4.5,3)}\pgflineto{\pgfxy(4.6,3.1)}\pgflineto{\pgfxy(4.6,2.9)}\pgflineto{\pgfxy(4.5,3)}\pgfclosepath\pgffillstroke

\pgfmoveto{\pgfxy(5,3.5)}\pgfpatharc{180}{270}{0.5cm}
\pgfstroke
\pgfmoveto{\pgfxy(5,3.5)}\pgflineto{\pgfxy(4.9,3.4)}\pgflineto{\pgfxy(5.1,3.4)}\pgflineto{\pgfxy(5,3.5)}\pgfclosepath\pgffillstroke


\pgfmoveto{\pgfxy(4.5,2)}\pgfpatharc{90}{0}{0.5cm}
\pgfstroke
\pgfmoveto{\pgfxy(4.5,2)}\pgflineto{\pgfxy(4.6,2.1)}\pgflineto{\pgfxy(4.6,1.9)}\pgflineto{\pgfxy(4.5,2)}\pgfclosepath\pgffillstroke

\pgfmoveto{\pgfxy(5,2.5)}\pgfpatharc{180}{270}{0.5cm}
\pgfstroke
\pgfmoveto{\pgfxy(5,2.5)}\pgflineto{\pgfxy(4.9,2.4)}\pgflineto{\pgfxy(5.1,2.4)}\pgflineto{\pgfxy(5,2.5)}\pgfclosepath\pgffillstroke

\pgfmoveto{\pgfxy(4.5,1)}\pgfpatharc{90}{0}{0.5cm}
\pgfstroke
\pgfmoveto{\pgfxy(4.5,1)}\pgflineto{\pgfxy(4.6,1.1)}\pgflineto{\pgfxy(4.6,0.9)}\pgflineto{\pgfxy(4.5,1)}\pgfclosepath\pgffillstroke

\pgfmoveto{\pgfxy(5,1.5)}\pgfpatharc{180}{270}{0.5cm}
\pgfstroke
\pgfmoveto{\pgfxy(5,1.5)}\pgflineto{\pgfxy(4.9,1.4)}\pgflineto{\pgfxy(5.1,1.4)}\pgflineto{\pgfxy(5,1.5)}\pgfclosepath\pgffillstroke

\pgfputat{\pgfxy(0.2,1)}{\pgfnode{rectangle}{center}{\color{black}\footnotesize $8$}{}{\pgfusepath{}}}
\pgfputat{\pgfxy(0.2,2)}{\pgfnode{rectangle}{center}{\color{black}\footnotesize $7$}{}{\pgfusepath{}}}
\pgfputat{\pgfxy(0.2,3)}{\pgfnode{rectangle}{center}{\color{black}\footnotesize $6$}{}{\pgfusepath{}}}

\pgfputat{\pgfxy(5.8,1)}{\pgfnode{rectangle}{center}{\color{black}\footnotesize $3$}{}{\pgfusepath{}}}
\pgfputat{\pgfxy(5.8,2)}{\pgfnode{rectangle}{center}{\color{black}\footnotesize $2$}{}{\pgfusepath{}}}
\pgfputat{\pgfxy(5.8,3)}{\pgfnode{rectangle}{center}{\color{black}\footnotesize $1$}{}{\pgfusepath{}}}

\pgfputat{\pgfxy(1,0.2)}{\pgfnode{rectangle}{center}{\color{black}\footnotesize $8$}{}{\pgfusepath{}}}
\pgfputat{\pgfxy(2,0.2)}{\pgfnode{rectangle}{center}{\color{black}\footnotesize $7$}{}{\pgfusepath{}}}
\pgfputat{\pgfxy(3,0.2)}{\pgfnode{rectangle}{center}{\color{black}\footnotesize $6$}{}{\pgfusepath{}}}
\pgfputat{\pgfxy(4,0.2)}{\pgfnode{rectangle}{center}{\color{black}\footnotesize $5$}{}{\pgfusepath{}}}
\pgfputat{\pgfxy(5,0.2)}{\pgfnode{rectangle}{center}{\color{black}\footnotesize $4$}{}{\pgfusepath{}}}

\pgfputat{\pgfxy(1,3.8)}{\pgfnode{rectangle}{center}{\color{black}\footnotesize $5$}{}{\pgfusepath{}}}
\pgfputat{\pgfxy(2,3.8)}{\pgfnode{rectangle}{center}{\color{black}\footnotesize $4$}{}{\pgfusepath{}}}
\pgfputat{\pgfxy(3,3.8)}{\pgfnode{rectangle}{center}{\color{black}\footnotesize $3$}{}{\pgfusepath{}}}
\pgfputat{\pgfxy(4,3.8)}{\pgfnode{rectangle}{center}{\color{black}\footnotesize $2$}{}{\pgfusepath{}}}
\pgfputat{\pgfxy(5,3.8)}{\pgfnode{rectangle}{center}{\color{black}\footnotesize $1$}{}{\pgfusepath{}}}
\end{pgfpicture}
\end{tabular}
\caption{A labelled $3\times 5$ (Cauchon-Le) diagram (left) with pipe dream construction (right).\label{FigRestrictedPerm}}
\end{center}
\end{figure}

 \begin{proposition}
Let $\lambda=\{ n=\lambda_1 \geq \dots \geq \lambda_m\}$ be a permutation. Let $\tau_{\lambda}$ its associated toric permutation, $D_{\lambda}$ its associated diagram and $w_\lambda$ its associated $w$-permutation.  Similarly to Section \ref{subsec-lambda-gamma}, we define $\gamma=\{\gamma_1 < \dots < \gamma_m\}$ by $\lambda_i+\gamma_i= n+i$ for all $1\leq i \leq m$.\begin{enumerate} 
\item $\tau_{\lambda}= w_0 w_{\lambda} w_0'$, where $w_0$ is the longest element of $S_{m+n}$ and $w_0'$ is the longest element of $S_m \times S_n$, that is, in 1-line notation, $w_0=[m+n ~m+n-1 ~\dots 2~1]$ and  $w_0'=[m ~m-1 ~\dots 2~1 ~m+n~m+n-1~ \dots m+2~m+1]$.
\item Let $\gamma'=\{\gamma'_1 < \dots < \gamma'_n\}$ be the complement of $\gamma$ in $\{1, \dots , m+n\}$. Then, in 1-line notation, we have:
$$w_{\lambda} = [\gamma_1 ~ \gamma_2 ~\dots \gamma_m ~ \gamma'_1 ~ \gamma'_2 ~ \dots ~ \gamma'_n].$$
\end{enumerate}
\end{proposition}
\begin{proof}
The first part is trivial as the permutations $w_0$ and $w_0'$ are just `fixing' the different ways of labelling the sides of a diagram in the constructions of toric and $w$ permutations.  

For the second part, we first consider the case where $i \leq m$. In this case, the pipe starting at $i$ will first go across $n-\lambda_i$ black boxes (starting from the right) on row $i$. After this, it will only encounter white boxes. As a consequence, we see that:
$$w_{\lambda}(i) = i +(n-\lambda_i) =\gamma_i.$$

To conclude we just observe that if $j>i>m$, then $w_{\lambda}(j)>w_{\lambda}(i)$.
\end{proof}

Putting together the previous proposition and Theorem \ref{thm-PIdegree-diag}, we are now ready to state our main result concerning partition subalgebras. 

\begin{theorem}
\label{thm-PIdegree-partition subalgebras}
Let $Y_{\lambda}$ be a Young tableau and $\tau$ be the toric permutation on $Y_{\lambda}$. 
Assume that $q$ is a primitive $\ell^{\mathrm{th}}$ root of unity with $\ell$ odd.\\
Then 
$$\PI(\yl)= \ell^{\frac{N-r}{2}},$$
where $N=\sum_{i=1}^m \lambda_i$  and $r$ is the number of odd cycles in the disjoint cycle decomposition of $\tau_{\lambda}= w_0 w_{\lambda} w_0'$.
\end{theorem}

For instance, for the partition $\lambda=(5,3,2)$, and $\ell=5$, we obtain $\tau_{\lambda}=(138764)(25)$ and so $\PI(\yl)= 5^{\frac{10-2}{2}}=5^4$.

\subsection{PI degree of quantum determinantal rings}

For $t\in \llbracket 1, n-1 \rrbracket$, let $\I_t$ be the \textit{quantum determinantal ideal} generated by all $(t+1)\times (t+1)$ quantum minors of $\oh_q(M_n(\K))$ and denote the \textit{quantum determinantal ring} as $R_t(M_n) := \oh_q(M_n(\K))/\I_t$. This is a noetherian domain, see for instance \cite{GoodearlLenagan}.

We begin by stating the following lemma, in which the first isomorphism is exactly  \cite[Lemma 4.4]{LenaganRigal-SchubertVar} and the second isomorphism follows from the map defined in   \cite[Proposition 3.7.1(3)]{ParshallWang} which preserves the normality of $\delta=[1,\ldots, t \mid 1,\ldots, t]$ when mapping it to $\delta_t=[n+t-1,\ldots, n\mid n+t-1,\ldots, n]$:
\begin{lemma}\label{lemequiv}
Let $\delta:=[1,\ldots, t \mid 1,\ldots, t]$ be a $t\times t$ quantum minor in $R=\oh_q(M_n(\K))$, for some $q\in \K^*$. Let $\bar{\delta}\in R_t(M_n)$ be its canonical image and $\delta_t:=[n+t-1,\ldots, n,\mid n+t-1,\ldots, n]$ be a $t\times t$ quantum minor in $R^{\mathrm{op}}$. Then
\[ R_t(M_n)[\bar{\delta}^{-1}]\cong A_t[\delta^{-1}] \cong B_t[\delta_t^{-1}],\]
where $A_t$ and $B_t$ are the following subalgebras:
\begin{align*}
A_t&{}:=\langle X_{i,j}\in R \mid i\leq t \text{ or } j\leq t \rangle \subseteq R,\\
B_t&{}:=\langle X_{i,j}\in R^{\mathrm{op}} \mid i\geq n-t+1 \text{ or } j\geq n-t+1 \rangle \subseteq R^{\mathrm{op}}.
\end{align*}
\end{lemma}

\begin{remark}\label{remFracR_tB_t}
Lemma \ref{lemequiv} implies that $\mathrm{Frac}(R_t(M_n)) \cong \mathrm{Frac}(B_t)$.
\end{remark}

We point out that $B_t$ may be expressed as an iterated Ore extension over $\K$ by taking the nonzero indeterminates $X_{i,j}\in B_t$, for $(i,j)\in \{(1, n-t+1), (1,n-t+2), \ldots, (n,n)\}$, and adjoining them in lexicographic order to the base field, $\K$. Let $\alpha_k$ be the $\K$-automorphism, and $\beta_k$ be the $\alpha_k$-derivation, appearing in the $k^{\text{th}}$ Ore extension of $B_t$ (these are easily computed using the commutation rules of $R$), and let $Y_{k,t}$ and $Y_{l,t}$ be the indeterminates in the $k^{\text{th}}$ and $l^{\text{th}}$ extensions of $B_t$ respectively. That is, 
\[Y_{1,t}=X_{1,n-t+1}, ~~ \ldots,  ~~ Y_{t,t}=X_{1,n}, ~~ Y_{t+1, t}= X_{2,n-t+1}, ~~ \ldots , ~~ Y_{2nt-t^2,t}=X_{n,n}. \]
Then,
\begin{align*}
 B_t &{}= \K[X_{1,n-t+1}][X_{1,n-t+2}; \alpha_2, \beta_2] \cdots [X_{n,n};\alpha_{2nt-t^2}, \beta_{2nt-t^2}]  \\
 &{} = \K[Y_{1,t}][Y_{2,t}; \alpha_2, \beta_2] \cdots [Y_{2nt-t^2,t};\alpha_{2nt-t^2}, \beta_{2nt-t^2}] 
\end{align*}
and, for all $1\leq l < k \leq 2nt-t^2$, we have
\begin{align*}
Y_{k,t} \ast Y_{l,t} &{}= \alpha_k(Y_{l,t}) \ast Y_{k,t} + \beta_k(Y_{l,t})\\
&{}= q^{m_{k,l}}Y_{l,t} \ast Y_{k,t} + \beta_k(Y_{l,t}),
\end{align*}
for some $m_{k,l}\in \Z$. These $m_{k,l} \in \Z$ form a skew-symmetric matrix $M_{n,t}=(m_{k,l})\in M_{2nt-t^2}(\Z)$. Note that we use $\ast$ above to denote the multiplication in the opposite algebra.

It was shown in \cite[Example 5.4]{Haynal} that $R$ satisfies the conditions of \cite[Theorem 4.6]{Haynal}. As $B_t$ is isomorphic to a subalgebra of $R$, it too must satisfy these conditions because the relevant properties of the iterated Ore extension of $R$ are preserved in $B_t$. We may therefore apply \cite[Corollary 4.7]{Haynal} to $B_t$ and observe the following results:
\begin{lemma}\label{lemHaynal} \
\begin{enumerate}
\item $\Frac(R_t(M_n)) \cong \mathrm{Frac}(B_t) \cong \mathrm{Frac}(\oh_{q^{M_{n,t}}}(\K^{2nt-t^2}))$.
\item If $q$ is a root of unity, then $B_t$ is a PI algebra and $$\PI(R_t(M_n))=\PI(B_t)=\PI(\oh_{q^{M_{n,t}}}(\K^{2nt-t^2})).$$
\end{enumerate}
\end{lemma}


It can easily be verified that $-M_{n,t}=M(C)$, where $C$ is the $n\times n$ Cauchon-Le diagram whose last $t$ rows and $t$ columns are white (and all other boxes are black). It follows from  \cite[Corollary]{PanovEnglish} that the PI degree of $R_t(M_n)$ depends only on the properties of this matrix $M(C)$.  Since these specific diagrams and matrices are important for the results that follow, we define notation for them in the following corollary to Lemmas \ref{lemHaynal} and \ref{lemPIdeg}:
\begin{corollary}\label{corNotation}
Let $C_{n,t}$ be the $n\times n$ Cauchon-Le diagram whose last $t$ rows and $t$ columns are white (and all other boxes are black) and let $\mathcal{M}_{n,t}:=M(C_{n,t})$ be its associated matrix with invariant factors $\{h_i\}_i$. Then:
\begin{enumerate}
\item $M_{n,t}\sim_C \M_{n,t}$ and thus $\Frac(R_t(M_n)) \cong \mathrm{Frac}(\oh_{q^{M_{n,t}}}(\K^{2nt-t^2})) \cong \mathrm{Frac}(\oh_{q^{\M_{n,t}}}(\K^{2nt-t^2}))$.

\item If $q$ is a primitive $\ell^{\mathrm{th}}$ root of unity, then 
\[\PI(R_t(M_n))=\PI(\oh_{q^{\M_{n,t}}}(\K^{2nt-t^2}))=\prod_{i=1}^{\frac{2nt-t^2-\dim(\ker(\M_{n,t}))}{2}} \frac{\ell}{\gcd(h_i, \ell)}.\]
\end{enumerate}
\end{corollary}
Having proved that all invariant factors of $\mathcal{M}_{n,t}$ are powers of 2 (Theorem \ref{leminvariant}), we now wish to find $\dim(\ker(\M_{n,t}))$. We observe that \cite[Proposition 6.2]{JakobsenJondrup} have computed the PI degree of $R_t(M_n)$ over a field of characteristic $0$. However, our above corollary shows that $\PI(R_t(M_n))$ is independent of the characteristic of the base field $\K$, and so we conclude from \cite[Proposition 6.2]{JakobsenJondrup} the following generalisation of their result. 

\begin{theorem}\label{PIdegQDR}
Let $n\in \N_{>0}$ and $t \in \llbracket 1, n-1 \rrbracket$ and take $q$ to be a primitive $\ell^{\mathrm{th}}$ root of unity with $\ell>2$. Then
\[\PI(R_t(M_n))= \begin{cases} \ell^{\frac{2nt-t^2-t}{2}} & \ell \text{ is odd}; \\ 	
										 \ell^{\frac{2nt-t^2-t}{2}} \cdot 2^{-nt+\frac{t^2+t}{2}+n-1}& \ell \text{ is even}.
							\end{cases} \]
\end{theorem}

We note that the above result shows that the toric permutation associated to  $C_{n,t}$ has $t$ odd cycles in its disjoint cycle decomposition. While we will not need to be more precise, for completeness, we state the following result whose proof can be found in \cite[Proposition 6.10]{Alex-thesis}. 

\begin{proposition}\label{lemtoricperm}
Fix some $n\in \N_{>0}$ and $t\in \llbracket 1, n-1 \rrbracket$. Then the toric permutation, $\tau$, associated to $C_{n,t}$ is the product of $t$ disjoint odd cycles $\tau=(\eta_1\cdots \eta_t)$. 

In particular, write $n=ut+r$, for some $u\in \N$ and $r\in \llbracket 0, t-1\rrbracket$. Then the cycles depend on $n-t$ and $t$, and are the following:

$\bullet$ \underline{$0<n-t<t$}:

\begin{equation*}
\eta_i:=\begin{cases}
		(i,i+t,2t+r+i,t+r+i), & 1\leq i \leq r;\\
		(i,t+r+i), &  r< i \leq t.
		\end{cases}
\end{equation*}

$\bullet$ \underline{$n-t=t$}:

\begin{equation*}
\eta_i:=(i,\, i+t,\, i+t+n,\, i+n), \quad 1\leq i \leq t.
\end{equation*}

$\bullet$ \underline{$t<n-t$}:

\begin{equation*}
\eta_i:=\begin{cases}
(i,\, i+t,\, \ldots,\, i+ut,\, i+ut+n,\, i+(u-1)t+n,\, \ldots,\, i+n), & 1\leq i \leq r; \\
(i,\, i+t,\, \ldots,\, i+(u-1)t,\, i+(u-1)t+n,\, i+(u-2)t+n,\, \ldots,\, i+n), & r< i \leq t.
\end{cases}
\end{equation*}
\end{proposition}

\section{Irreducible Representations of Quantum Determinantal Rings}\label{sectionIrreds}
In this section we construct an irreducible representation for $R_t(M_n)$ at a root of unity, with dimension equal to the PI degree calculated above.

\begin{definition}
An index pair $(I,J)\in \Delta_{n,n}$ corresponds to a \emph{final quantum minor} of $\oh_q(M_n(\K))$ of size $s\in \llbracket 1, n \rrbracket$  if $I=\llbracket i, i+(s-1)\rrbracket$ and $J=\llbracket j, j+(s-1)\rrbracket$, and either $i+(s-1)=n$ or $j+(s-1)=n$.
\end{definition}
We denote the set of all final quantum minors of $\oh_q(M_n(\K))$ by $\Omega_{n,n}\subseteq \Delta_{n,n}$, and the set of all final quantum minors of $\oh_q(M_n(\K))$ of size less than or equal to $t$ by $\Omega_{n,n}^t\subseteq \Omega_{n,n}$. Note that all the quantum minors in $\Omega_{n,n}^t$ ``survive'' in the quotient algebra $R_t(M_n)$ and generate a subalgebra, which we denote by $\K\langle \Omega_{n,n}^t \rangle \subseteq R_t(M_n)$. It may also be verified that $|\Omega_{n,n}^t|= 2nt-t^2$.

\begin{lemma}\label{lemAQAS}
Let $A=\K\langle \Omega_{n,n}^t \rangle$ and $\Sigma \subset A$ be the multiplicative set generated by $\Omega_{n,n}^t$. Then,
\begin{enumerate}[(i)]
\item The elements of $\Omega_{n,n}^t$ commute with each other up to powers of $q$;
\item $A \subseteq R_t(M_n) \subseteq A\Sigma^{-1}$;
\item $A$ is a quantum affine space $\oh_{q^{M'}}(\K^{2nt-t^2})$, for some $M'\in M_{2nt-t^2}(\Z)$, and $A\Sigma^{-1}$ is the quantum torus associated to $A$;
\item $M'\in M_{2nt-t^2}(\Z)$ has rank $2nt-t^2-t$ and all its invariant factors are powers of $2$. 
\end{enumerate}
\end{lemma}
\begin{proof}
\
\begin{enumerate}[(i)]
\item This result follows from an analogous result for initial quantum minors of $\oh_q(M_n(\K))$: a quantum minor $[I|J] \in \Delta_{n,n}$ is called \emph{initial} if $I=\{i, i+1, \ldots, i+t\}$ and $J=\{j, j+1, \ldots, j+t\}$, for some $t \in \llbracket 0, n \rrbracket$, and $1\in I \cup J$. Noting that all pairs of initial quantum minors $[I|J], \, [M|N]$ satisfying $i=m=1$ are \emph{weakly separated} in the sense of Leclerc and Zelevinsky \cite{LeclercZelevinsky}, we may apply their result \cite[Lemma 2.1]{LeclercZelevinsky} to see that the $[I|J], \, [M|N]$ above quasi-commute (that is, commute up to a power of $q$). Applying the transpose automorphism $\tau_q^1$ of \cite[Proposition 3.7.1(1)]{ParshallWang} reveals that any pair of initial quantum minors $[I'|J'], \, [M'|N']$ satisfying $j'=n'=1$ also quasi-commute, as is shown in \cite[(4-13)]{Goodearl}. Finally, \cite[Corollary 6.5]{Goodearl} shows that pairs of initial quantum minors $[I|J], \, [M'|N']$ satisfying $i=n'=1$ quasi-commute and hence all initial quantum minors of $\oh_q(M_n(\K))$ quasi-commute with each other. This result passes to the final quantum minors of $\oh_q(M_n(\K))$, and hence to all elements of $\Omega_{n,n}^t$, by use of the anti-automorphism $\tau_q^2$ of \cite[Proposition 3.7.1(2)]{ParshallWang} and the property found in  \cite[Lemma 4.3.1]{ParshallWang}.

\item $A\subseteq R_t(M_n)$ is immediate. We deduce from part (i) that the elements of $\Omega_{n,n}^t$ are normal in $A$. They are also regular, since $R_t(M_n)$ is a noetherian domain. Therefore, $\Sigma\subseteq A$ is an Ore set at which we can localise. We now show that each generator of $R_t(M_n)$ is contained in $A \Sigma^{-1}$:

Define sets $Q_1:=\llbracket n-t+1, n \rrbracket \times \llbracket 1, n \rrbracket, \;Q_2:=\llbracket 1, n \rrbracket \times \llbracket n-t+1, n\rrbracket$, and $Q:=Q_1\cup Q_2$. To prove $R_t(M_n)\subseteq A\Sigma^{-1}$ we first use decreasing induction on $(i,j) \in Q$ (for the lexicographic order) to show that $\bar{X}_{i,j}\in A\Sigma^{-1}$, where $\bar{X}_{i,j}\in R_t(M_n)$ is the canonical image of the generator $X_{i,j}\in R$. We then show that $\bar{X}_{i,j}\in A\Sigma^{-1}$, for all $(i,j)\in \llbracket 1, n-t \rrbracket \times \llbracket 1, n-t \rrbracket$. This will prove the result for all $ (i,j) \in  \llbracket 1, n \rrbracket \times \llbracket 1, n \rrbracket$ since $\llbracket 1, n \rrbracket \times \llbracket 1, n \rrbracket = Q \cup (\llbracket 1, n-t \rrbracket \times \llbracket 1, n-t \rrbracket)$.

We start with the induction argument. Since $\bar{X}_{i,j}\in A$ for all $(i,j)\in (\{n\} \times \llbracket 1, n \rrbracket)\cup (\llbracket 1, n \rrbracket \times \{n\})$, by definition, the first non-trivial $(i,j)\in Q$ to prove in the induction is when $(i,j)=(n-1, n-1)$. We take this as our base case.  

By the definition of quantum minors, we may write 
\[[\{n-1, n\}|\{n-1, n\}]=\bar{X}_{n-1,n-1}\bar{X}_{n,n} - q\bar{X}_{n-1,n}\bar{X}_{n,n-1}.\]
Rearranging this we obtain
\[ \bar{X}_{n-1,n-1} = ([\{n-1, n\}|\{n-1, n\}] + q\bar{X}_{n-1,n}\bar{X}_{n,n-1} )\bar{X}_{n,n}^{-1} \in A\Sigma^{-1}. \]
Since  $[\{n-1, n\}|\{n-1, n\}], \, \bar{X}_{n-1,n}, \, \bar{X}_{n,n-1}, \, \bar{X}_{n,n} \in A$, and $\bar{X}_{n,n}^{-1}\in A\Sigma^{-1}$, this proves the base case.

For ease of reading we now fix some notation: for some $(i,j)\in Q$, denote by $[I_{i,j}|J_{i,j}]\in \Delta_{n,n}$ the final quantum minor which has $i$ as its first row index and $j$ as its first column index; that is,
\[ I_{i,j}:=\llbracket i, i+\min\{n-i, n-j\} \rrbracket, ~~ J_{i,j}:=\llbracket j, j+\min\{n-i, n-j\} \rrbracket. \]
Setting $s_{i,j}:=\min\{n-i, n-j\}$ we see that $[I_{i,j}|J_{i,j}]$ has size $s_{i,j}+1$ and  $s_{i,j}+1 \in \llbracket 1, t \rrbracket$ because $(i,j)\in Q$. Therefore, $[I_{i,j} | J_{i,j}] \in \Omega_{n,n}^t \subseteq A$ (note that to avoid cumbersome notation we do not distinguished between a quantum minor and its coset). One may also verify that $[I_{i,j} | J_{i,j}]$ is generated by sums of monomials in $\bar{X}_{l,k}$, where $(l,k)\geq (i,j)$ and $(l,k)\in Q$.

For the inductive step, fix some $(i,j)\in Q$, with $(i,j)<(n-1,n-1)$, and assume that $\bar{X}_{l,k}\in A\Sigma^{-1}$, for all $(i,j) < (l,k) \leq (n,n)$ and $(l,k) \in Q$. Consider $\bar{X}_{i,j}$. With the notation above we can rewrite $[I_{i,j} | J_{i,j}] \in A$, using the quantum Laplace relations from \cite[Proposition 1.1]{Noumi}, as
\begin{align}\label{eqnMinor}
[I_{i,j}|J_{i,j}] =  \bar{X}_{i,j}[I_{i+1,j+1} | J_{i+1,j+1}] + \sum_{r=1}^{s_{i,j}}(-q )^{r} \bar{X}_{i, j+r}[I_{i,j}\backslash \{i\} | J_{i,j}\backslash \{j+r\}].
\end{align} 
Since $(i+1, j+1) \in Q$ then $[I_{i+1, j+1} | J_{i+1, j+1}] \in \Omega_{n,n}^t$ and is hence invertible in $A\Sigma^{-1}$. Using the inductive hypothesis we deduce that $[I_{i,j}\backslash \{i\} | J_{i,j}\backslash \{j+r\}] \in A\Sigma^{-1}$, for all $r \in \llbracket 0, s_{i,j} \rrbracket$, as it is the sum of monomials in $\bar{X}_{l,k}$, where $(l,k)>(i,j)$ and $(l,k)\in Q$.
Similarly, $\bar{X}_{i, j+r} \in A\Sigma^{-1}$, for all $r \in \llbracket 1, s_{i,j}\rrbracket$, because $(i, j+r)\in Q$ and $(i, j+r)>(i,j)$. We may therefore rearrange (\ref{eqnMinor}) to obtain
\begin{align*}
\bar{X}_{i,j}=\left([I_{i,j}|J_{i,j}] - \sum_{r=1}^{s_{i,j}}(-q )^{r} \bar{X}_{i, j+r}[I_{i,j}\backslash \{i\} | J_{i,j}\backslash \{j+r\}] \right) [I_{i+1,j+1} | J_{i+1,j+1}]^{-1} \in A\Sigma^{-1}.
\end{align*}
This proves the inductive step and we conclude that $\bar{X}_{i,j}\in A\Sigma^{-1}$, for all $(i,j)\in Q$.

For any $(i,j)\in \llbracket 1, n-t \rrbracket \times  \llbracket 1, n-t \rrbracket$, define
\[\hat{I}_{i,j}:= \{i, n-t+1, n-t+2, \ldots, n\}, \quad \hat{J}_{i,j}:=\{j, n-t+1, n-t+2, \ldots, n\} \]
so that $|\hat{I}_{i,j}|=|\hat{J}_{i,j}|=t+1$. Then $[\hat{I}_{i,j} | \hat{J}_{i,j}] =0$ in $R_t(M_n)$ and hence also in $A$. We may write this as
\begin{align}\label{eqn0Minor}
0= \bar{X}_{i,j}[\hat{I}_{i,j}\backslash \{i\} | \hat{J}_{i,j}\backslash \{j\}] + \sum_{r=1}^{t}(-q )^{r} \bar{X}_{i, n-t+r}[\hat{I}_{i,j}\backslash \{i\} | \hat{J}_{i,j}\backslash \{n-t+r\}].
\end{align}
Note that, for all $r \in \llbracket 1, t \rrbracket$, the quantum minor $[\hat{I}_{i,j}\backslash \{i\} | \hat{J}_{i,j}\backslash \{n-t+r\}]$ is a sum of monomials in $\bar{X}_{l,k}$, where $l\in \llbracket n-t+1, n \rrbracket$ so that $(l,k)\in Q$. By the induction above we deduce that $\bar{X}_{l,k}\in A\Sigma^{-1}$ and hence $[\hat{I}_{i,j}\backslash \{i\} | \hat{J}_{i,j}\backslash \{n-t+r\}]\in A\Sigma^{-1}$,  for all $r \in \llbracket 1, t \rrbracket$. Furthermore, $[\hat{I}_{i,j}\backslash \{i\} | \hat{J}_{i,j}\backslash \{j\}] = [I_{n-t+1, n-t+1}, I_{n-t+1, n-t+1}] \in \Omega_{n,n}^t$ and is therefore invertible in $A\Sigma^{-1}$. Rearranging (\ref{eqn0Minor}) we obtain
 \[\bar{X}_{i,j}= \left(-\sum_{r=1}^{t}(-q )^{r} \bar{X}_{i, n-t+r}[\hat{I}_{i,j}\backslash \{i\} | \hat{J}_{i,j}\backslash \{n-t+r\}]\right) [\hat{I}_{i,j}\backslash \{i\} | \hat{J}_{i,j}\backslash \{j\}]^{-1} \in A\Sigma^{-1}. \]
Hence $\bar{X}_{i,j}\in A\Sigma^{-1}$, for all $(i,j) \in \llbracket 1, n-t \rrbracket \times \llbracket 1, n-t \rrbracket$, and, together with the induction proof above, this proves that $R_t(M_n) \subseteq A\Sigma^{-1}$.

\item 
It was shown in part (i) that elements of $\Omega_{n,n}^t$ commute with each other up to powers of $q$ as determined by a skew-symmetric matrix, which we denote by $M'\in M_{2nt-t^2}(\Z)$.  Hence, there is a surjective homomorphism
\begin{align}\label{eqnfSurj}
f: \oh_{q^{M'}}(\K^{2nt-t^2}) &{} \longrightarrow A.
\end{align}
We will use the GK-dimension to show that this surjection is, in fact, an isomorphism.

\textbf{Claim 1:} $\mathrm{GKdim}(A) = \mathrm{GKdim}(R_t(M_n))=2nt-t^2$.

When  $q$ is not a root of unity then \cite[Corollary 4.8]{GoodearlLaunoisLenagan-Tauvel} applied to $R_t(M_n)$, with $J_w=\I_t$ and $w=\llbracket 1, n-t \rrbracket \times \llbracket 1, n-t \rrbracket$, shows that $R_t(M_n)$ is Tdeg-stable in the generic case. When $q$ is a root of unity then $R_t(M_n)$ is a PI algebra which is a noetherian domain. Thus  \cite[Theorem 5.3]{Zhang} says that $R_t(M_n)$ is also Tdeg-stable in the root of unity case. Taking total rings of fractions of the algebras in part (ii) (possible since all algebras involved are noetherian domains), we obtain
\[ \Frac(A) \subseteq \Frac(R_t(M_n))\subseteq \Frac(A\Sigma^{-1})= \Frac(A) \quad \implies \quad \Frac(R_t(M_n)) = \Frac(A).\]
The algebras $A$ and $R_t(M_n)$ therefore satisfy the conditions of \cite[Proposition 3.5(3)]{Zhang}, which tells us that $\mathrm{GKdim}(A)\geq \mathrm{GKdim}(R_t(M_n))$. The equality of GK-dimension follows from the basic property that $A\subseteq R_t(M_n)$ implies $\mathrm{GKdim}(A)\leq \mathrm{GKdim}(R_t(M_n))$.

Finally, we see that $\mathrm{GKdim}(R_t(M_n))=2nt-t^2$, using \cite[Remark 4.2 (iii) \& (iv)]{LenaganRigal-SchubertVar}, replacing the $t$ with $t+1$ to make the result in the paper compatible with our definition of $R_t(M_n)$.

\textbf{Claim 2:} The map, $f$, in (\ref{eqnfSurj}) is an isomorphism.

From (\ref{eqnfSurj}) and the First Isomorphism Theorem, we have $\oh_{q^{M'}}(\K^{2nt-t^2})/\ker(f) \cong A$ and hence
\[ \mathrm{GKdim}\left(\oh_{q^{M'}}(\K^{2nt-t^2})/\ker(f)\right) =\mathrm{GKdim}(A) = 2nt-t^2.\] 
As a consequence of Goldie's Theorem (see \cite[Corollary 6.4]{GoodearlWarfield}), every nonzero ideal of $\oh_{q^{M'}}(\K^{2nt-t^2})$ contains a regular element. Therefore, if $\ker(f)$ were not trivial then, by \cite[Proposition 3.15]{KrauseLenagan}, 
\begin{align}\label{eqnGKdims}
\mathrm{GKdim}\left(\oh_{q^{M'}}(\K^{2nt-t^2})/\ker(f)\right) +1 &{} \leq \mathrm{GKdim}\left(\oh_{q^{M'}}(\K^{2nt-t^2})\right).
\end{align}
However, $\mathrm{GKdim}\left(\oh_{q^{M'}}(\K^{2nt-t^2})\right) = 2nt-t^2$, by \cite[Lemma 2]{LeroyMatczukOkninski}, so the inequality in (\ref{eqnGKdims}) becomes $2nt-t^2+1 \leq 2nt-t^2$, which is clearly false. Therefore $\ker(f)$ must be trivial and $\oh_{q^{M'}}(\K^{2nt-t^2}) \cong A$.

\item 
In Corollary \ref{corNotation} it was shown that $\Frac(R_t(M_n)) \cong \Frac(\oh_{q^{\M_{n,t}}}(\K^{2nt-t^2}))$ and in Theorem \ref{leminvariant} and Proposition \ref{lemtoricperm} we deduced that $\M_{n,t}$ has rank $2nt-t^2-t$ and that all its invariant factors are powers of $2$. Parts (ii) and (iii) of this lemma imply that  $\Frac(R_t(M_n))=\Frac(A) \cong \Frac(\oh_{q^{M'}}(\K^{2nt-t^2}))$, therefore
\[ \Frac(\oh_{q^{M'}}(\K^{2nt-t^2})) \cong \Frac(R_t(M_n))\cong \Frac(\oh_{q^{\M_{n,t}}}(\K^{2nt-t^2})). \]
Specialising $q$ to be a non-root of unity allows us to apply \cite[Theorem 2.19]{PanovRussian} to this isomorphism, which states that
\[ \Frac(\oh_{q^{M'}}(\K^{2nt-t^2})) \cong \Frac(\oh_{q^{\M_{n,t}}}(\K^{2nt-t^2})) \quad \iff \quad M'  \sim_C \M_{n,t}. \]
Hence, $M'$ shares the same invariant factors and rank as $\M_{n,t}$ thus proving part (iv) in the case when $q$ is not a root of unity. Since the matrix $M'$ also defines the commutation relations on the quantum affine space $A$ when $q$ is a root of unity, this also proves part (iv) for any $1\neq q\in \K^*$.
\end{enumerate}
\end{proof}

For the rest of this section we take $\K$ to be algebraically closed as this is required for the application of the Jacobson Density Theorem in Proposition \ref{propReponQAS}(ii). All results that follow in this section will take $q$ to be a root of unity, as will be clearly stated.

We start by constructing an irreducible representation of a quantum affine space at a root of unity, with dimension equal to its PI degree.  This result is not believed to be new, however, we present a proof of it here to convince the reader that we do indeed get an irreducible representation, as this is used heavily in the results that follow.
The notation of the lemma has been chosen to suit its application in the results to come, in which we will take tensor products of a family of quantum affine planes indexed by $i$.

\begin{lemma}\label{lemReponQAP}
Let $q$ be a primitive $\ell^{\mathrm{th}}$ root of unity, for some $\ell > 1$. Let $\K_{q^{h_i}}[x_i, y_i]$ denote the quantum affine plane with relations $x_i y_i=q^{h_i}y_i x_i$, for some $h_i\in \Z\backslash \{0\}$ satisfying $\gcd(h_i, \ell)=1$.  Let $V_i$ be an $\ell$-dimensional $\K$-vector space with basis $\{v^{(i)}_1, \ldots, v^{(i)}_{\ell}\} \subseteq V_i$ and define the map 
\begin{align} \label{reponQAS}
\varphi_i: \K_{q^{h_i}}[x_i, y_i] &{} ~ \longrightarrow ~ \End_\K(V_i) \\
x_i &{} ~ \longmapsto ~ \varphi_i(x_i) \nonumber \\
y_i &{} ~ \longmapsto ~ \varphi_i(y_i), \nonumber
\end{align}
where the endomorphisms $\varphi_i(x_i), \, \varphi_i(y_i)$ act on the basis vectors $v_j^{(i)} \in \{v_1^{(i)}, \ldots, v_{\ell}^{(i)} \}$ in the following way:
\begin{align}
\varphi_i(x_i)\cdot v^{(i)}_j &{}= \lambda_i q^{(j-1)h_i} v^{(i)}_j, \label{eqnActionXonBasis}\\
\varphi_i(y_i)\cdot v^{(i)}_j &{}= \begin{cases} v^{(i)}_{j+1}, & \text{ if } j \in \llbracket 1, \ell-1 \rrbracket; \\
											v^{(i)}_1, & \text{ if } j=\ell,
							\end{cases} \label{eqnActionYonBasis}
\end{align}
for some $\lambda_i\in K^*$. In particular, we can write $\varphi_i(x_i), \, \varphi_i(y_i)$ as matrices in $M_{\ell}(\K)$ with respect to this basis:
\begin{equation}\label{matrices}
 \varphi_i(x_i)= \left(\begin{smallmatrix} \lambda_i &  & &  & \\
												 & \lambda_i q^{h_i} & & & \\
												  &  & \lambda_i q^{2 h_i} &  & \\
												  & & & \ddots &  \\
												  &  & &  & \lambda_i q^{(\ell-1)h_i}
							\end{smallmatrix} \right), \qquad  \varphi_i(y_i) = \left(\begin{smallmatrix} 0 & 0 &  \ldots & 0 & 1 \\
											1 & 0 &  \ldots & 0 & 0 \\
											0& 1 &  \ldots & 0 & 0 \\
											\vdots &  & \ddots  &\vdots & \vdots \\
											0 & \ldots & \ldots  & 1 & 0 \end{smallmatrix}\right).
\end{equation}

Then $\varphi_i$ is a surjective algebra homomorphism which defines an irreducible representation of $\K_{q^{h_i}}[x_i, y_i]$ on $V$ of dimension $\ell$ and satisfies the property $\varphi_i(x_i)^{\ell}= \lambda_i^{\ell} \Id_{V_i}$ and $\varphi_i(y_i)^{\ell}= \Id_{V_i}$.
\end{lemma} 
\begin{proof}
Using (\ref{eqnActionXonBasis}) and (\ref{eqnActionYonBasis}) it may be verified that $\varphi_i(x_i) \varphi_i(y_i)=q^{h_i} \varphi_i(y_i) \varphi_i(x_i)$, so that the relations between generators $x_i, y_i \in \K_{q^{h_i}}[x_i, y_i]$ are preserved in the image of $\varphi_i$. By the universal property of algebras with generators and relations, $\varphi_i$ then becomes an algebra homomorphism and thus defines a representation of $\K_{q^{h_i}}[x_i, y_i]$ on $V_i$. This gives $V_i$ a $\K_{q^{h_i}}[x_i, y_i]$-module structure with multiplication defined as $rv=\varphi_i(r)\cdot v$, for all $r\in \K_{q^{h_i}}[x_i, y_i]$ and $v\in V$.

To show that $(\varphi_i, V_i)$ defines an irreducible representation of the quantum affine space, we suppose that there exists a nonzero subspace $W\subseteq V_i$ which is a sub-representation of $V_i$ and we argue that $W=V_i$. Assume (for contradiction) that $v_j^{(i)} \notin W$, for all $j \in \llbracket 1, \ell \rrbracket$, and let $w:=\sum_{j=1}^{\ell} \alpha_j v^{(i)}_j \in W$, with $\alpha_1, \ldots, \alpha_{\ell} \in \K$, be a nonzero element of $W$ which is minimal with respect to the number of nonzero summands. Note that the assumption $v_j^{(i)}\notin W$ implies that at least two coefficients in $\{\alpha_1, \ldots, \alpha_{\ell}\}$ are nonzero. Since $W$ is a sub-representation then $\varphi_i(y_i)\cdot w \in W$ and $\varphi_i(x_iy_i)\cdot w \in W$, where
\begin{align*}
\varphi_i(y_i)\cdot w &{}= \sum_{j=1}^{\ell-1} \alpha_j v^{(i)}_{j+1} + \alpha_{\ell} v^{(i)}_1,  \\
\varphi_i(x_iy_i)\cdot w = \varphi_i(x_i) \cdot (\varphi_i(y_i) \cdot w) &{}=  \sum_{j=1}^{\ell-1} \lambda_i q^{j h_i} \alpha_j v^{(i)}_{j+1} + \lambda_i \alpha_{\ell} v^{(i)}_1.
\end{align*} 
We now use these elements of $W$ to arrive at a contradiction. If $\alpha_{\ell} \neq 0$ then we consider $\varphi_i(y_i - \lambda_i^{-1} x_i y_i) \cdot w \in W$, where
\[\varphi_i(y_i - \lambda_i^{-1} x_i y_i) \cdot w = \sum_{j=1}^{\ell-1} \alpha_j(1-q^{j h_i}) v^{(i)}_{j+1}. \]
This is nonzero, since $\ell \nmid jh_i$ for all $j \in \llbracket 1, \ell-1 \rrbracket$, and it has fewer summands than $w$, thus contradicting the minimality of $w$. If, instead, $\alpha_{\ell}=0$ and $\alpha_k\neq 0$, for some $k \in \llbracket 1, \ell-1 \rrbracket$, then we consider the element $\varphi_i(y_i - q^{-k h_i} \lambda_i^{-1} x_i y_i) \cdot w$ and argue in a similar way to show that this also contradicts the minimality of $w$. Therefore, our assumption is false and there must exist some $j\in \llbracket 1, \ell \rrbracket$ such that $v_j^{(i)}\in W$. But then $W=V_i$, since $\varphi_i(y_i)$ permutes all the basis vectors, as can be seen in (\ref{eqnActionYonBasis}). We deduce from this that $(\varphi_i, V_i)$ is irreducible.

To verify surjectivity of $\varphi_i$, note that $V_i$ is a simple $\K_{q^{h_i}}[x_i, y_i]$-module, finite-dimensional over $\K$, and $\K$ is algebraically closed. Thus, by Schur's Lemma, $\End_{\K_{q^{h_i}}[x_i, y_i]}(V_i) \cong \K$ and we can apply the Jacobson Density Theorem to $R=\K_{q^{h_i}}[x_i, y_i]$, with  $R'= \End_{\K_{q^{h_i}}[x_i, y_i]}(V_i)$, to see that the map $\rho: \K_{q^{h_i}}[x_i, y_i] \rightarrow \End_{\K}(V_i)$, sending $r$ to $f_r$, is surjective. Recall that $f_r$ was defined as  $f_r(v)=rv$, for all $r\in \K_{q^{h_i}}[x_i, y_i]$ and $v\in V_i$, and this is precisely the definition of $\varphi_i(r)$ as stated earlier. Hence $\varphi_i= \rho$ is surjective.

Finally, it is easily verified that $\varphi_i(x_i)^{\ell}= \lambda_i^{\ell} \Id_{V_i}$ and $\varphi_i(y_i)^{\ell}= \Id_{V_i}$. This proves the final property in the statement.
\end{proof}

Consider a quantum affine space $\K_{q^M}[T_1, \ldots, T_N]$ for some matrix $M$ with a kernel of dimension $t \in \llbracket 0, N-1 \rrbracket$ and invariant factors $\{h_i\}_{i=1}^s$, where $2s=N-t$. The skew normal form of $M$ is $S=EME^T \in M_{N}(\Z)$, where $E=(e_{i,j})_{i,j}\in M_N(\Z)$ is invertible. We denote the quantum affine space associated to $S$ as
\[D:=\K_{q^S}[x_1, y_1, x_2, y_2, \ldots, x_s, y_s, z_1, \ldots, z_t]. \]
If $\gcd(h_i, \ell)=1$, for all $i \in \llbracket 1, s \rrbracket$, then each subalgebra $\K_{q^{h_i}}[x_i, y_i]\subseteq D$ is a quantum affine plane satisfying the conditions of Lemma \ref{lemReponQAP}, hence it must have an irreducible representation $(\varphi_i, V_i)$ of dimension $\ell$. Let $V:=V_1 \otimes \cdots \otimes V_s$ so that the dimension of $V$ is $\ell^s$. Using this notation we state the following proposition:

\begin{proposition}\label{propReponQAS}
Let $q$ be a primitive $\ell^{\mathrm{th}}$ root of unity, for some $\ell >1$, and $M\in M_N(\Z)$ be a skew-symmetric matrix with invariant factors $h_1, \ldots, h_s\in \Z\backslash \{0\}$ satisfying $\gcd(h_i, \ell)=1$, for all $i \in \llbracket 1, s \rrbracket$. The following statements hold:
\begin{enumerate}[(i)]
\item $\PI(\K_{q^M}[T_1, \ldots, T_N])=\ell^s$.
\item There is an algebra homomorphism $\varphi: D \rightarrow \End_\K(V)$ which defines an irreducible representation of $D$ on $V$ of dimension $\ell^s$. It is defined using the tensor product of the maps $\varphi_i$ found in Lemma \ref{lemReponQAP}
 and it acts on the generators of $D$ as follows:
\begin{align*}
\varphi(x_i) &{}= \Id_{V_1} \otimes \cdots \otimes \Id_{V_{i-1}} \otimes \varphi_i(x_i) \otimes \Id_{V_{i+1}} \otimes \cdots \otimes \Id_{V_s}, \\
\varphi(y_i) &{}= \Id_{V_1} \otimes \cdots \otimes \Id_{V_{i-1}} \otimes \varphi_i(y_i) \otimes \Id_{V_{i+1}} \otimes \cdots \otimes \Id_{V_s}, \\
\varphi(z_j) &{}= \xi_j \Id_V,
\end{align*}
where $\xi_j \in \K^*$ for all $j \in \llbracket 1, t \rrbracket$, $\Id_{V_i}$ denotes the identity map on $V_i$ for all $i \in \llbracket 1, s \rrbracket$, and $\Id_V:=\Id_{V_1} \otimes \cdots \otimes \Id_{V_s}$ denotes the identity map on $V$.
Moreover, $\varphi(x_i)^{-1}=\lambda_i^{-\ell} \varphi(x_i)^{\ell-1}$ and  $\varphi(y_i)^{-1}=\varphi(y_i)^{\ell-1}$, for all $i\in \llbracket 1, s \rrbracket$.

\item $(\varphi, V)$ induces an irreducible representation, $(\phi, V)$, of $\K_{q^M}[T_1, \ldots, T_N]$ where, for all $i\in \llbracket 1, N \rrbracket$, we have
\[\phi(T_i):= \varphi(x_1)^{e'_{i, 1}} \varphi(y_1)^{e'_{i,2}} \cdots  \varphi(x_s)^{e'_{i, 2s-1}}  \varphi(y_s)^{e'_{i,2s}} \varphi(z_1)^{e'_{i,2s+1}} \cdots \varphi(z_t)^{e'_{i, 2s+t}},\]
with $E^{-1} := (e'_{i,j})_{i,j} \in M_N(\Z)$. Moreover, for all $i\in \llbracket 1, 2s+t\rrbracket$ there exists some $\nu_i \in \K^*$ such that $\phi(T_i)^{-1}=\nu_i^{-1} \phi(T_i)^{\ell-1}$.
\end{enumerate}
\end{proposition}
\begin{proof}
\
\begin{enumerate}[(i)]
\item Straightforward application of Lemma \ref{lemPIdeg}.

\item It is clear from the matrix $S$ that the elements  $z_1, \ldots, z_t \in Z(D)$ generate a polynomial ring $\K[z_1, \ldots, z_t]\subseteq D$ and that, for each $i \in \llbracket 1, s \rrbracket$, the pair $x_i, \, y_i$ generates a quantum affine plane $\K_{q^{h_i}}[x_i, y_i]\subseteq D$ such that $x_i, \, y_i$ commute with $x_j, \, y_j$, for all $j \in \llbracket 1, s \rrbracket \backslash \{i\}$. The generators of $D$ therefore share the same commutation relations as the generators of $\left(\bigotimes_{i=1}^s \K_{q^{h_i}}[x_i, y_i]\right) \otimes \K[z_1, \ldots, z_t]$. Using the universal property of Ore extensions we deduce that the following map extends to an algebra homomorphism, and that this is in fact an isomorphism:
\begin{align}\label{isomIota}
\iota: D &{} ~ \longrightarrow ~ \K_{q^{h_1}}[x_1, y_1] \otimes \ldots \otimes \K_{q^{h_s}}[x_s, y_s] \otimes \K[z_1, \ldots, z_t]  \\
x_i &{} ~ \longmapsto ~ 1_1 \otimes \cdots \otimes 1_{i-1} \otimes x_i \otimes 1_{i+1} \otimes \cdots \otimes 1_s \otimes 1 \nonumber \\
y_i &{} ~ \longmapsto ~ 1_1 \otimes \cdots \otimes 1_{i-1} \otimes y_i \otimes 1_{i+1} \otimes \cdots \otimes 1_s \otimes 1 \nonumber \\
z_j &{} ~ \longmapsto ~ 1_1 \otimes \cdots \otimes 1_s \otimes z_j, \nonumber
\end{align}
where $1_i$ is the identity element in $\K_{q^{h_i}}[x_i, y_i]$, for all $i\in \llbracket 1, s \rrbracket$. 

That is, for a general element $d = \sum_{\underline{l}\in \N^{2s+t}} \alpha_{\underline{l}} x_1^{l_1} y_1^{l_2} \cdots x_s^{l_{2s-1}} y_s^{l_{2s}} z_1^{l_{2s+1}} \cdots z_t^{l_{2s+t}} \in D$,
\[ \iota(d) =  \sum_{\underline{l}\in \N^{2s+t}} \alpha_{\underline{l}} x_1^{l_1} y_1^{l_2} \otimes \cdots  \otimes x_s^{l_{2s-1}} y_s^{l_{2s}} \otimes  z_1^{l_{2s+1}} \cdots z_t^{l_{2s+t}}. \]

We can define an irreducible representation $(\varphi', V)$ of the tensor product algebra on the right-hand side of (\ref{isomIota}) by applying the Jacobson Density Theorem, which maintains irreducibility, and by noting that the image of $\K[z_1, \ldots, z_t]$ under $\varphi'$ is contained in $\K$ by Schur's Lemma. Here, $\varphi':=\varphi_1 \otimes \cdots \varphi_s$ for the maps $\varphi_i$ as defined in Lemma \ref{lemReponQAP}. We then pull this representation back to an irreducible representation $(\varphi, V)$ of $D$ via the isomorphism above by setting $\varphi:=\varphi' \circ \iota$. This is defined on the generators of $D$ as

\begin{align}\label{eqnVarphiFromD}
\varphi: D &{} ~ \longrightarrow ~ \End_\K(V) \nonumber \\
x_i &{} ~ \longmapsto ~ \Id_{V_1} \otimes \cdots \otimes \Id_{V_{i-1}} \otimes \varphi_i(x_i) \otimes \Id_{V_{i+1}} \otimes \cdots \otimes \Id_{V_s} \nonumber \\
y_i &{} ~ \longmapsto ~ \Id_{V_1} \otimes \cdots \otimes \Id_{V_{i-1}} \otimes \varphi_i(y_i) \otimes \Id_{V_{i+1}} \otimes \cdots \otimes \Id_{V_s} \nonumber \\
z_j &{} \longmapsto ~ \xi_j \Id_V,
\end{align}
where $\Id_{V_i}$ is the identity map on $V_i$ and $\Id_V:=\Id_{V_1} \otimes \cdots \otimes \Id_{V_s}$ is the identity map on $V$.

We may form a basis $\{\mu_{i_1, i_2, \ldots, i_s} \mid 1 \leq i_1, i_2, \ldots, i_s \leq \ell \}$ of $V$ in terms of the basis of each $V_i$ by setting:
\[ \mu_{i_1, i_2, \ldots, i_s} := v_{i_1}^{(1)} \otimes v_{i_2}^{(2)} \otimes \cdots \otimes v_{i_s}^{(s)}.\]
Then, using (\ref{eqnActionXonBasis}) and (\ref{eqnActionYonBasis}), one may verify the following:
\begin{align*}
\varphi(x_j)\cdot \mu_{i_1, \ldots, i_s} &{} = v_{i_1}^{(1)} \otimes \cdots \otimes v_{i_{j-1}}^{(j-1)} \otimes \varphi_j(x_j) \cdot v_{i_j}^{(j)} \otimes v_{i_{j+1}}^{(j+1)} \otimes \cdots \otimes v_{i_s}^{(s)} \\
&{}= v_{i_1}^{(1)} \otimes \cdots \otimes v_{i_{j-1}}^{(j-1)} \otimes \lambda_j q^{(i_j-1) h_j} v_{i_j}^{(j)} \otimes v_{i_{j+1}}^{(j+1)} \otimes \cdots \otimes v_{i_s}^{(s)}\\
&{}= \lambda_j q^{(i_j-1) h_j} \mu_{i_1, \ldots, i_s} \\
\varphi(y_j)\cdot \mu_{i_1,\ldots,  i_s} &{}=   v_{i_1}^{(1)}  \otimes \cdots \otimes v_{i_{j-1}}^{(j-1)} \otimes \varphi_j(y_j)\cdot v_{i_j}^{(j)}\otimes v_{i_{j+1}}^{(j+1)} \otimes \cdots \otimes v_{i_s}^{(s)} \\
&{}= \begin{cases} v_{i_1}^{(1)}  \otimes \cdots \otimes v_{i_{j-1}}^{(j-1)} \otimes v_{1+i_j}^{(j)}\otimes v_{i_{j+1}}^{(j+1)} \otimes \cdots \otimes v_{i_s}^{(s)},  &  \text{ if } i_j\neq \ell; \\
				v_{i_1}^{(1)}  \otimes \cdots \otimes v_{i_{j-1}}^{(j-1)} \otimes  v_{1}^{(j)}\otimes v_{i_{j+1}}^{(j+1)} \otimes \cdots \otimes v_{i_s}^{(s)},  & \text{ if } i_j= \ell,
	\end{cases} \\
&{}=  \begin{cases} 
													\mu_{i_1, \ldots, i_{j-1}, 1+i_j, i_{j+1}, \ldots,  i_s}, & \text{ if } i_j\neq \ell; \\
													\mu_{i_1, \ldots, i_{j-1}, 1, i_{j+1}, \ldots, i_s}, & \text{ if } i_j= \ell, 
												\end{cases} \\
\varphi(z_j) \cdot \mu_{i_1, \ldots, i_s}  &{}= \xi_j \mu_{i_1, \ldots, i_s}.
\end{align*}

Fixing this basis for $V$ allows us to represent the maps $\varphi(x_i), \, \varphi(y_i), \, \varphi(z_j) \in \End_{\K}(V)$ as matrices in $M_{\ell^s}(\K)$ by first defining $\varphi_i(x_i), \, \varphi_i(y_i)$ as the matrices in (\ref{matrices}), and then taking the images in (\ref{eqnVarphiFromD}) to be the Kronecker product of these matrices.

\item 
The algebra homomorphism $\varphi: D \rightarrow \End_\K(V)$ preserves the relations 
\[\varphi(x_i)\varphi(y_i)=q^{h_i} \varphi(y_i) \varphi(x_i),\]
for all $i \in \llbracket 1, s \rrbracket$, and, as was shown in Lemma \ref{lemReponQAP}, $\varphi_i(x_i)^{\ell}=\lambda_i^{\ell} \Id_{V_i}$ and $ \varphi_i(y_i)^{\ell}=\Id_{V_i}$. Hence we obtain 
\begin{equation}\label{eqnPhiX^l}
\begin{rcases}
\varphi(x_i)^{-1} &{}=\lambda_i^{-\ell}\varphi(x_i)^{\ell-1}=\varphi(\lambda_i^{-\ell}x_i^{\ell-1}) \\
\varphi(y_i)^{-1} &{}=\varphi(y_i)^{\ell-1}= \varphi(y_i^{\ell-1}) \\
\varphi(z_j)^{-1} &{}=\xi_j^{-1} \Id_V
\end{rcases}
\in \varphi(D).
\end{equation}
These identities show that $\varphi$ may be extended to define a representation of the quantum torus associated to $D$ and hence, any negative powers of the generators $x_i, \, y_i, \, z_j$ in the argument of $\varphi$ may be replaced with positive powers, upon an appropriate multiplication by a scalar.

From the identity $E^{-1}S(E^{-1})^T=M$ we obtain the equality
\begin{equation} \label{eqnM_ij}
 m_{i,j} = \sum_{k=1}^s h_k(e'_{i, 2k-1} e'_{j, 2k} - e'_{i, 2k} e'_{j, 2k-1}).
\end{equation}
Any algebra homomorphism $\phi: \K_{q^M}[T_1, \ldots, T_N] \rightarrow \End_\K(V)$ must preserve the commutation rules between the $T_i$. That is, for all $i, j \in \llbracket 1, N \rrbracket$,
\[ \phi(T_i)\phi(T_j) = q^{m_{i,j}} \phi(T_j) \phi(T_i) = q^{\sum_{k=1}^s h_k(e'_{i, 2k-1} e'_{j, 2k} - e'_{i, 2k} e'_{j, 2k-1})} \phi(T_j)\phi(T_i). \]
Let $E^{-1}_i$ denote the $i^{\text{th}}$ row of $E^{-1}$, for all $i \in \llbracket 1, N \rrbracket$, and let $\underline{d}:=x_1 y_1 \ldots x_s y_s z_1 \ldots z_t \in D$ be the ordered monomial of all the generators of $D$. We define the following elements of the quantum torus $D \Sigma^{-1}$, where $\Sigma\subset D$ is the multiplicatively closed set generated by the generators of $D$:
\begin{equation}\label{eqnU_i}
\underline{d}^{E^{-1}_i}:=x_1^{e'_{i, 1}} y_1^{e'_{i,2}} \cdots x_s^{e'_{i, 2s-1}}  y_s^{e'_{i,2s}} z_1^{e'_{i,2s+1}} \cdots z_t^{e'_{i,2s+t}} \in D\Sigma^{-1}.
 \end{equation}
We may apply $\varphi$ to $\underline{d}^{E_i^{-1}}$, using (\ref{eqnPhiX^l}) to replace any negative powers $e'_{i,j}$ with positive powers $\ell - e'_{i,j}$ so that $\varphi(\underline{d}^{E_i^{-1}})\in \varphi(D)$, and, using (\ref{eqnM_ij}) and (\ref{eqnU_i}),  it may be verified that
\[ \varphi(\underline{d}^{E^{-1}_i}) \varphi(\underline{d}^{E^{-1}_j}) =  q^{m_{i,j}}  \varphi(\underline{d}^{E^{-1}_j}) \varphi(\underline{d}^{E^{-1}_i}), \]
for all $i, j \in \llbracket 1, N \rrbracket$. Therefore, by the universal property of algebras with generators and relations,
the following map extends to an algebra homomorphism and hence defines a representation of $\K_{q^M}[T_1, \ldots, T_N]$ on $V$:
\begin{align*}
\phi: \K_{q^M}[T_1, \ldots, T_N] & \longrightarrow \End_\K(V) \\
T_i & \longmapsto \varphi(\underline{d}^{E^{-1}_i}).
\end{align*}
Note that $\Ima(\phi) \subseteq \Ima(\varphi)$ and, for all $i\in \llbracket 1, N \rrbracket$,
\[\phi(T_i)^{\ell}=\varphi(\underline{d}^{E^{-1}_i})^{\ell} = \varphi(x_1^{e'_{i, 1}} y_1^{e'_{i,2}} \cdots x_s^{e'_{i, 2s-1}}  y_s^{e'_{i,2s}} z_1^{e'_{i,2s+1}} \cdots z_t^{e'_{i,2s+t}})^{\ell} = \nu_i \Id_V \]
for some $\nu_i \in \K^*$, by properties of $\varphi$. Thus $\phi(T_i)^{-1}= \nu_i^{-1} \phi(T_i)^{\ell - 1}$, which allows us to deal with negative powers of $T_i$ under $\phi$ in a similar manner to how we dealt with negative powers of $x_i, \, y_i, \, z_j$ under $\varphi$.

We now wish to show that $(\phi, V)$ is irreducible. In a similar way to above, we may write each $\varphi(x_i), \, \varphi(y_i), \, \varphi(z_j)$ in terms of the $\phi(T_i)$. The identity $EME^T = S$ allows us to write the entries $s_{i,j}\in S$ in terms of the entries $m_{i,j}\in M$ and $e_{i,j}\in E$ as follows:
\[ s_{i,j} = \sum_{l, k = 1}^N m_{l,k}e_{i,l}e_{j,k}. \]
Letting $E_i$ denote the $i^{\text{th}}$ row of $E$, for all $i \in \llbracket 1, N \rrbracket$, and $\underline{T}:=T_1\cdots T_N$ the ordered monomial of the generators of $\K_{q^M}[T_1, \ldots, T_N]$, we define the following elements of $\K_{q^M}[T_1^{\pm 1}, \ldots, T_N^{ \pm 1}]$:
\begin{align}
\underline{T}^{E_i}:=T_1^{e_{i,1}}\cdots T_N^{e_{i, N}}.
\end{align}
Using this, one may verify, by explicit calculation on $\underline{T}^{E_i}\underline{T}^{E_j}$, that
\[ \phi(\underline{T}^{E_i}) \phi(\underline{T}^{E_j}) = q^{s_{i,j}}  \phi(\underline{T}^{E_j})\phi(\underline{T}^{E_i}), \]
for all $i, j\in \llbracket 1, N \rrbracket$. In particular, the elements
\[\phi(\underline{T}^{E_1}), \; \phi(\underline{T}^{E_2}), \; \ldots, \; \phi(\underline{T}^{E_{2s-1}}), \; \phi(\underline{T}^{E_{2s}}), \; \phi(\underline{T}^{E_{2s+1}}), \; \ldots, \, \phi(\underline{T}^{E_N})\]
share the same relations as the elements $x_1, \, y_1, \, \ldots, \, x_s, \, y_s, \, z_1, \, \ldots, \, z_t$. Therefore, again by the universal property of algebras with generators and relations, the following map extends to an algebra homomorphism:
\begin{align*}
\tau: D &{} \longrightarrow \End_{\K}(V) \\
x_i &{} \longmapsto \phi(\underline{T}^{E_{2i-1}}) \\
y_i &{} \longmapsto \phi(\underline{T}^{E_{2i}}) \\
z_j &{} \longmapsto \phi(\underline{T}^{E_{2s+j}}),
\end{align*}
with $\Ima(\tau)\subseteq \Ima(\phi)$. Since $E^{-1}E=\Id$, where $\Id$ is the identity matrix, we deduce that $\sum_{k=1}^N e'_{k,i} e_{j,k} = 1$ if $i=j$, and otherwise the sum is zero. Using this, along with the definition of $\phi$, we obtain
\begin{align*}
\tau(x_i) &{}= \phi(\underline{T}^{E_{2i-1}})= \phi(T_1^{e_{2i-1,1}}\cdots T_N^{e_{2i-1, N}}) \\
&{}= \varphi(\underline{d}^{E^{-1}_1})^{e_{2i-1,1}} \ldots \varphi(\underline{d}^{E^{-1}_N})^{e_{2i-1,N}} \\
&{} = \varphi(x_1^{e'_{1, 1}} y_1^{e'_{1,2}} \cdots  z_t^{e'_{1,N}})^{e_{2i-1, 1}} \cdots \varphi(x_1^{e'_{i, 1}} y_1^{e'_{i,2}}  \cdots z_t^{e'_{N,N}})^{e_{2i-1, N}} \\
&{}= \varphi(\nu x_1^{\sum_{k=1}^N e'_{k,1}e_{2i-1, k}} y_1^{\sum_{k=1}^N e'_{k,2}e_{2i-1, k}} \cdots z_t^{\sum_{k=1}^N e'_{k,N}e_{2i-1, k}}) \\
&{} = \varphi(\nu x_i),
\end{align*}
where $\nu \in \K^*$ is the scalar resulting from reordering the generators of $D$. Similarly we can show that $\tau(y_i)=\varphi(\nu' y_i)$ and $\tau(z_j)=\varphi(\nu'' z_j)$, for some $\nu', \nu'' \in \K^*$.  From this we deduce that $\Ima(\tau)=\Ima(\varphi)$ and hence $\Ima(\varphi) \subseteq \Ima(\phi)$. It  then follows that $(\phi, V)$ is an irreducible representation of $\K_{q^M}[T_1, \ldots, T_N]$ because $(\varphi, V)$ is irreducible.
\end{enumerate}
\end{proof}

Using the irreducible representation of $A$ above, we construct an irreducible representation of $R_t(M_n)$ of the same dimension.
\begin{proposition}\label{propIrredRepOnQDR}
Take $q$ to be a primitive $\ell^{\mathrm{th}}$ root of unity, with $\ell$ odd.  Let $A:=\K\langle \Omega_{n,n}^t \rangle \subseteq R_t(M_n)$ be the subalgebra generated by the final quantum minors of size less than or equal to $t$, denoted by $T_1, \ldots, T_{2s+t} \in R_t(M_n)$, where $2s=2nt-t^2-t$. Let $(\phi, V)$ be the $\ell^s$-dimensional irreducible representation of $A$ defined in Proposition \ref{propReponQAS}.

Then, every element $r\in R_t(M_n)$ may be written as
\[ r= \sum_{\underline{i} \in \Z^{2s+t}} \alpha_{\underline{i}} T_1^{i_1} \cdots T_{2s+t}^{i_{2s+t}} \in \K_{q^M}[T_1^{\pm 1}, \ldots, T_{2s+t}^{\pm 1}] \]
where $\alpha_{\underline{i}}\in \K$ and $i_j\in \Z$, for all $j \in \llbracket 1, 2s+t \rrbracket$, and there is an algebra homomorphism
\begin{align*}
\rho: R_t(M_n) &{} ~ \longrightarrow ~ \End_{\K}(V) \\
r &{} ~ \longmapsto ~ \sum_{\underline{i}\in \Z^{2s+t}} \alpha_{\underline{i}} \phi(T_1)^{i_1} \cdots \phi(T_{2s+t})^{i_{2s+t}}
\end{align*}
which defines an irreducible representation of $R_t(M_n)$ of dimension $\ell^s$.
\end{proposition}
\begin{proof}
By Lemma \ref{lemAQAS}(iii), $A$ is a quantum affine space so, denoting the final quantum minors in the set $\Omega_{n,n}^t$ by $T_1, \ldots, T_{2s+t}$, this allows us to write $A=\K_{q^{M'}}[T_1, \ldots, T_{2s+t}]$, where $M'\in M_{2nt-t^2}(\Z)$ is a skew-symmetric matrix. We saw in Lemma \ref{lemAQAS}(iv) that the dimension of the kernel of $M'$ is $t$ and that all its invariant factors, $h_1, \ldots, h_s$, are powers of $2$. Since $\ell$ is odd then $\gcd(h_i, \ell)=1$, for all $i \in \llbracket 1, s \rrbracket$, and we may apply Proposition \ref{propReponQAS} to obtain the irreducible representation $(\phi, V)$ of $A$ of dimension $\ell^s$ defined therein.

Recall that $\Sigma \subseteq A$ is the Ore set generated by all $T_i$, and let $\pi: A \rightarrow A \Sigma^{-1}$ be the localisation map which defines $A\Sigma^{-1}$ as a right ring of fractions of $A$. Recall also, from Proposition \ref{propReponQAS}, that $\phi(T_i) \subseteq \End_\K(V)$ is invertible for all $T_i\in \Sigma$. Then, by the universal property of localisations (see \cite[Proposition 10.4]{GoodearlWarfield}), there exists a unique ring homomorphism $\hat{\phi}: A\Sigma^{-1} \rightarrow \End_{\K}(V)$ such that $\phi=\hat{\phi} \circ \pi$. Since $\phi$ is a $\K$-algebra homomorphism then so too are $\hat{\phi}$ and $\pi$. We restrict $\hat{\phi}$ to $R_t(M_n)$ to obtain an algebra homomorphism 
\begin{align*}
\rho: R_t(M_n) &{} ~ \longrightarrow ~ \End_{\K}(V) \\
r &{} ~ \longmapsto ~ \hat{\phi}(r\cdot 1^{-1}),
\end{align*}
which defines a representation $(\rho, V)$ of $R_t(M_n)$. In particular, $\hat{\phi}(T_i^{-1})= \hat{\phi}(T_i)^{-1} = \phi(T_i)^{-1}$, for all $T_i \in \Sigma$, because of the observation in Proposition \ref{propReponQAS}(iii). Thus, writing $r\in R_t(M_n)$ in terms of $T_i^{\pm 1}\in A\Sigma^{-1}$, as is possible by the inclusions shown in Lemma \ref{lemAQAS}(ii), we obtain
\[ \rho(r) = \hat{\phi}(r\cdot 1^{-1}) = \hat{\phi}\left(\sum_{\underline{i}\in \Z^{2s+t}} \alpha_{\underline{i}} T_1^{i_1} \cdots T_{2s+t}^{i_{2s+t}}\right) = \sum_{\underline{i}\in \Z^{2s+t}} \alpha_{\underline{i}} \phi(T_1)^{i_1} \cdots \phi(T_{2s+t})^{i_{2s+t}},\]
where $i_j \in \Z$, for all $j\in \llbracket 1, 2s+t \rrbracket$, and $\alpha_{\underline{i}}\in \K^*$.

To show irreducibility of $(\rho, V)$ we use the fact that $A\subseteq R_t(M_n)$ to deduce $\Ima(\phi) \subseteq \Ima(\rho)$. Therefore, if $(\rho, V)$ were reducible this would force $(\phi, V)$ to be reducible, thus contradicting Proposition \ref{propReponQAS}. Hence $(\rho, V)$ is an irreducible representation of $R_t(M_n)$ of dimension $\ell^s$.
\end{proof} 

\begin{remark}
General techniques to produce irreducible representations of maximal dimension for completely prime quotients of Quantum Nilpotent Algebras at roots of unity were developed in \cite{LauLopRo}. Methods from \cite{LauLopRo} would apply to quantum determinantal rings viewed as completely prime quotients of quantum matrices since quantum matrices are an example of Quantum Nilpotent Algebra at roots of unity. 
\end{remark}

\section{Quantum Schubert Varieties}\label{sectionQSV}

We now return to the quantum Grassmannian and its quantum Schubert varieties $\oh_q(G_{m,n}(\K))_{\gamma} := \oh_q(G_{m,n}(\K))/ \langle \Pi_{m,n}^{\gamma} \rangle$ for $\gamma\in \Pi_{m,n}$. It follows from Proposition \ref{prop-PIdegree-yl} that the PI degree of  $\oh_q(G_{m,n}(\K))_{\gamma} $ is controlled by the matrix  $M_{\gamma} =\left(\begin{array}{cc} M_{\lambda} & {\bf 1} \\{\bf -1} &0 \end{array}\right)$. This leads us to introduce the following notation.

\begin{definition}
Given any square matrix $M\in M_N(\Z)$, we define its \emph{extended matrix} as $M^E:=\left(\begin{smallmatrix} M & \mathbf{-1} \\ \mathbf{1}^T & 0 \end{smallmatrix}\right) \in M_{N+1}(\Z)$, where $\mathbf{-1} \in \Z^N$ is a column vector consisting of $-1$ in each entry and $\mathbf{1}^T\in \Z^N$ is a row vector consisting of $1$ in each entry. 
\end{definition}

Before we start discussing the properties of the extended matrices $M_{\gamma}=M_{\lambda}^E$, we observe that quantum determinantal rings are also related to quantum Schubert varieties. 

\begin{lemma}\label{lemPIdegQSVtoQAS}
Let $q$ be a primitive $\ell^{\mathrm{th}}$ root of unity and $\gamma=\{1, \ldots, t, n+1, \ldots, 2n-t\}$. Then
\[\PI(\oh_q(G_{n,2n}(\K))_{\gamma})=\PI(R_t(M_n)[y; \phi])=\PI\left(\oh_{q^{\M_{n,t}^E}}(\K^{2nt-t^2+1})\right),\]
where $\phi$ is the automorphism on $R_t(M_n)$ defined by $\phi(X_{i,j})=q^{-1}X_{i,j}$ for all generators $X_{i,j}\in R_t(M_n)$ and $\M_{n,t}^E \in M_{2nt-t^2+1}(\Z)$ is the skew-symmetric matrix extended from $\M_{n,t}$.
\end{lemma}
\begin{proof}
We apply the isomorphism in \cite[Proof of Proposition 4.3]{LenaganRigal-SchubertVar}, along with \cite[Remark 4.2(iv)]{LenaganRigal-SchubertVar}, to obtain
\begin{equation}\label{eqnDehomog1}
 R_t(M_n)[y,y^{-1};\phi] \cong \oh_q(G_{n,2n}(\K))_{\gamma}[[\overline{M}]^{-1}],
\end{equation}
where $[M]= [\{n+1, \ldots, 3n\}] \in \oh_q(G_{n,2n}(\K))$ is a normal element and $[\overline{M}]$ is its canonical image in the quotient algebra.

Take $q$ to be a root of unity so that $R_t(M_n)$ is a PI algebra (as a finite module over a central subalgebra $Z_0$ \cite[Corollary 13.1.13]{McConnellRobson}). It is clear from the definition of $\phi$ that $y^{\ell}$ is a central element and hence that the skew-Laurent extension $R_t(M_n)[y, y^{-1}; \phi]$ is finite-dimensional over the central subalgebra $\K\langle Z_0, y^{\ell} \rangle$. Thus it is also a PI algebra. Since the PI degree is invariant under localisation, the isomorphism in (\ref{eqnDehomog1}) allows us to deduce
\[ \PI(\oh_q(G_{n,2n}(\K))_{\gamma})=\PI(R_t(M_n)[y; \phi]).\]

Observing Corollary \ref{corNotation}(1) and extending $\oh_{q^{\M_{n,t}}}(\K^{2nt-t^2})$ by the Ore extension $[y'; \phi']$, where $\phi'(T_{i,j})=q^{-1}T_{i,j}$ for all generators $T_{i,j}\in \oh_{q^{\M_{n,t}}}(\K^{2nt-t^2})$, we deduce that
\[\Frac(R_t(M_n)[y;\phi]) \cong \Frac(\oh_{q^{\M_{n,t}}}(\K^{2nt-t^2})[y';\phi']).\]
Upon noting the commutation rules of $y'$ with generators $T_{i,j}$, as defined by the action of $\phi'$, we may rewrite the Ore extension $\oh_{q^{\M_{n,t}}}(\K^{2nt-t^2})[y';\phi']$ as a quantum affine space in $2nt-t^2+1$ indeterminates whose commutation matrix is $\M_{n,t}^E$. The result of the lemma follows immediately.
\end{proof}

By Proposition \ref{prop-PIdegree-yl}  and Lemma \ref{lemPIdeg}, the key to compute the PI degree of quantum Schubert varieties is to understand the invariant factors and the dimension of the kernel of $M_{\lambda}^E$ and $\M_{n,t}^E$. Since we know the invariant factors for the matrices  $M_{\lambda}$ and $\M_{n,t}$, we devote the next sections to the study of how these invariants change upon extending the matrix.

\subsection{Properties of Arbitrary Extended Matrices}
Let $M\in M_N(\Z)$ be an arbitrary skew-symmetric matrix. Let $\{h_i(M)\}_{i=1}^s$ denote the set of (nonzero) invariant factors of $M$ and $\{h_i(M^E)\}_{i=1}^{s+1}$ denote the set of not necessarily nonzero invariant factors of $M^E$, where $h_i(M^E)\neq 0$ for all $i\in \{1, \ldots, s\}$ and $h_{s+1}(M^E)$ is either 0 or it is the last nonzero invariant factor of $M^E$.
\begin{proposition}\label{propInvFac}
For a skew-symmetric matrix $M\in M_N(\Z)$, 
\[\dim(\ker(M^E))=\dim(\ker(M))+\epsilon ~~ \text{ with } ~~ \epsilon = \begin{cases} 1 & \text{if } \ker(M)\subseteq \orth \\ -1 & \text{otherwise} \end{cases},\]
where $\mathbf{1}^{\bot}$ denotes the orthogonal complement of $\mathbf{1}$ in $\Z^N$.
\end{proposition}
\begin{proof}
Let $W =\{\mathbf{w}\in \Z^N \mid M\mathbf{w}=0 \} = \ker(M)$,  $V=\{\mathbf{v}\in \Z^N \mid M^E\mathbf{v}=0 \}=\ker(M^E)$, and $\orth:=\{\mathbf{v}\in \Z^N \mid \sum_{i=1}^N v_i=0\}$ be the orthogonal complement of $\mathbf{1}$ in $\Z^N$.

First we prove the identity for $\epsilon=1$. Suppose $W\subseteq \orth$ so that $\sum_{i=1}^N w_i=0$ for all $\mathbf{w}\in W$. Introducing the notation $\mathbf{w}[0]:=(w_1 w_2 \ldots w_N 0)^T$ for any $\mathbf{w}\in W$ we see that $\mathbf{w}[0]\in V$ since 
\[ M^E\cdot \mathbf{w}[0] = \left(\begin{matrix} M & \mathbf{-1} \\ \mathbf{1}^T & 0 \end{matrix}\right) \mathbf{w}[0] = \begin{pmatrix} M\mathbf{w} \\ \sum_{i=1}^N w_i \end{pmatrix} = \mathbf{0}. \]
This allows us to define an injective map $f:W \rightarrow V; \; \mathbf{w} \mapsto \mathbf{w}[0]$ which gives the inequality $\dim(W)\leq \dim(V)-1$, since both $M$ and $M^E$ are skew-symmetric so the parity of the dimensions of their kernels must match the parity of the size of the matrix. To get equality here we will show that $\dim(V)\leq \dim(W)+1$. 

Let $V':=\{\mathbf{v}\in V \mid v_{N+1}=0\} \subseteq V$ and let $U:=V\backslash V'$ so that $V=V'\sqcup U$. First we note that $\dim(V')=\dim(W)$ since $V'$ is precisely the image of the injective map $f$ defined above. The inequality before shows that $U\neq \emptyset$ since otherwise $V'=V$ and $\dim(V)=\dim(W)$, thus breaking the inequality $\dim(W)\leq \dim(V)-1$. Therefore there exists some $\mathbf{u}=(u_1, \ldots, u_{n+1}) \in U$  with $u_{N+1} \neq 0$. If $U=\Span\{\mathbf{u}\}$ then we immediately obtain $\dim(V)=\dim(W)+1$. Suppose instead that there exists some other $\mathbf{u'}=(u'_1, \ldots, u'_{N+1})\in U$ with  $u'_{N+1}\neq 0$ and $\mathbf{u'}\neq \mathbf{u}$. With this we see that
\[ u_{N+1}\mathbf{u'} - u'_{N+1}\mathbf{u} = \begin{pmatrix} u_{N+1}u'_1 - u'_{N+1} u_1 \\
\vdots \\
u_{N+1} u'_N - u'_{N+1}u_N \\
u_{N+1} u'_{N+1} - u'_{N+1}u_{N+1} 
\end{pmatrix} =  \begin{pmatrix} u_{N+1}u'_1 - u'_{N+1} u_1 \\
\vdots \\ 
u_{N+1} u'_N - u'_{N+1}u_N \\
0
\end{pmatrix} \in V'.\]
Denoting the basis vectors of $V'$ as $\{\mathbf{b_i}\}_i \subset V'$ we write the following:
\[u_{N+1}\mathbf{u'} - u'_{N+1}\mathbf{u} = \sum_i \lambda_i \mathbf{b_i} \]
for some $\lambda_i \in \Z$. When rearranged, this equation expresses $u_{N+1} \mathbf{u'}$ in terms of these basis vectors as well as $u'_{N+1}\mathbf{u}$. Since $\mathbf{u'}$ was arbitrary, we conclude that the set $\{\mathbf{b_i}\}_i$ does not span $V$ but that $\{\mathbf{b_i}\}_i \cup \{\mathbf{u}\} $ is a linearly independent spanning set for  $V$. Since $\{\mathbf{b_i}\}_i$ is a basis for $V'$ and $\dim(V')=\dim(W)$ we conclude that $\dim(V) \leq \dim(W) + 1$ as desired. This proves the statement of the lemma for $\epsilon = 1$.

To prove the statement for $\epsilon=-1$ we suppose $W \nsubseteq \orth$ and let $W'\subseteq W$ be the set of all $\mathbf{w}\in W$ which are orthogonal to $\mathbf{1}$, i.e. $W'=\{w\in W \mid \sum_{i=1}^N w_i = 0\}$. Then by the injective map $f|_{W'}: W' \rightarrow V$ given by $\mathbf{w'} \mapsto \mathbf{w'}[0]$  we obtain the inequality $\dim(W')=\dim(W)-1 \leq \dim(V)$. Using parity arguments, this means that either $\dim(V)=\dim(W)-1$ or $\dim(V)>\dim(W)+1$. Suppose $\dim(V)>\dim(W)+1$ and and consider $V'\subseteq V$ as before so that $\dim(V')=\dim(V)-1$. The injective map $V' \rightarrow W; \; \mathbf{v'} = (v_1 \ldots v_N 0) \mapsto = (v_1 \ldots v_N)$ and our supposition implies 
\[ \dim(W) \geq \dim(V') =\dim(V)-1 \geq \dim(W) \implies \dim(V')=\dim(W) \]
but then the injective map becomes bijective and for each $\mathbf{w}\in W$ we have some $\mathbf{w}[0]\in V'$ meaning that $\mathbf{w} \in \orth$ and $W\subseteq \orth$. This contradicts our initial assumption, therefore we conclude $\dim(V)=\dim(W)-1$.
\end{proof}

The proof above used arguments from module theory that relied on $\ker(M)$ being a free $\Z$-module with well-defined rank or dimension. When working over the finite field $\Z/p\Z$, for prime $p$, the module arguments translate directly to vector space arguments and a similar proof can be used to show:
\begin{corollary}\label{corolDim}
Fix some prime number $p$. Let $\overline{M}:=M \mod{p} \in M_N(\Z/p\Z)$ and $\overline{M^E} =\left(\begin{smallmatrix} \overline{M} & -\bar{\mathbf{1}} \\ \bar{\mathbf{1}}^T & 0 \end{smallmatrix}\right) \in M_{N+1}(\Z/p\Z)$ be its extended matrix. Then 
\[\dim(\ker(\overline{M^E}))=\dim(\ker(\overline{M}))+\epsilon, ~ \text{ where } ~\epsilon = \begin{cases} 1 & \text{if } \ker(\overline{M})\subseteq \bar{\mathbf{1}}^{\bot}; \\ -1 & \text{otherwise,} \end{cases}\]
where $\bar{\mathbf{1}}^{\bot}$ denotes the orthogonal complement of $\bar{\mathbf{1}}$ in $(\Z/p\Z)^N$.
\end{corollary}

\begin{proposition}\label{propKerMod}
Fix a prime number $p$ and let $\overline{M}:=M \mod{p} \in M_N(\Z/p\Z)$. Then 
\[ \ker(\overline{M})\subseteq \bar{\mathbf{1}}^{\bot} \iff  p\mid h_{s'+1}(M^E), \]
where $s'$ is the number of nonzero invariant factors of $\overline{M}$, with $s' \leq s$, and $\overline{\mathbf{1}}^{\bot}$ denotes the orthogonal complement of $\mathbf{1}$ in $(\Z/p\Z)^N$.
\end{proposition}
\begin{proof}
From Corollary \ref{corolDim} we see that $\ker(\overline{M})\subseteq \overline{\mathbf{1}}^{\bot}$ if and only if $\dim(\ker(\overline{M^E}))=\dim(\ker(\overline{M}))+1$. Applying \cite[Theorem 3]{Thompson} to $\overline{M}$ and $\overline{M^E}$ to get
\begin{equation} \label{eqnThompson}
h_i(\overline{M^E})|h_i(\overline{M})|h_{i+1}(\overline{M^E})  \quad \forall \; i\in [1, s']
\end{equation}
we see that $\overline{M^E}$ has at least as many nonzero invariant factors as $\overline{M}$. Hence the dimension of the kernel increases in the extended matrix if and only if $h_{s'+1}(\overline{M^E})=\overline{0}$. That is, if and only if $p \mid h_{s'+1}(M^E)$.
\end{proof}

\subsection{Properties of Extended Matrices Coming From (Cauchon-Le) Diagrams}

Our results in this section use notation and results proved in \cite{BellCasteelsLaunois}. In particular, for a permutation $\sigma \in S_k$, we denote by $P_{\sigma}$ the corresponding $k \times k$ permutation matrix defined by $P_{\sigma}[i,j]:=\delta_{j,\sigma (i)}$, where $\delta$ denotes the Kronecker symbol. 

We now take $M=M(D) \in M_{N}(\Z)$ to be the matrix corresponding to an $m\times n$ diagram $D$ with $N$ white squares labelled $1, \ldots, N$. Let $\tau=\sigma \omega^{-1}$ be the toric permutation associated to $D$, where $\sigma$ is the restricted permutation associated to $D$ and $\omega$ is the restricted permutation associated to the $m\times n$ diagram consisting only of black squares. As noted in the proof of \cite[Lemma 4.3]{BellCasteelsLaunois}, the odd cycles $\{\eta_1, \ldots, \eta_r\}$ in $\tau$ give a basis $B=\{ \mathbf{v}^{\eta_i} \mid i \in \llbracket 1, r \rrbracket\}$ for $\ker(P_{\omega}+P_{\sigma})$. Recall that $\ker(P_{\omega}+P_{\sigma}) \cong \ker(M)$ via injective functions $\phi: \ker(P_{\omega}+P_{\sigma}) \rightarrow \ker(M(D))$ and $\phi: \ker(M(D)) \rightarrow \ker(P_{\omega}+P_{\sigma})$ as defined in \cite[Theorem 4.6]{BellCasteelsLaunois}. This allows us to transfer the basis $B$ to a basis $\phi(B)$ of $\ker(M)$.

\begin{lemma}\label{lemSumw}
For any basis element $\mathbf{w}:=\phi(\mathbf{v})\in \ker(M)$ corresponding to basis element $\mathbf{v}:=\mathbf{v}^{\eta}\in \ker(P_{\omega}+P_{\sigma})$ and odd cycle $\eta$ in $\tau$, we have that 
\begin{equation}\label{eqnSumw_i}
\sum_{i=1}^N w_i = \sum_{\substack{ \{j\in \llbracket m+1, m+n \rrbracket |  \\  \eta(j)\in \llbracket 1, m \rrbracket \} }} v_{\eta(j)} ~~ - \sum_{\substack{ \{k\in \llbracket 1, m \rrbracket | \\ \eta(k)\in \llbracket m+1, m+n \rrbracket\} }} v_{\eta(k)}.
\end{equation}

Writing $\eta=(R_1, \, C_1, \, R_2, \, C_2, \, \ldots \, R_k, \, C_k)$ as contiguous sequences which alternate between increasing subsequences of row-labels (the $R_i$, which form a partition of $\llbracket 1, m \rrbracket$) and decreasing subsequences of column-labels (the $C_i$, which form a partition of $\llbracket m+1,  m+n\rrbracket$), and letting $v_{1'}$ be the first nonzero entry in $\mathbf{v}$, we can rewrite the above sum as
\[\sum_{i=1}^N w_i = \left(-(-1)^{|R_1|} +(-1)^{|R_1|+|C_1|} - \cdots + (-1)^{|R_1|+ |C_1|+\cdots + |R_{k}|+|C_k|} \right) v_{1'}.  \]
\end{lemma}
\begin{proof}
From \cite[Lemmas 4.2 \& 4.3]{BellCasteelsLaunois}, $\mathbf{v} \in \ker(P_{\omega}+P_{\sigma})$ is a vector with entries in $\{-1, 0, 1\}$ satisfying $v_i=(-1)^k v_{\eta^k(i)}$ for all $k\in \N$ and $i\in \llbracket 1, m+n \rrbracket$. The element $\mathbf{w}\in \ker(M)$ is defined in \cite[Theorem 4.6]{BellCasteelsLaunois} to be the vector with entries $w_i= v_{\text{left}(i)} - v_{\text{up}(i)}$, for all $i\in \llbracket 1, N \rrbracket$, where $\text{left}(i)$ and $\text{up}(i)$ is the row- or column-label reached by following the pipe exiting white square $i$ to the left and at the top, respectively. Figure \ref{illustration from v to w} illustrates this construction in the case where $D$ is the $3 \times 5$ diagram from Figure \ref{FigToricPerm}.  

\definecolor{light-gray}{gray}{0.6}
\begin{figure}[h]
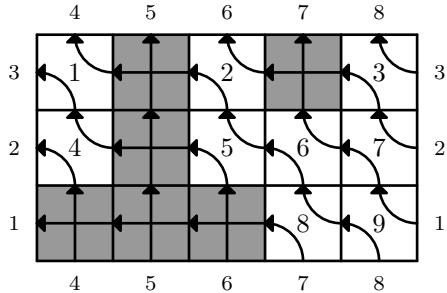

\begin{center}
\begin{pgfpicture}{0cm}{0cm}{6cm}{4cm}%

\pgfsetroundjoin \pgfsetroundcap%
\pgfsetfillcolor{light-gray}
\pgfmoveto{\pgfxy(0.5,0.5)}\pgflineto{\pgfxy(0.5,1.5)}\pgflineto{\pgfxy(3.5,1.5)}\pgflineto{\pgfxy(3.5,0.5)}\pgflineto{\pgfxy(0.5,0.5)}\pgffill
\pgfmoveto{\pgfxy(1.5,1.5)}\pgflineto{\pgfxy(1.5,3.5)}\pgflineto{\pgfxy(2.5,3.5)}\pgflineto{\pgfxy(2.5,1.5)}\pgflineto{\pgfxy(1.5,1.5)}\pgffill
\pgfmoveto{\pgfxy(3.5,2.5)}\pgflineto{\pgfxy(3.5,3.5)}\pgflineto{\pgfxy(4.5,3.5)}\pgflineto{\pgfxy(4.5,2.5)}\pgflineto{\pgfxy(3.5,2.5)}\pgffill
\pgfsetlinewidth{1pt}
\pgfxyline(0.5,0.5)(4.5,0.5)
\pgfxyline(0.5,1.5)(4.5,1.5)
\pgfxyline(0.5,2.5)(4.5,2.5)
\pgfxyline(0.5,3.5)(4.5,3.5)
\pgfxyline(0.5,0.5)(0.5,3.5)
\pgfxyline(1.5,0.5)(1.5,3.5)
\pgfxyline(2.5,0.5)(2.5,3.5)
\pgfxyline(3.5,0.5)(3.5,3.5)
\pgfxyline(4.5,0.5)(4.5,3.5)
\pgfxyline(5.5,0.5)(5.5,3.5)
\pgfxyline(4.5,0.5)(5.5,0.5)
\pgfxyline(4.5,1.5)(5.5,1.5)
\pgfxyline(4.5,2.5)(5.5,2.5)
\pgfxyline(4.5,3.5)(5.5,3.5)

\pgfputat{\pgfxy(1,3)}{\pgfnode{rectangle}{center}{\color{black} $1$}{}{\pgfusepath{}}}
\pgfputat{\pgfxy(3,3)}{\pgfnode{rectangle}{center}{\color{black} $2$}{}{\pgfusepath{}}}
\pgfputat{\pgfxy(5,3)}{\pgfnode{rectangle}{center}{\color{black} $3$}{}{\pgfusepath{}}}
\pgfputat{\pgfxy(1,2)}{\pgfnode{rectangle}{center}{\color{black} $4$}{}{\pgfusepath{}}}
\pgfputat{\pgfxy(3,2)}{\pgfnode{rectangle}{center}{\color{black} $5$}{}{\pgfusepath{}}}
\pgfputat{\pgfxy(4,2)}{\pgfnode{rectangle}{center}{\color{black} $6$}{}{\pgfusepath{}}}
\pgfputat{\pgfxy(5,2)}{\pgfnode{rectangle}{center}{\color{black} $7$}{}{\pgfusepath{}}}
\pgfputat{\pgfxy(4,1)}{\pgfnode{rectangle}{center}{\color{black} $8$}{}{\pgfusepath{}}}
\pgfputat{\pgfxy(5,1)}{\pgfnode{rectangle}{center}{\color{black} $9$}{}{\pgfusepath{}}}

\pgfsetlinewidth{1pt}
\color{black}

\pgfmoveto{\pgfxy(0.5,3)}\pgfpatharc{90}{0}{0.5cm}
\pgfstroke
\pgfmoveto{\pgfxy(0.5,3)}\pgflineto{\pgfxy(0.6,3.1)}\pgflineto{\pgfxy(0.6,2.9)}\pgflineto{\pgfxy(0.5,3)}\pgfclosepath\pgffillstroke

\pgfmoveto{\pgfxy(1,3.5)}\pgfpatharc{180}{270}{0.5cm}
\pgfstroke
\pgfmoveto{\pgfxy(1,3.5)}\pgflineto{\pgfxy(0.9,3.4)}\pgflineto{\pgfxy(1.1,3.4)}\pgflineto{\pgfxy(1,3.5)}\pgfclosepath\pgffillstroke

\color{black}
\pgfxyline(2,2.5)(2,3.5)
\pgfmoveto{\pgfxy(1.9,3.4)}\pgflineto{\pgfxy(2,3.5)}\pgflineto{\pgfxy(2.1,3.4)}\pgflineto{\pgfxy(1.9,3.4)}\pgfclosepath\pgffillstroke

\pgfxyline(1.5,3)(2.5,3)
\pgfmoveto{\pgfxy(1.5,3)}\pgflineto{\pgfxy(1.6,3.1)}\pgflineto{\pgfxy(1.6,2.9)}\pgflineto{\pgfxy(1.5,3)}\pgfclosepath\pgffillstroke

\pgfmoveto{\pgfxy(2.5,3)}\pgfpatharc{90}{0}{0.5cm}
\pgfstroke
\pgfmoveto{\pgfxy(2.5,3)}\pgflineto{\pgfxy(2.6,3.1)}\pgflineto{\pgfxy(2.6,2.9)}\pgflineto{\pgfxy(2.5,3)}\pgfclosepath\pgffillstroke

\pgfmoveto{\pgfxy(3,3.5)}\pgfpatharc{180}{270}{0.5cm}
\pgfstroke
\pgfmoveto{\pgfxy(3,3.5)}\pgflineto{\pgfxy(2.9,3.4)}\pgflineto{\pgfxy(3.1,3.4)}\pgflineto{\pgfxy(3,3.5)}\pgfclosepath\pgffillstroke


\pgfxyline(4,2.5)(4,3.5)
\pgfmoveto{\pgfxy(3.9,3.4)}\pgflineto{\pgfxy(4,3.5)}\pgflineto{\pgfxy(4.1,3.4)}\pgflineto{\pgfxy(3.9,3.4)}\pgfclosepath\pgffillstroke

\pgfxyline(3.5,3)(4.5,3)
\pgfmoveto{\pgfxy(3.5,3)}\pgflineto{\pgfxy(3.6,3.1)}\pgflineto{\pgfxy(3.6,2.9)}\pgflineto{\pgfxy(3.5,3)}\pgfclosepath\pgffillstroke


\pgfmoveto{\pgfxy(0.5,2)}\pgfpatharc{90}{0}{0.5cm}
\pgfstroke
\pgfmoveto{\pgfxy(0.5,2)}\pgflineto{\pgfxy(0.6,2.1)}\pgflineto{\pgfxy(0.6,1.9)}\pgflineto{\pgfxy(0.5,2)}\pgfclosepath\pgffillstroke

\pgfmoveto{\pgfxy(1,2.5)}\pgfpatharc{180}{270}{0.5cm}
\pgfstroke
\pgfmoveto{\pgfxy(1,2.5)}\pgflineto{\pgfxy(0.9,2.4)}\pgflineto{\pgfxy(1.1,2.4)}\pgflineto{\pgfxy(1,2.5)}\pgfclosepath\pgffillstroke


\pgfxyline(2,1.5)(2,2.5)
\pgfmoveto{\pgfxy(1.9,2.4)}\pgflineto{\pgfxy(2,2.5)}\pgflineto{\pgfxy(2.1,2.4)}\pgflineto{\pgfxy(1.9,2.4)}\pgfclosepath\pgffillstroke

\pgfxyline(1.5,2)(2.5,2)
\pgfmoveto{\pgfxy(1.5,2)}\pgflineto{\pgfxy(1.6,2.1)}\pgflineto{\pgfxy(1.6,1.9)}\pgflineto{\pgfxy(1.5,2)}\pgfclosepath\pgffillstroke

\pgfmoveto{\pgfxy(2.5,2)}\pgfpatharc{90}{0}{0.5cm}
\pgfstroke
\pgfmoveto{\pgfxy(2.5,2)}\pgflineto{\pgfxy(2.6,2.1)}\pgflineto{\pgfxy(2.6,1.9)}\pgflineto{\pgfxy(2.5,2)}\pgfclosepath\pgffillstroke

\pgfmoveto{\pgfxy(3,2.5)}\pgfpatharc{180}{270}{0.5cm}
\pgfstroke
\pgfmoveto{\pgfxy(3,2.5)}\pgflineto{\pgfxy(2.9,2.4)}\pgflineto{\pgfxy(3.1,2.4)}\pgflineto{\pgfxy(3,2.5)}\pgfclosepath\pgffillstroke


\pgfmoveto{\pgfxy(3.5,2)}\pgfpatharc{90}{0}{0.5cm}
\pgfstroke
\pgfmoveto{\pgfxy(3.5,2)}\pgflineto{\pgfxy(3.6,2.1)}\pgflineto{\pgfxy(3.6,1.9)}\pgflineto{\pgfxy(3.5,2)}\pgfclosepath\pgffillstroke

\pgfmoveto{\pgfxy(4,2.5)}\pgfpatharc{180}{270}{0.5cm}
\pgfstroke
\pgfmoveto{\pgfxy(4,2.5)}\pgflineto{\pgfxy(3.9,2.4)}\pgflineto{\pgfxy(4.1,2.4)}\pgflineto{\pgfxy(4,2.5)}\pgfclosepath\pgffillstroke


\pgfxyline(1,0.5)(1,1.5)
\pgfmoveto{\pgfxy(0.9,1.4)}\pgflineto{\pgfxy(1,1.5)}\pgflineto{\pgfxy(1.1,1.4)}\pgflineto{\pgfxy(0.9,1.4)}\pgfclosepath\pgffillstroke

\pgfxyline(0.5,1)(1.5,1)
\pgfmoveto{\pgfxy(0.5,1)}\pgflineto{\pgfxy(0.6,1.1)}\pgflineto{\pgfxy(0.6,0.9)}\pgflineto{\pgfxy(0.5,1)}\pgfclosepath\pgffillstroke


\pgfxyline(2,0.5)(2,1.5)
\pgfmoveto{\pgfxy(1.9,1.4)}\pgflineto{\pgfxy(2,1.5)}\pgflineto{\pgfxy(2.1,1.4)}\pgflineto{\pgfxy(1.9,1.4)}\pgfclosepath\pgffillstroke

\pgfxyline(1.5,1)(2.5,1)
\pgfmoveto{\pgfxy(1.5,1)}\pgflineto{\pgfxy(1.6,1.1)}\pgflineto{\pgfxy(1.6,0.9)}\pgflineto{\pgfxy(1.5,1)}\pgfclosepath\pgffillstroke


\pgfxyline(3,0.5)(3,1.5)
\pgfmoveto{\pgfxy(2.9,1.4)}\pgflineto{\pgfxy(3,1.5)}\pgflineto{\pgfxy(3.1,1.4)}\pgflineto{\pgfxy(2.9,1.4)}\pgfclosepath\pgffillstroke

\pgfxyline(2.5,1)(3.5,1)
\pgfmoveto{\pgfxy(2.5,1)}\pgflineto{\pgfxy(2.6,1.1)}\pgflineto{\pgfxy(2.6,0.9)}\pgflineto{\pgfxy(2.5,1)}\pgfclosepath\pgffillstroke

\pgfmoveto{\pgfxy(3.5,1)}\pgfpatharc{90}{0}{0.5cm}
\pgfstroke
\pgfmoveto{\pgfxy(3.5,1)}\pgflineto{\pgfxy(3.6,1.1)}\pgflineto{\pgfxy(3.6,0.9)}\pgflineto{\pgfxy(3.5,1)}\pgfclosepath\pgffillstroke

\pgfmoveto{\pgfxy(4,1.5)}\pgfpatharc{180}{270}{0.5cm}
\pgfstroke
\pgfmoveto{\pgfxy(4,1.5)}\pgflineto{\pgfxy(3.9,1.4)}\pgflineto{\pgfxy(4.1,1.4)}\pgflineto{\pgfxy(4,1.5)}\pgfclosepath\pgffillstroke

\pgfmoveto{\pgfxy(4.5,3)}\pgfpatharc{90}{0}{0.5cm}
\pgfstroke
\pgfmoveto{\pgfxy(4.5,3)}\pgflineto{\pgfxy(4.6,3.1)}\pgflineto{\pgfxy(4.6,2.9)}\pgflineto{\pgfxy(4.5,3)}\pgfclosepath\pgffillstroke

\pgfmoveto{\pgfxy(5,3.5)}\pgfpatharc{180}{270}{0.5cm}
\pgfstroke
\pgfmoveto{\pgfxy(5,3.5)}\pgflineto{\pgfxy(4.9,3.4)}\pgflineto{\pgfxy(5.1,3.4)}\pgflineto{\pgfxy(5,3.5)}\pgfclosepath\pgffillstroke


\pgfmoveto{\pgfxy(4.5,2)}\pgfpatharc{90}{0}{0.5cm}
\pgfstroke
\pgfmoveto{\pgfxy(4.5,2)}\pgflineto{\pgfxy(4.6,2.1)}\pgflineto{\pgfxy(4.6,1.9)}\pgflineto{\pgfxy(4.5,2)}\pgfclosepath\pgffillstroke

\pgfmoveto{\pgfxy(5,2.5)}\pgfpatharc{180}{270}{0.5cm}
\pgfstroke
\pgfmoveto{\pgfxy(5,2.5)}\pgflineto{\pgfxy(4.9,2.4)}\pgflineto{\pgfxy(5.1,2.4)}\pgflineto{\pgfxy(5,2.5)}\pgfclosepath\pgffillstroke

\pgfmoveto{\pgfxy(4.5,1)}\pgfpatharc{90}{0}{0.5cm}
\pgfstroke
\pgfmoveto{\pgfxy(4.5,1)}\pgflineto{\pgfxy(4.6,1.1)}\pgflineto{\pgfxy(4.6,0.9)}\pgflineto{\pgfxy(4.5,1)}\pgfclosepath\pgffillstroke

\pgfmoveto{\pgfxy(5,1.5)}\pgfpatharc{180}{270}{0.5cm}
\pgfstroke
\pgfmoveto{\pgfxy(5,1.5)}\pgflineto{\pgfxy(4.9,1.4)}\pgflineto{\pgfxy(5.1,1.4)}\pgflineto{\pgfxy(5,1.5)}\pgfclosepath\pgffillstroke

\pgfputat{\pgfxy(0.2,1)}{\pgfnode{rectangle}{center}{\color{black}\footnotesize $1$}{}{\pgfusepath{}}}
\pgfputat{\pgfxy(0.2,2)}{\pgfnode{rectangle}{center}{\color{black}\footnotesize $2$}{}{\pgfusepath{}}}
\pgfputat{\pgfxy(0.2,3)}{\pgfnode{rectangle}{center}{\color{black}\footnotesize $3$}{}{\pgfusepath{}}}

\pgfputat{\pgfxy(5.8,1)}{\pgfnode{rectangle}{center}{\color{black}\footnotesize $1$}{}{\pgfusepath{}}}
\pgfputat{\pgfxy(5.8,2)}{\pgfnode{rectangle}{center}{\color{black}\footnotesize $2$}{}{\pgfusepath{}}}
\pgfputat{\pgfxy(5.8,3)}{\pgfnode{rectangle}{center}{\color{black}\footnotesize $3$}{}{\pgfusepath{}}}

\pgfputat{\pgfxy(1,0.2)}{\pgfnode{rectangle}{center}{\color{black}\footnotesize $4$}{}{\pgfusepath{}}}
\pgfputat{\pgfxy(2,0.2)}{\pgfnode{rectangle}{center}{\color{black}\footnotesize $5$}{}{\pgfusepath{}}}
\pgfputat{\pgfxy(3,0.2)}{\pgfnode{rectangle}{center}{\color{black}\footnotesize $6$}{}{\pgfusepath{}}}
\pgfputat{\pgfxy(4,0.2)}{\pgfnode{rectangle}{center}{\color{black}\footnotesize $7$}{}{\pgfusepath{}}}
\pgfputat{\pgfxy(5,0.2)}{\pgfnode{rectangle}{center}{\color{black}\footnotesize $8$}{}{\pgfusepath{}}}

\pgfputat{\pgfxy(1,3.8)}{\pgfnode{rectangle}{center}{\color{black}\footnotesize $4$}{}{\pgfusepath{}}}
\pgfputat{\pgfxy(2,3.8)}{\pgfnode{rectangle}{center}{\color{black}\footnotesize $5$}{}{\pgfusepath{}}}
\pgfputat{\pgfxy(3,3.8)}{\pgfnode{rectangle}{center}{\color{black}\footnotesize $6$}{}{\pgfusepath{}}}
\pgfputat{\pgfxy(4,3.8)}{\pgfnode{rectangle}{center}{\color{black}\footnotesize $7$}{}{\pgfusepath{}}}
\pgfputat{\pgfxy(5,3.8)}{\pgfnode{rectangle}{center}{\color{black}\footnotesize $8$}{}{\pgfusepath{}}}
\end{pgfpicture}
\caption{In this example, $\text{left}(7)=7$ and $\text{up}(7)=6$, so that $w_7= v_{7} - v_{6}$. \label{illustration from v to w}}
\end{center}
\end{figure}

From this we obtain
\begin{equation}\label{eqnSum}
\sum_{i=1}^N w_i = \sum_{i=1}^N  v_{\text{left}(i)} - v_{\text{up}(i)}.
\end{equation}
Notice that we can split this sum up into the contributions coming from paths $j \rightarrow \tau(j)$ in $D$, for all $j\in \llbracket 1, m+n \rrbracket$, and if $j\notin \eta$ then $v_j=v_{\tau(j)}=0$. Therefore we only need to consider contributions from paths $j \rightarrow \eta(j)$ for $j\in \eta$. Label the white squares that the path $j \rightarrow \eta(j)$ passes through as $i_{j_1}, \ldots, i_{j_s}$. For each such path we sum the contributions coming from the white squares $i_{j_1}, \ldots, i_{j_s}$. If $j$ and $\eta(j)$ are both column-labels then the pipe must exit square $i_{j_1}$ to the left and traverse an even number of white squares before exiting square $i_{j_s}$ at the top.  Thus the contribution of this path to the sum (\ref{eqnSum}) is:
\[ v_{\text{left}(i_{j_1})} - v_{\text{up}(i_{j_2})} + \ldots + v_{\text{left}(i_{j_{s-1}})}-v_{\text{up}(i_{j_s})} = v_{\eta(j)} - v_{\eta(j)} + \ldots +v_{\eta(j)}-v_{\eta(j)} = 0.\]
Similarly, if $j$ and $\eta(j)$ are both row-labels then the contribution of the path to the sum (\ref{eqnSum}) is also $0$. If, however, $j$ is a row-label and $\eta(j)$ is a column-label then $s$ is odd and the contribution becomes:
\[  - v_{\text{up}(i_{j_1})} +  v_{\text{left}(i_{j_2})} - \ldots + v_{\text{left}(i_{j_{s-1}})}-v_{\text{up}(i_{j_s})} = - v_{\eta(j)} + v_{\eta(j)} - \ldots + v_{\eta(j)}-v_{\eta(j)} = -v_{\eta(j)} .\]
With similar calculations we see that if $j$ is a column-label and $\eta(j)$ is a row-label then the contribution of the path is $v_{\eta(j)}$. 
Hence the only contributions to the sum in (\ref{eqnSum}) come from the paths $j \rightarrow \eta(j)$ going from a row-label to a column-label, or vice-versa, and these contributions are $-v_{\eta(j)}$ when $j$ is a row-label and $v_{\eta(j)}$  when $j$ is a column-label. This proves identity (\ref{eqnSumw_i}).

We now group the consecutive entries of $\eta$ into contiguous subsequences as discussed in \cite[Section 5]{BellCasteelsLaunois} and write the permutation as $\eta=(R_1, \, C_1, \, R_2, \, C_2, \, \ldots \, R_k, \, C_k)$. If we let $v_{1'}$ be the first nonzero entry of $\mathbf{v}$ then these subsequences become: 
\[R_1 = (v_{1'}, \, v_{\eta(1')}, \, \ldots, \, v_{\eta^{|R_1|-1}(1')}),  \quad C_1 = (v_{\eta^{|R_1|}(1')}, \, \ldots, \, v_{\eta^{|R_1|+|C_1|-1}(1')}) \]
and, in general, for any $i\in \llbracket 1, k-1 \rrbracket$:
\begin{align*}
 R_{i+1} = ( v_{\eta^{|R_1|+|C_1|+\cdots +|R_i|+|C_i|}(1')}, \, \ldots, \, v_{\eta^{|R_1|+|C_1|+\cdots +|R_i|+|C_i|+|R_{i+1}|-1}(1')}), \\
 C_{i+1} = (v_{\eta^{|R_1|+|C_1|+\cdots +|C_i|+|R_{i+1}|}(1')}, \, \ldots, \, v_{\eta^{|R_1|+|C_1|+\cdots +|R_{i+1}|+|C_{i+1}|-1}(1')}).
\end{align*}
The sums on the right-hand side of equation (\ref{eqnSumw_i}) then correspond to summing the first entries in $R_1, \ldots, R_k$ and the first entries in $C_1, \ldots, C_k$ respectively. That is, we may write these sums as
\begin{align*}
\sum_{\substack{ \{j\in \llbracket m+1, m+n \rrbracket | \\ \eta(j)\in \llbracket 1, m \rrbracket \} }} v_{\eta(j)} & = v_{1'}+\sum_{i=1}^{k-1} v_{\eta^{|R_1|+|C_1|+\cdots +|R_i|+|C_i|}(1')} \\
 \sum_{\substack{ \{ k\in \llbracket 1, m \rrbracket |\\ \eta(k)\in \llbracket m+1, m+n \rrbracket \} }} v_{\eta(k)} & = v_{\eta^{|R_1|}(1')} + \sum_{i=1}^{k-1} v_{\eta^{|R_1|+|C_1|+\cdots +|C_{i}|+|R_{i+1}|}(1')}.
 \end{align*}
Upon applying the identity $(-1)^{i}v_{1'}=v_{\eta^i(1')}$, for all $i\in \N$, we can rewrite the sums above in terms of $v_{1'}$ only, thus completing the proof.
\end{proof}

\begin{proposition}\label{propSomething}
Take $D$ to be an $m\times n$ diagram. $M:=M(D)$, and let $\{h_i(M)\}_{i=1}^s$ be the nonzero invariant factors of the associated matrix.
Then $h_i(M^E)$ is a power of 2, for all $i\in \llbracket 1,s \rrbracket$, and $h_{s+1}(M^E)$ is  either 0 or its odd prime factors are bounded above by $\min(m,n)$.
\end{proposition}
\begin{proof}
Let $M=M(D)$ be the matrix corresponding to an $m\times n$ diagram. By \cite[Theorem 3]{Thompson} we obtain the following interlacing inequalities:
\begin{equation}\label{Thompson}
h_1(M^E) | h_1(M), ~ h_2(M^E) | h_2(M), ~ \ldots,  h_s(M^E) | h_s(M),
\end{equation}
where the $|$ means `divides'. From this we see that $h_i(M^E)$ are all powers of 2 for $i\in \llbracket 1, s \rrbracket$, but not necessarily for $i=s+1$.

Suppose that $h_{s+1}(M^E)\neq 0$ and that $p|h_{s+1}(M^E)$ for some odd prime $p>2$. In this case, by Proposition \ref{propInvFac} we deduce that $\ker(M)\nsubseteq \orth$ and $\ker(M \mod p) \subseteq \orth \mod p$. To see this we note that when $h_{s+1}(M^E)\neq 0$ then $\dim(\ker(M^E))=\dim(\ker(M))-1$ (hence $\ker(M)\nsubseteq \orth$, by Proposition \ref{propInvFac}) and when $p|h_{s+1}(M^E)$ then $\dim(\ker(M^E))<\dim(\ker(M^E \mod p))$ and $\dim(\ker(M)) = \dim(\ker(M \mod p))$, since all $h_i(M)$ are powers of 2. Applying Proposition \ref{propInvFac} to $M \mod p$ and $M^E \mod p$, we obtain
\begin{align*} 
\dim(\ker(M^E \mod p)) &{}> \dim(\ker(M^E))\\
&{}=\dim(\ker(M))-1 \\
&{} = \dim(\ker(M \mod p)) -1 \\
&{} = \dim(\ker(M^E \mod p)) - \epsilon -1,
\end{align*}
where $\epsilon$ is defined as in Corollary \ref{corolDim}. This implies $\epsilon =1$ and thus $\ker(M\mod p)\subseteq \orth \mod p$.

We now find an upper bound for $p$ given that $\ker(M)\nsubseteq \orth$ and $\ker(M \mod p) \subseteq \orth \mod p$. Without loss of generality we will take $m<n$ so that $\min(m,n)=m$. Suppose that the disjoint cycle decomposition of $\tau$ consists of exactly one odd cycle $\tau=(1' \; \tau(1') \; \ldots \; \tau^{m+n-1}(1'))$, where $1'$ is the first nonzero entry of $\tau$ and we let each successive label $\tau^i(1')$ alternate between row- and column-labels for as long as this is possible. That is, we suppose that $1', \tau^2(1'), \ldots, \tau^{2m-2}(1')$ are row-labels and $\tau(1'), \tau^3(1'), \ldots, \tau^{2m-1}(1'), \tau^{2m}(1'), \allowbreak \ldots, \tau^{m+n-1}(1')$ are column-labels. Thus the pair $\{\tau^j(1'), \, \tau^{j+1}(1')\}$ denotes a switch from a row-label to a column-label $m$ times (when $j=0, 2, \ldots, 2m-2$) and from a column-label to a row-label $m$ times (when $j=1, 3, \ldots, 2m-3, m+n-1$). In this case, the kernel of $M$ has 1 basis element $\mathbf{w}=\phi(\mathbf{v})$ associated to the odd cycle $\tau$ and vector $\mathbf{v}=\mathbf{v^{\tau}}\in \ker(P_{\omega} + P_{\sigma})$. Applying Proposition \ref{propSomething} to $\mathbf{w}$ we obtain the alternating sum:
\begin{align}\label{eqnSum2}
\sum_{i=1}^N w_k = v_{1'} - v_{\tau(1')} + v_{\tau^2(1')} - \ldots + v_{\tau^{2m-2}(1')} - v_{\tau^{m+n-1}(1')} = 2m v_{1'} = \pm 2m.
\end{align}
Thus, since $\ker(M \mod p) \subseteq \orth \mod p$, the sum above must be congruent to $0 \mod p$. That is, $p$ must be less than $m$, since $p$ is odd. A similar argument can be used if $n\leq m$ to conclude that $p<n$ hence the general result concludes that $p<\min(m,n)$.

It is obvious that if the disjoint cycle decomposition of $\tau$ consists of shorter length odd cycles $\eta_1, \ldots, \eta_r$ then the absolute value of the sum (\ref{eqnSum2}) corresponding to basis element $\mathbf{w}$ associated to any $\eta_i$ will be less than $2\min(m,n)$. Hence $2\min(m,n)$ is the upper bound for any toric permutation, $\tau$, and hence any diagram, $D$.
\end{proof}

We note that the invariant factors of a matrix $M(D)^E$ do not need to be powers of $2$ as the following example shows. 

\begin{example}
\label{example off invariant factors}
An example of a diagram with a matrix whose extended matrix has an invariant factor which is not a power of 2 is given in Figure \ref{FigEG}. The $3 \times 3$ diagram produces a matrix $M(D)$, which has nonzero invariant factors $(1,1)$, whereas its extended matrix has nonzero invariant factors $(1,1,3)$. This still obeys the theorem above as the invariant factor $h_3(M^E)$ has prime factor $3 \leq \min(3,3)$.

\definecolor{light-gray}{gray}{0.6}
\begin{figure}[h]
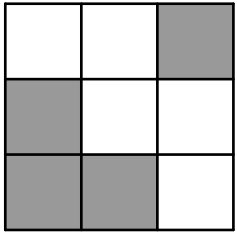

\begin{center}
\begin{tabular}{cc}
\begin{pgfpicture}{0cm}{0cm}{3.75cm}{3.75cm}%

\pgfputat{\pgfxy(-0.5,2)}{$D: $}

\pgfsetroundjoin \pgfsetroundcap%
\pgfsetfillcolor{light-gray}
\pgfmoveto{\pgfxy(0.5,0.5)}\pgflineto{\pgfxy(0.5,2.5)}\pgflineto{\pgfxy(1.5,2.5)}\pgflineto{\pgfxy(1.5,0.5)}\pgflineto{\pgfxy(0.5,0.5)}\pgffill

\pgfmoveto{\pgfxy(1.5,0.5)}\pgflineto{\pgfxy(1.5,1.5)}\pgflineto{\pgfxy(2.5,1.5)}\pgflineto{\pgfxy(2.5,0.5)}\pgflineto{\pgfxy(1.5,0.5)}\pgffill

\pgfmoveto{\pgfxy(2.5,2.5)}\pgflineto{\pgfxy(2.5,3.5)}\pgflineto{\pgfxy(3.5,3.5)}\pgflineto{\pgfxy(3.5,2.5)}\pgflineto{\pgfxy(2.5,2.5)}\pgffill

\pgfsetlinewidth{1pt}
\pgfxyline(0.5,0.5)(3.5,0.5)
\pgfxyline(0.5,0.5)(0.5,3.5)
\pgfxyline(0.5,3.5)(3.5,3.5)
\pgfxyline(3.5,0.5)(3.5,3.5)
\pgfxyline(2.5,0.5)(2.5,3.5)
\pgfxyline(1.5,0.5)(1.5,3.5)
\pgfxyline(0.5,1.5)(3.5,1.5)
\pgfxyline(0.5,2.5)(3.5,2.5)
\end{pgfpicture}
& 
\begin{pgfpicture}{0cm}{0cm}{8cm}{3.75cm}%
\color{black}
\pgfputat{\pgfxy(0,2)}{$\mapsto $ }

\pgfputat{\pgfxy(0.75,2)}{ $\mathrm{Sk}(M(D)^E)= \left( \begin{matrix} 0& 1 & 0 & 0 & 0 &0 \\ -1 & 0 & 0 & 0 & 0 &0 \\ 0 & 0 & 0 & 1 & 0 & 0 \\  0 & 0 & -1 & 0 & 0 & 0  \\ 0 & 0 & 0 & 0 & 0 & 3 \\ 0 & 0 & 0 & 0 & -3 & 0 \end{matrix} \right)$}

\end{pgfpicture}
\end{tabular}
\caption{A diagram $D$ with the skew-normal form of its extended adjacency matrix $M(D)^E$.}\label{FigEG}
\end{center}
\end{figure}
\end{example}

\subsection{PI degree of Quantum Schubert Varieties associated to quantum determinantal ideals }
Recall the matrix $\M_{n,t}$ defining the quantum affine space associated to $R_t(M_n)$, for $1\leq t < n$. 
Let $2s_t:=2nt-t^2-t$ be its rank so that it has $s_t$ (pairs of nonzero) invariant factors, which we will denote by $\{h_i^{(t)}\}_i$. 

The aim of this section is to compute the PI degree of the quantum Schubert variety $\oh_q(G_{n,2n}(\K))_{\gamma}$ when $\gamma=\{1, \ldots, t, n+1, \ldots, 2n-t\}$.

\begin{proposition} \
\begin{enumerate}
\item\label{something1} $h_{s_t+1}(\M^E_{n,t}) = 0$ if and only if $(n=2m \text{ and } t|m)$. 

\item\label{something2} All invariant factors of $\M_{n,t}^E$ are powers of $2$.
\item\label{something5} Let $\gamma=\{1, \ldots, t, n+1,\ldots, 2n-t\} \in \Pi_{n,2n}$.  Let $q$ be a primitive $\ell^{\mathrm{th}}$ root of unity with $\ell$ odd. Then,
\begin{equation*}
\PI(\oh_q(G_{n,2n}(\K))_{\gamma}) =\begin{cases} \ell^{\frac{2nt-t^2-t}{2}} & \text{if } n=2m \text{ and } t|m,\\
\ell^{\frac{2nt-t^2-t}{2} +1} & \text{otherwise.}
\end{cases}
\end{equation*}
\end{enumerate}
\end{proposition}
\begin{proof}\
\begin{enumerate}
\item For each of the $t$ odd cycles $\eta_i=(i \; \tau(i) \; \ldots \; \tau^l(i))=(R_1, C_1,  R_2, C_2, \ldots, R_k, C_k )$ in the disjoint cycle decomposition of $\tau$, as calculated in Proposition \ref{lemtoricperm}, we denote its corresponding basis element in $\ker(\M_{n,t})$ by $\mathbf{w}^{\eta_i}=(w^{\eta_i}_1, \ldots, w^{\eta_i}_N)^T$. We apply Lemma \ref{lemSumw} to each $\eta_i$, using the fact that $\tau^{|R_1|+|C_1|+ \cdots + |R_k| + |C_k|}(i)=i$ and noting that $1'=i$, to obtain:
\[\sum_{j=1}^N w^{\eta_i}_j = v^{\eta_i}_i - v^{\eta_i}_{\tau^{|R_1|}(i)} + v^{\eta_i}_{\tau^{|R_1|+|C_1|}(i)} - \cdots -v^{\eta_i}_{\tau^{|R_1|+|C_1|+|R_2|+\cdots + |R_k|}(i)}. \]
We saw in Proposition \ref{lemtoricperm} that the odd cycles $\{\eta_i\}_{i=1}^t$ are divided into 3 sets depending on whether $t<n-t, \, n-t<t,$ or $n-t=t$, with the cycles in each set sharing similar permutation rules. Recall we can write $n=ut+r$ for some $u\in \N$ and $r\in \llbracket 0,t-1 \rrbracket$. Below we look at the cycles in each of these 3 main sets, further dividing one of these sets into 2 subsets based on the value of $r$:

\begin{enumerate}[(a)]
\item \label{n-t=t} \underline{$n-t=t$:} For all $i\in \llbracket 1, t \rrbracket$ we have $\eta_i=(R_1, \, C_1)$, where $|R_1|=2$. Therefore, $\sum_{j=1}^N w^{\eta_i}_j = v^{\eta_i}_i - v^{\eta_i}_{\tau^2(i)} = v_i - v_i =0$.

\item  \underline{$t<n-t$} we split into further subsets:
\begin{enumerate} [(i)]
\item\label{t<n-t, r=0} \underline{$r=0$:} For all $i\in \llbracket 1,t \rrbracket$ we have $\eta_i = ( R_1, C_1)$ where $|R_1|=u$. Therefore,
\[\sum_{j=1}^N w^{\eta_i}_j = v^{\eta_i}_i - v^{\eta_i}_{\tau^u(i)} = \begin{cases} 
												0 & \text{ if $u$ is even}; \\
												2v_i & \text{ if $u$ is odd}.
											\end{cases} \]

If $u$ is even then $n=ut+r=2m't$, for some $m'\in \N_{>0}$, hence $n=2m$ and $t|m$. If $u$ is odd then $n=(2m'+1)t$ is odd when $t$ is odd; otherwise $n$ is even with $n=2m$ and $t\nmid m$. We therefore obtain:
\[ \left|\sum_{j=1}^N w^{\eta_i}_j\right| = \begin{cases} 0 & \text{ if } n=2m\text{ and } t|m;\\
												2 & \text{ otherwise.} 
									\end{cases}\]
									
\item  \label{t<n-t, r>0} \underline{$r\in \llbracket 1, t-1 \rrbracket$:} For all $i\in \llbracket 1,t \rrbracket$ we have $\eta_i=(R_1, C_1)$, where $|R_1|=u+1$ for all $i\in \llbracket 1, r \rrbracket$, and $|R_1|=u$ for all $i\in \llbracket r+1, t \rrbracket$. Therefore, if $u$ is odd we obtain
\begin{align*}
\sum_{j=1}^N w^{\eta_i}_j &{}= \begin{cases} v^{\eta_i}_i -v^{\eta_i}_{\tau^{u+1}(i)} =  0 & \text{ for all } i\in \llbracket 1, r \rrbracket; \\
			 v^{\eta_i}_i -v^{\eta_i}_{\tau^u(i)} = 2v_i & \text{ for all } i\in \llbracket r+1, t \rrbracket,
\end{cases}
\end{align*}
and if $u$ is even we obtain
\begin{align*}
\sum_{j=1}^N w^{\eta_i}_j &{}= \begin{cases} v^{\eta_i}_i -v^{\eta_i}_{\tau^{u+1}(i)} =  2v_i & \text{ for all } i\in \llbracket 1, r \rrbracket; \\
			 v^{\eta_i}_i -v^{\eta_i}_{\tau^u(i)} = 0 & \text{ for all } i\in \llbracket r+1, t \rrbracket.
\end{cases}
\end{align*}
Since $r\in \llbracket 1,t-1 \rrbracket$ then both of the sets $\llbracket 1,r \rrbracket$ and $\llbracket r+1, t \rrbracket$ are nonempty, therefore, for any $n$ there exist $i, k \in \llbracket 1, t \rrbracket$ such that $\left|\sum_{j=1}^N w^{\eta_i}_j\right|=0$ and $\left|\sum_{j=1}^N w^{\eta_k}_j\right|=2$.
\end{enumerate}
\item \label{0<n-t<t}\underline{$0<n-t<t$:} For all $i\in \llbracket 1,t \rrbracket$ we have $\eta_i=(R_1, C_1)$ where $|R_1|= 2$ if $i\in \llbracket 1,r \rrbracket$ and $|R_1|=1$ if $i\in \llbracket r+1, t \rrbracket$. Therefore,
\[ \sum_{j=1}^N w^{\eta_i}_j = \begin{cases} v^{\eta_i}_i -v^{\eta_i}_{\tau^2(i)} =  0 & \text{ for all } i\in \llbracket 1, r \rrbracket; \\
			 v^{\eta_i}_i -v^{\eta_i}_{\tau(i)} = 2v_i & \text{ for all } i\in \llbracket r+1, t \rrbracket,
\end{cases}. \]
If $r=0$ then the above sum takes only one value, $2v_i$, for all $i\in \llbracket 1, t \rrbracket$. Therefore, for any $n$ there exists $i\in \llbracket 1, t \rrbracket$ such that $\left|  \sum_{j=1}^N w^{\eta_i}_j \right| =2$.
\end{enumerate}

If $n=2m$ with $t|m$ then we are in cases \ref{n-t=t} and \ref{t<n-t, r=0} for $u$ even. Both of these have $\left|  \sum_{j=1}^N w^{\eta_i}_j \right| =0$ for all $i \in \llbracket 1, t \rrbracket$, hence we have $\ker(\M_{n,t})\subseteq \orth$ and $h_{s_t+1}(\M_{n,t}^E)=0$.

If $n$ is odd, or $n=2m$ and $t\nmid m$, then we are in cases \ref{t<n-t, r=0} for $u$ odd, \ref{t<n-t, r>0}, and \ref{0<n-t<t}. In each of these cases we have at least one cycle $\eta_i$ with $\left|  \sum_{j=1}^N w^{\eta_i}_j \right| =2$ and no odd cycles with an absolute sum greater than this. Therefore we conclude that $\ker(\M_{n,t})\nsubseteq \orth$ and $\ker(\M_{n,t}) \subseteq \orth \mod 2$, hence $h_{s_t+1}(\M_{n,t}^E)\neq 0$. This proves part \ref{something1}.

\item Using \cite[Theorem 5]{Thompson} to obtain $h_i(\M_{n,t}^E) | h_i(\M_{n,t})$, for all $i\in \llbracket 1, s_t \rrbracket$, it is clear that for all $t \in \llbracket 1, n-1 \rrbracket$ and $i\in \llbracket 1, s_t \rrbracket$, $h_i(\M_{n,t}^E)$ is a power of 2. The question is therefore regarding the value of $h_{s_t+1}(\M_{n,t}^E)$.

When $n=2m$ with $t|m$ we see by part \ref{something1} of this lemma that $h_{s_t+1}(\M_{n,t}^E)=0$. Therefore the statement is true for this case. When $n$ is not of this form, we note from part \ref{something1} of this lemma that $\ker(\M_{n,t}) \subseteq \orth \mod 2$ and $\ker(\M_{n,t}) \nsubseteq \orth \mod p$ for any $p>2$. Therefore, the only prime dividing $h_{s_{t+1}}(\M_{n,t}^E)$ is 2 by Proposition \ref{propKerMod}. This completes the proof of part \ref{something2}.

\item This result follows immediately from Lemma \ref{lemPIdegQSVtoQAS}, parts \ref{something1} and \ref{something2} of this lemma, and Theorem \ref{PIdegQDR}.

\end{enumerate}
\end{proof}

\subsection{PI degree of quantum affine spaces associated to the extended matrix of a diagram matrix and application to quantum Schubert varieties}

We are now in position to compute the PI degree of a quantum affine space associated to the extended matrix of a Young diagram matrix. This allows the computation of the PI degree of an arbitrary quantum Schubert variety.

\begin{theorem}\label{propGeneralPIdegME}
Let $D$ be a diagram on the Young tableau $Y_{\lambda}$ with $n=\lambda_1 \geq \dots \geq \lambda_m>0$. Denote by $N$ the number of white boxes, by $M=M(D)$ its associated matrix, and by $M^E$ the extended matrix. Suppose $\tau$ has $r$ odd cycles and $M$ has $s=(N-r)/2$ nonzero invariant factors. Let $q$ be a primitive $\ell^{\mathrm{th}}$ root of unity with $\ell$ odd and $\ell_0$ the smallest prime factor of $\ell$. Then the PI degree of the quantum affine space associated to $M^E$ is 
\[ \PI \left(\oh_{q^{M^E}} \left(\K^{N+1} \right) \right) = \begin{cases}  \ell^{\frac{N-r}{2}} & \text{ if  $\ker(M) \subseteq \mathbf{1}^{\bot}$}; \\
 \ell^{\frac{N-r}{2} + 1} & \text{ if $\ker(M) \nsubseteq \mathbf{1}^{\bot}$ and $\ell_0 > \min(m,n)$}; \\
\ell^{\frac{N-r}{2}} \cdot \frac{\ell}{\gcd(h_{s+1}(M^E), \ell)} & \text{ if $\ker(M) \nsubseteq \mathbf{1}^{\bot}$ and $\ell_0 \leq \min(m,n)$}.
								\end{cases}
\]
\end{theorem}
\begin{proof}
This easily follows from Lemma \ref{lemPIdeg}, Proposition \ref{propInvFac} and Proposition \ref{propSomething}. 

\end{proof}

\begin{example}
For the diagram of Example \ref{example off invariant factors}, we get the following for the PI degree of $\oh_{q^{M^E}} \left(\K^{N+1} \right) $:
\[ \PI \left(\oh_{q^{M^E}} \left(\K^{N+1} \right) \right) 
= \begin{cases}   \ell^3 & \text{ if } \mathrm{gcd} (\ell,3)=1; \\ 
\ell^{3} /3 & \text{ if $3$ divides $\ell$}.
\end{cases} \]
\end{example}

\begin{corollary}
\label{PIdegreeqSchubertvar}
 Let $q$ be a primitive $\ell^{\mathrm{th}}$ root of unity such that the smallest prime factor $\ell_0$ of $\ell$ satisfies $\ell_0>\min\{m,n,2\}$.
 
 Let $\gamma=[\gamma_1 < \dots < \gamma_m]$ be a quantum Pl\"ucker coordinate of $\oqgmn$ and let $\lambda $ be the partition associated to $\gamma$. We denote by $N$ the number of boxes in $Y_{\lambda}$ and by $r$ the number of odd cycles for the toric permutation associated to  $Y_{\lambda}$ (with all boxes being white).
 
  Then the PI degree of the quantum Schubert variety $\oqgmn_{\gamma}$ is given by 
\[ \PI \left(\oqgmn_{\gamma} \right) = \begin{cases}  \ell^{\frac{N-r}{2}} & \text{ if  $\ker(M_{\lambda}) \subseteq \mathbf{1}^{\bot}$}; \\
 \ell^{\frac{N-r}{2} + 1} & \text{ if $\ker(M_{\lambda}) \nsubseteq \mathbf{1}^{\bot}$}.
								\end{cases}
\]
\end{corollary}

It would be interesting to characterise when $\ker(M_{\lambda}) \subseteq \mathbf{1}^{\bot}$ via $\gamma$ or $\lambda$. Theorem \ref{PIdegqgrass} in the following subsection suggests that this might be an interesting question.

\subsection{PI degree of quantum Grassmannians}

We conclude this article with a computation of the PI degree of any quantum Grassmannian. The dimension of the kernel of the relevant skew-symmetric integral matrix was computed in \cite[Proposition 2.4]{LaunoisLenagan2007}. To prove the following theorem, we will just need to prove that the kernel of that matrix is contained in $\mathbf{1}^{\bot}$ exactly when this kernel is trivial. 

If $i$ is a positive integer greater than or equal to $2$, we denote by $\mu_2(i)$ the 2-adic valuation of $i$; that is, the largest integer $j$ such that $2^j$ divides $i$.

\begin{theorem}\label{PIdegqgrass}
Let $q$ be a primitive $\ell^{\mathrm{th}}$ root of unity such that the smallest prime factor $\ell_0$ of $\ell$ satisfies $\ell_0>\min\{m,n,2\}$.

  Then the PI degree of the quantum Schubert variety $\oqgmn$ is given by 
\[ \PI \left(\oqgmn \right) = \begin{cases}  \ell^{\frac{m(n-m)}{2}} & \text{ if  } \mu_2(m)\neq \mu_2(n); \\
 \ell^{\frac{m(n-m)-\gcd(m,n)}{2} + 1} & \text{ otherwise}.
								\end{cases}
\]
\end{theorem}
\begin{proof} Given a positive integer $\eta$, we let $B_{\eta}$ denote the $\eta\times \eta$ skew-symmetric matrix with $1$'s in every entry above the diagonal.  Given positive integers $\mu, \eta$ with $\mu\le \eta$, we let $U_{\eta,\mu}$ denote the $\eta\times \mu$ matrix given by
$(I_{\mu} ~|~{\bf 0}_{\mu\times (\eta-\mu)})^T$.  Finally, let ${\bf 1}_{\eta}$ denote the $\eta \times 1$ column vector whose entries are all $1$.  
Let $\lambda_1\ge \cdots \ge \lambda_s\ge 1$ with $\sum \lambda_i=N$.  Then the matrix $M_{\lambda}$ can be expressed as a block $s\times s$ matrix in which the $(i,j)$-block is $U_{\lambda_i,\lambda_j}$ if $i<j$; $-U_{\lambda_j,\lambda_i}^T$ if $i>j$; and $B_{\lambda_i}$ if $i=j$.  Let $v_i\in \mathbb{C}^{\lambda_i\times 1}$ and suppose that 
$${\bf v}=(v_1~|~v_2~|~\cdots~|~v_s)^T$$ is in the kernel of $M_{\lambda}$.  We would like to know when this implies that ${\bf 1}_{N}^T\cdot {\bf v}=0$.  

In order to compute the PI degree of $\oqgmn$, we need to apply Corollary \ref{PIdegreeqSchubertvar} in the case when $\gamma = [1 \dots m]$ since $\oqgmn_{[1 \dots m]}=\oqgmn$. In this case, the partition $\lambda$ is given by $(\lambda_1=n-m, \dots ,\lambda_m=n-m)$.

 In this case, we set $B:=B_{n-m}$. If we subtract the $i$-th row from our block matrix $A_{\lambda}$ from the $(i+1)$-row of our block matrix, then we get
$(-I-B)v_i + (B-I)v_{i+1}=0$.  Since $B$ is skew-symmetric, its eigenvalues are purely imaginary and so $B-I$ is invertible.  Then $v_{i+1} = (B-I)^{-1}(B+I)v_i$ for $i=1,\ldots ,m-1$.  In particular, if we let $C:=(B-I)^{-1}(B+I)$, then we see that $v_j = C^{j-1} v_1$ for $j=1,\ldots ,m$.  Now looking at the first row of our block matrix and multiplying it by ${\bf v}$ and using the fact that ${\bf v}$ is in the kernel, we see that
\begin{equation}
\label{eq:v}
-v_1-Cv_1-\cdots - C^{m-2} v_1 + B C^{m-1} v_1 = 0.
\end{equation} Notice that the eigenvalues of $C$ are of the form 
$(\beta+1)/(\beta-1)$ with $\beta$ a purely imaginary number and since $C$ and $B$ commute, we see that
$v_1=0$ unless $0$ is an eigenvalue of 
$-I-C-\cdots - C^{m-2} + B C^{m-1}$.  
Notice that $B$ and $C$ commute and are simultaneously diagonalizable (since $B$ is skew-symmetric).  If $\gamma=(\beta+1)/(\beta-1)$ is an eigenvalue of $C$ then the corresponding eigenvalue, $\beta$, of $B$ is $(\gamma+1)/(\gamma-1)$.  
So the eigenvalues of $X:=-I-C-\cdots - C^{n-2} + B C^{n-1}$ are of the form
$-(\gamma^{m-1}-1)/(\gamma-1) + \gamma^{m-1} (\gamma+1)/(\gamma-1)$. If $X$ has a zero eigenvalue then since $\gamma\neq 0$ we must have 
$$-(\gamma^{n-1}-1) +\gamma^{n-1}(\gamma+1)=0,$$ or equivalently 
$$1+\gamma^m=0.$$
Thus $\gamma=\exp(\pi i j/m)$ for some odd $j$.  Thus the kernel of $A_{\lambda}$ in this case is spanned by the vectors of the form ${\bf v}=(v~|~Cv~|~\cdots ~| C^{m-1}v)^T$ with $v$ an eigenvector of $B$ with eigenvalue $\gamma$ with $\gamma^m=1$.  It is easily checked that the eigenspace (if such an eigenvector exists) associated with $\gamma$ is one-dimensional and spanned by $(1,\gamma,\ldots , \gamma^{n-m-1})^T$, and that this occurs only if $\gamma^{n-m}=-1$.  In other words
$(\gamma-1)^2{\bf 1}_{mn}^T \cdot {\bf v} = (\gamma^n-1)(\gamma^m-1)=4$. In other words, the kernel is contained in the orthogonal complement of ${\bf 1}_{mn}^T$ if and only if the kernel of $M_{(n-m, \dots , n-m)}$ is zero.

Now, by \cite[Proposition 2.4]{LaunoisLenagan2007}, the kernel is zero if and only if $\mu_2(m) \neq \mu_2(n)$, and in the case where $\mu_2(m) = \mu_2(n)$, the kernel of $M_{(n-m, \dots , n-m)}$ has dimension $\gcd (m,n)$. We conclude thanks to Corollary \ref{PIdegreeqSchubertvar}.

\end{proof}

\bibliographystyle{alpha}

\begin{minipage}{\textwidth}
\noindent J.P. Bell \\
Department of Pure Mathematics\\ 
University of Waterloo\\
Waterloo, ON, N2L 3G1\\ Canada\\[0.5ex]
email: jpbell@waterloo.ca \\

\noindent S. Launois \\
School of Mathematics, Statistics and Actuarial Science,\\
University of Kent\\
Canterbury, Kent, CT2 7FS,\\ UK\\[0.5ex]
email: S.Launois@kent.ac.uk \\

\noindent A. Rogers\\
email: alex@hence.ai \\

\end{minipage}

\end{document}